\documentclass[reqno,10pt]{amsart}
\usepackage{amssymb,amsmath,amsthm}
\usepackage{a4wide}
\usepackage{mathrsfs}
\usepackage{color}
\usepackage{esint}
\usepackage{stmaryrd}
\usepackage[colorlinks,linkcolor=blue,anchorcolor=blue,citecolor=blue]{hyperref}
\usepackage{geometry}
\newtheorem{theorem}{Theorem}

\newtheorem{lemma}[theorem]{Lemma}
\newtheorem{proposition}[theorem]{Proposition}
\newtheorem{remark}[theorem]{Remark}

\def\beq{\begin{equation}}
\def\eeq{\end{equation}}
\def\beqs{\begin{equation*}}
\def\eeqs{\end{equation*}}
\def\bal#1\eal{\begin{align}#1\end{align}}
\def\bals#1\eals{\begin{align*}#1\end{align*}}
\def\bsp#1\esp{\begin{split}#1\end{split}}

\def\d{{\mathrm{d}}}

\let\e=\varepsilon

\let\pt=\partial

\let\l=\lambda
\let\L=\Lambda

\let\b=\beta

\numberwithin{equation}{section}
\numberwithin{theorem}{section}

\allowdisplaybreaks[3]

\begin{document}
\date{}
\title[   Euler--Poisson limit to the ionic Vlasov--Poisson--Boltzmann  ]
{Global compressible Euler--Poisson limit of the ionic Vlasov--Poisson--Boltzmann system for all cutoff Potentials}

\author{Qin Ye$^\dagger$, Fujun Zhou$^*$ and Weijun Wu $^\ddagger$}

\address[Qin Ye$^\dagger$]{School of Mathematics, South China University of Technology, Guangzhou 510640, China}
\email{yeqin811@163.com}

\address[Fujun Zhou$^*$, Corresponding author]{School of Mathematics, South China University of Technology, Guangzhou 510640, China}
\email{fujunht@scut.edu.cn}

\address[Weijun Wu$^\ddagger$]{School of Cyber Security,   Guangdong Police College,   Guangzhou 510230,   China}
\email{scutweijunwu@qq.com}

	\begin{abstract}
		The ionic Vlasov--Poisson--Boltzmann system is a fundamental model in dilute collisional plasmas. In this work, we study the compressible ionic Euler--Poisson limit of the ionic Vlasov--Poisson--Boltzmann system for the full range of cutoff potentials $-3 < \gamma \le 1$. By employing a truncated Hilbert expansion together with a novel weighted $H^1_{x,v}$--$W^{1,\infty}_{x,v}$ framework, we prove that the solution of the ionic Vlasov--Poisson--Boltzmann converges globally in time to the smooth global solution of the compressible ionic Euler--Poisson system.
	\end{abstract}

 \makeatletter
\@namedef{subjclassname@2020}{\textup{2020} Mathematics Subject Classification}\makeatother
\maketitle

\tableofcontents

	\section{Introduction }
	
	\subsection{Problem Setup and Main Results}
	The dynamics of dilute ions result from the combined effects of free transport, binary collisions, and self-consistent electric fields. At the kinetic level, these mechanisms are mathematically captured by the \emph{ionic Vlasov--Poisson--Boltzmann (IVPB) system}:
	\begin{equation} \label{vpb}
		\begin{split}
			\partial_t F^{\varepsilon} + v \cdot \nabla_x F^{\varepsilon}
			- \nabla_x \phi^{\varepsilon} \cdot \nabla_v F^{\varepsilon}
			&= \frac{1}{\varepsilon} Q(F^{\varepsilon},   F^{\varepsilon}),   \\
			\Delta \phi^{\varepsilon} &= e^{\phi^{\varepsilon}} - \int_{\mathbb{R}^3} F^{\varepsilon} \,   \d v.
		\end{split}
	\end{equation}
	Here,   \( F^\varepsilon(t,   x,   v) \geq 0 \) denotes the number density function of ions at time \( t \geq 0 \),   position \( x = (x_1,   x_2,   x_3) \in \mathbb{R}^3 \),   and velocity \( v = (v_1,   v_2,   v_3) \in \mathbb{R}^3 \). The IVPB system arises as the massless-electron limit of the two-species Vlasov--Poisson--Boltzmann equations,   in which electrons equilibrate instantaneously due to their much smaller mass,   giving the Poisson-Boltzmann relation for the electric potential \cite{BardosGolseNguyenSentis2018}. The exponential term $e^{\phi^\varepsilon}$ constitutes a major analytical challenge compared to the classical electronic Vlasov--Poisson system.
	The Boltzmann collision operator $Q(G_1,   G_2)$ for binary elastic collisions is
	\begin{equation*}
		Q(G_1,   G_2) = \int_{\mathbb{R}^3 \times \mathbb{S}^2} |(u - v)   |^{\gamma} b_0(\alpha)\left\{ G_1(v') G_2(u') - G_1(v) G_2(u) \right\} \d u \d\omega,  \quad -3<\gamma\leq 1,
	\end{equation*}
	where  the angular kernel satisfies the Grad cutoff $0\le b_0(\alpha)\le C|\cos\alpha|$ with $\cos\alpha = \omega\cdot\frac{v-u}{|v-u|}$,  and the post-collisional velocities
	$
	v' = v - [(v-u)\cdot\omega]\omega,     u' = u + [(v-u)\cdot\omega]\omega
	$.
	\par
	At the macroscopic level,   dilute ion gases behave as an isentropic charged fluid,   which can be modeled by the compressible ionic Euler--Poisson system:
	\begin{equation} \label{ep}
		\begin{aligned}
			\partial_t \rho + \nabla \cdot (\rho u) &= 0,   \\
			\rho \left( \partial_t u + u \cdot \nabla u \right) &= -\nabla (K \rho^{5/3}) - \rho \nabla \phi,   \\
			\Delta \phi &= e^\phi - \rho,
		\end{aligned}
	\end{equation}
	where \( \rho(t,   x) \) and \( u(t,   x) \) denote the ion density and velocity field,   respectively,   and \( K > 0 \) is a constant related to the equation of state (e.g.,   in the Thomas--Fermi approximation).
	This system can be formally derived from a two-fluid Euler--Poisson model for ions and electrons by taking the massless-electron limit. In this regime,   electrons relax instantaneously to a Boltzmann equilibrium,   resulting in the semilinear Poisson--Boltzmann equation \(\eqref{ep}_3\) for the electric potential \cite{GrenierGuoPausaderSuzuki2020}.	
	\par
	To the best of our knowledge,   the fluid dynamic limit of the IVPB system remains largely open. In this work,   we rigorously derive the compressible ionic Euler--Poisson system \eqref{ep} as the fluid dynamic limit of the IVPB system \eqref{vpb} when the Knudsen number \( \varepsilon \) tends to zero.
	\medskip
	\par
	In order to justify the Euler--Poisson limit, we employ the standard Hilbert expansion for the system \eqref{vpb}, similarly to  \cite{Caflisch1980,GuoCMP2010},
	\begin{equation}\label{exp}
		F^{\varepsilon}=\sum_{n=0}^{2k-1}\varepsilon^nF_n+\varepsilon^kF_R^{\varepsilon};\qquad\phi^{\varepsilon}=\sum_{n=0}^{2k-1}\varepsilon^n\phi_n+\varepsilon^k\phi_R^{\varepsilon}, \quad  k\geq 4.
	\end{equation}
	Applying the Taylor expansion to  $\exp\{\sum_{n=1}^{2k-1}\e^n\phi_n\} $, we have
	\begin{align}\label{ephi}
		e^{\phi^\varepsilon}=&e^{\phi_0} \left(1 + \sum_{n=1}^{2k-1} \varepsilon^n B_n+ \sum_{n=2k}^{\infty} \varepsilon^n H_n \right) \left(1 +e^{ \varepsilon^k \phi_R^\varepsilon}-1   \right)\\\nonumber
		=&e^{\phi_0} \left(1 + \sum_{n=1}^{2k-1} \varepsilon^n B_n  \right)+e^{\phi_0}\sum_{n=2k}^{\infty} \varepsilon^n H_n+ e^{\phi_0} \left(1 + \sum_{n=1}^{2k-1} \varepsilon^n B_n+ \sum_{n=2k}^{\infty} \varepsilon^n H_n \right)\left(e^{ \varepsilon^k \phi_R^\varepsilon}-1   \right),
	\end{align}
	where
	\begin{align}
		\label{B1}B_1:&=\phi_1,  \\\nonumber
		B_n :&= \sum_{m=1}^{n} \frac{1}{m!} \sum_{\substack{n_1 + \cdots + n_m = n \\ 1 \le n_i \le n}} \prod_{i=1}^{m} \phi_{n_i}=\phi_n+ \sum_{m=2}^n \frac{1}{m!} \sum_{\substack{n_1 + \cdots + n_m = n \\ 1 \leq n_i < n}} \prod_{i=1}^m \phi_{n_i}=\phi_n+\widetilde{B}_n, \;  2\leq n\leq2k-1,  	 \\\nonumber
		H_n :&= \sum_{m=2}^{n} \frac{1}{m!} \sum_{\substack{n_1 + \cdots + n_m = n \\ 1 \le n_i \leq2k-1}} \prod_{i=1}^{m} \phi_{n_i}, \; n\geq2k.
	\end{align}
	Plugging the formal expansion \eqref{exp} and \eqref{ephi} into the rescaled equations \eqref{vpb} and comparing the coefficients on both sides
	with different powers of $\e$,   we have
	\begin{align}
		\label{coeff}
		\frac{1}{\varepsilon }: & \quad Q(F_{0},  F_{0})=0,   \\\nonumber
		\varepsilon ^{0}: & \quad \partial _{t}F_{0}+v\cdot \nabla _{x}F_{0} -\nabla_{x}\phi _{0}\cdot \nabla _{v}F_{0}=Q(F_{1},  F_{0})+Q(F_{0},  F_{1}),   \\\nonumber
		& \quad \Delta \phi _{0}=e^{\phi _{0}}-\int_{\mathbb{R}^{3}}F_{0}\d v,   \\\nonumber
		& \quad \cdots \\\nonumber
		\varepsilon ^{n}: & \quad \partial _{t}F_{n}+v\cdot \nabla _{x}F_{n} -\nabla_{x}\phi _{0}\cdot \nabla _{v}F_{n}-\nabla _{x}\phi _{n}\cdot \nabla_{v}F_{0}=\sum_{\substack{ i+j=n+1 \\ i,   j\geq 0}}Q(F_{i},  F_{j}) +\sum_{\substack{ i+j=n \\ i,   j\geq 1}}\nabla _{x}\phi _{i}\cdot \nabla _{v}F_{j},   \\\nonumber
		& \quad \Delta \phi _{n}=e^{\phi_{0}} \phi _{n}+ e^{\phi_{0}}  \widetilde{B}_{n}-\int_{\mathbb{R}^{3}}F_{n}\d v, \;  1\leq n \leq 2k-2,  \\\nonumber
		\varepsilon ^{2k-1}: & \quad \partial _{t}F_{2k-1}+v\cdot \nabla _{x}F_{2k-1}  =\sum_{\substack{ i+j=2k\\ i,   j\geq 1}}Q(F_{i},  F_{j}) +\sum_{\substack{ i+j=2k-1\\ i,   j\geq 0}}\nabla _{x}\phi _{i}\cdot \nabla _{v}F_{j},   \\\nonumber
		& \quad \Delta \phi _{2k-1}=e^{\phi_{0}} \phi _{2k-1}+ e^{\phi_{0}}  \widetilde{B}_{2k-1}-\int_{\mathbb{R}^{3}}F_{2k-1}\d v.
	\end{align}
	\par
	The remainder equations for $F_{R}^{\varepsilon }$ and $\phi_{R}^{\varepsilon }$ are given as follows:
	\begin{equation}
		\begin{split}\label{F_R}
			\partial _{t}F_{R}^{\varepsilon }	+\; &v\cdot \nabla _{x}F_{R}^{\varepsilon}
			- \nabla _{x}\phi _{0}\cdot \nabla _{v}F_{R}^{\varepsilon } - \nabla _{x}\phi_{R}^{\varepsilon }\cdot \nabla _{v}F_{0} - \frac{1}{\varepsilon }
			\{Q(F_{0},  F_{R}^{\varepsilon })+Q(F_{R}^{\varepsilon },  F_{0})\} \\
			=& \varepsilon ^{k-1}Q(F_{R}^{\varepsilon },   F_{R}^{\varepsilon })
			+ \sum_{i=1}^{2k-1}\varepsilon ^{i-1}\{Q(F_{i},   F_{R}^{\varepsilon })+Q(F_{R}^{\varepsilon },   F_{i})\}
			+ \varepsilon ^{k}\nabla _{x}\phi_{R}^{\varepsilon }\cdot \nabla _{v}F_{R}^{\varepsilon } \\
			&  + \sum_{i=1}^{2k-1}\varepsilon ^{i}\{\nabla _{x}\phi _{i}\cdot \nabla_{v}F_{R}^{\varepsilon }+\nabla _{x}\phi _{R}^{\varepsilon }\cdot \nabla_{v}F_{i}\}
			+ \varepsilon ^{k-1}A,   \\
			\Delta \phi_R^\varepsilon =&
			e^{\phi_0}\frac{(e^{\varepsilon^k \phi_R^{\varepsilon}}-1
				)}{{\varepsilon}^k}+G
			- \int_{\mathbb{R}^3} F_R^\varepsilon\,   \d v
			,
		\end{split}
	\end{equation}	
	where
	\begin{align}
		\label{A} A=&\sum_{\substack{ i+j\geq 2k+1 \\ 2\leq i,   j\leq 2k-1}}\varepsilon
		^{i+j-2k}Q(F_{i},   F_{j})-\sum_{\substack{ i+j\geq 2k \\ 1\leq i,   j\leq 2k-1}}%
		\varepsilon ^{i+j-2k+1}\nabla _{x}\phi _{i}\cdot \nabla _{v}F_{j},  \\
		\label{G}
		G=&
		e^{\phi_0}\sum_{n=2k}^{\infty} \varepsilon^{n-k} H_n
		+ e^{\phi_0} \left( \sum_{n=1}^{2k-1} \varepsilon^n B_n+ \sum_{n=2k}^{\infty} \varepsilon^n H_n \right)\frac{(e^{\varepsilon^k \phi_R^{\varepsilon}}-1
			)}{{\varepsilon}^k}.	
	\end{align}
	\medskip
	\par

	Then, let us briefly determine the leading order  $F_0$ and the coefficients $F_n$ for $1\leq n \leq 2k-1$.
	According to $Q(F_{0},  F_{0})=0$ from $\eqref{coeff}_1$,  we deduce $F_{0}$ should be a local Maxwellian $\mu  $:
	\begin{equation}\label{LM}
		F_{0}(t,   x,   v)=\mu  =\frac{\rho _{0}(t,   x)}{(2\pi {\theta } _{0}(t,   x))^{3/2}}%
		e^{-\frac{|v-u_{0}(t,   x)|^{2}}{2{\theta } _{0}(t,   x)}},  \quad {\theta } _{0}(t,   x)=K\rho
		_{0}^{2/3}(t,   x),
	\end{equation}
	where $\rho _{0}(t,   x),   u_{0}(t,   x)$ and ${\theta } _{0}(t,   x)$ represent the macroscopic density,   velocity,   and temperature fields respectively. Note that
	\begin{equation*}
		\int_{\mathbb{R}^{3}}F_{0}\,  \d v=\rho _{0},  \quad \int_{\mathbb{R}^{3}}vF_{0}\,  \d v=\rho _{0}u_{0},  \quad \int_{\mathbb{R}^{3}}|v|^{2}F_{0}\,  \d v=\rho _{0}|u_{0}|^{2}+3\rho _{0}{\theta } _{0}.
	\end{equation*}
	Project the equation for $F_0$ from the $\varepsilon^{0}$ step in $\eqref{coeff}_1$ onto $1,   v,   \frac{|v|^{2}}{2}$,     setting $p_{0}= \rho_{0}{\theta }_{0}$ and  ${\theta }_{0}\equiv K\rho_{0}^{2/3}$,   we deduce that $(\rho_0,u_0,\phi_0)$ satisfies the compressible ionic Euler--Poisson system \eqref{ep}.
	\par
	
	We define the linearized Boltzmann operator at $\mu$ as
	\begin{align}\label{deln}
		L g =-\frac{1}{\sqrt{\mu}} \left\{ Q(\sqrt{\mu} g, \mu) + Q(\mu, \sqrt{\mu} g) \right\}, \quad
		\Gamma(g_1, g_2) = \frac{1}{\sqrt{\mu}} Q(\sqrt{\mu} g_1, \sqrt{\mu} g_2).
	\end{align}
	Recall that $L \geq 0$ and the null space of $L$ is generated by the orthonormal basis:
	\begin{align*}
		\chi_0(v) \equiv \frac{1}{\sqrt{\rho_0}} \sqrt{\mu}, \quad
		\chi_i(v) \equiv \frac{v^i - u_0^i}{\sqrt{\rho_0 {\theta }_0}} \sqrt{\mu}(i=1, 2, 3), \quad
		\chi_4(v) \equiv \frac{1}{\sqrt{6\rho_0}} \Big \{ \frac{|v - u_0|^2}{{\theta }_0} - 3  \Big \} \sqrt{\mu},
	\end{align*}
	and  $\langle \chi_i, \chi_j \rangle = \delta_{ij}$ for $0 \leq i, j \leq 4$ obviously. For each $i\geq 1$,
	split $\frac{F_{i}}{\sqrt{\mu}}=\mathbb{P}\left(\frac{F_{i}}{\sqrt{\mu}}\right)+\mathbb{P}^{\bot} \left(\frac{F_{i}}{\sqrt{\mu}}  \right)$,
	\begin{equation} \label{F_1}
		\mathbb{P} \Big (\frac{F_{i}}{\sqrt{\mu}} \Big )  \equiv  \Big \{ \frac{\rho _{i}}{\sqrt{\rho _{0}}}\chi _{0}+\sum_{j=1}^{3}%
		\sqrt{\frac{\rho _{0}}{{\theta } _{0}}}u_{i}^{j} \chi _{j}+\sqrt{\frac{%
				\rho _{0}}{6}}\frac{{\theta } _{i}}{{\theta } _{0}}\chi _{4} \Big \}.
	\end{equation}
	Using the same approach as  \cite{GuoCMP2010}, for $0\leq n \leq 2k-2 $, once $F_n$ are found, then
	\begin{equation}\label{i-p}
		\mathbb{P}^{\bot}\left(\frac{F_{n+1}}{\sqrt{\mu}}\right)
		=- L^{-1}\bigg[ {\sqrt{\mu}}^{-1}\bigg(
		\left\{\partial
		_{t}+v\cdot \nabla _{x}\right\}F_{n} -\sum_{\substack{ i+j=n \\ i, j\geq 0}}\nabla
		_{x}\phi _{i}\cdot \nabla _{v}F_{j}-\sum_{\substack{ i+j=n+1 \\ i, j\geq 1}}%
		Q(F_{i},F_{j})\bigg) \bigg].
	\end{equation}
	For the macroscopic part, $\rho _{n+1},u_{n+1},{\theta } _{n+1}$ satisfy
	\begin{align}
		\label{eqF_k}
		& \partial _{t}\rho _{n+1}+\nabla \cdot (\rho _{0}u_{n+1}+\rho
		_{n+1}u_{0})=0, \\\nonumber
		& \rho _{0} \Big \{\partial _{t}u_{n+1}+(u_{n+1}\cdot \nabla )u_{0}+(u_{0}\cdot
		\nabla )u_{n+1}+\nabla \phi _{n+1} \Big \}-\frac{\rho _{n+1}}{\rho _{0}}\nabla
		(\rho _{0}{\theta } _{0})+\nabla  \Big (\frac{\rho _{0}{\theta } _{n+1}+3{\theta } _{0}\rho
			_{n+1}}{3} \Big )=f_{n}, \\\nonumber
		& \rho _{0} \Big \{\partial _{t}{\theta } _{n+1}+\frac{2}{3}({\theta } _{n+1}\nabla
		\cdot u_{0}+3{\theta } _{0}\nabla \cdot u_{n+1})+u_{0}\cdot \nabla {\theta }
		_{n+1}+3u_{n+1}\cdot \nabla {\theta } _{0} \Big \}=g_{n}, \\\nonumber
		& \Delta \phi _{n+1}=e^{\phi_{0}} \phi _{n+1}+ e^{\phi_{0}} \widetilde{B}_{n+1}-\rho _{n+1},
	\end{align}
	where
	\begin{equation*}
		\begin{split}
			f_{n}& =-\partial _{j}\int  \Big \{(v^{i}-u_{0}^{i})(v^{j}-u_{0}^{j})-\delta _{ij}
			\frac{|v-u_{0}|^{2}}{3} \Big \}F_{n+1} \d v-\sum_{\substack{ i+j=n+1 \\ i, j\geq 1}}
			\rho _{j}\nabla _{x}\phi _{i}, \\[4pt]
			g_{n}& =-\partial _{i} \Big \{\int (v^{i}-u_{0}^{i})(|v-u_{0}|^{2}-5{\theta }
			_{0})F_{n+1}\,\d v+2u_{0}^{j}\int  \Big [(v^{i}-u_{0}^{i})(v^{j}-u_{0}^{j})-\delta
			_{ij}\frac{|v-u_{0}|^{2}}{3} \Big ]F_{n+1}\,\d v \Big \} \\[4pt]
			& \quad -2u_{0}\cdot f_{n}-\sum_{\substack{ i+j=n+1 \\ i, j\geq 1}}(\rho
			_{0}u_{j}+\rho _{j}u_{0})\cdot \nabla _{x}\phi _{i}.
		\end{split}
	\end{equation*}
	Here we use the subscript $n$ for forcing terms $f$ and $g$ in order to emphasize that
	the right-hand sides depend only on $F_{i}$'s and $\nabla _{x}\phi _{i}$'s
	for $0\leq i\leq n$.
	\medskip
	\par
	For the estimate of the remainder $F_R^\varepsilon$,
	when the velocity $|v|$ is large, the term $\frac{(\partial_t + v \cdot \nabla_x)\sqrt{\mu}}{\mu}$ becomes difficult to control. To overcome this, we introduce a global Maxwellian
	
	\[
	\mathcal{M} = \frac{1}{(2\pi {\theta }_M)^{3/2}} \exp\left( -\frac{|v|^2}{2{\theta }_M} \right),
	\]
	where $  {\theta }_M = K n_0^{2/3}$ and $n_0$ is  a given background   density.  	Define the Boltzmann operator at $\mathcal{M}$ as
	\begin{align}\label{delm}
		L_{M} g =-\frac{1}{\sqrt{\mathcal{M}}} \left\{ Q(\sqrt{\mathcal{M}} g,   \mathcal{M}) + Q(\mathcal{M},   \sqrt{\mathcal{M}} g) \right\},    \quad
		\Gamma_{M}(g_1,   g_2) =\frac{1}{\sqrt{\mathcal{M}}} Q(\sqrt{\mathcal{M}} g_1,   \sqrt{\mathcal{M}} g_2).
	\end{align}As \cite{GuoCMP2010},  we assume that
	$
	{\theta }_M\leq  \inf_{t,   x}{\theta_0(t,   x) }\leq\max_{t,   x}{\theta_0(t,   x) }\leq2	{\theta }_M,  	
	$
	which implies there exist two constant $c>0,   C>0$ and $\varpi\in(1/2,   1)$ such that
	\begin{align}\label{nunum}
		c\mathcal{M}\leq \mu \leq C \mathcal{M}^{\varpi}.
	\end{align}
	Define the fluctuations $f^\varepsilon$ and $h^\varepsilon$ and the weight function $w$ by
	\begin{equation} \label{f}
		\frac{F_R^\varepsilon}{\sqrt{\mu}}=f^\varepsilon =\frac{\sqrt{\mathcal{M}}}{w  \sqrt{\mu}} h^\varepsilon,  \qquad  w =
		e^{\widetilde{\vartheta}|v|^2} \text{ with } \widetilde{\vartheta}= \widetilde{\vartheta}(t) = \vartheta \bigl[ 1 + (1+t)^{-\sigma} \bigr],
	\end{equation}
	where $0<\vartheta\ll1$,   and   $\sigma$ is chosen such that $\sigma>\frac13$ for $0\le\gamma\le1$,   while $\frac13<\sigma<\frac23$ for $-3<\gamma<0$.
	\smallskip

	We  state the following two main results of this work.
	\begin{theorem}[\bf Soft potentials $-3< \gamma <0$]\label{t1} \;
		Let $F_0 =\mu $ be as in \eqref{LM}. Fix an integer $s_1\ge 5$ and a constant $n_0>0$.  Assume the initial perturbation $(\rho_0^{\mathrm{in}}, u_0^{\mathrm{in}})$ satisfies $\nabla_x\times u_0^{\mathrm{in}}=0$ and, for the constant $0 < \delta_0 \ll 1$,
		\begin{align*}
			\|(\rho_0^{\mathrm{in}}-n_0,   u_0^{\mathrm{in}})\|_{H^{2s_1+1}} + \|(\rho_0^{\mathrm{in}}-n_0,   u_0^{\mathrm{in}})\|_{W^{s_1+12/5,   10/9}} \le 2\delta_0,
		\end{align*}
		so that the existence of a unique global solution $(\rho_0(t,x), u_0(t,x), \phi_0(t,x))$
		to the Euler--Poisson system \eqref{ep} is guaranteed by \cite{GuoCMP2011}.
		Then, the remainder $F_R^\varepsilon = \sqrt{\mu} f^\varepsilon$ defined in \eqref{F_R} satisfies the following bound:
		\begin{align}\label{t1s}
			&\sup_{0\le t \le \varepsilon^{-1/2}} \bigg\{ \|f^\varepsilon(t)\| + \|\sqrt{\varepsilon} \nabla_{x,v} f^\varepsilon(t)\| + \|\sqrt{\varepsilon} \langle v\rangle^{2-2\gamma} f^\varepsilon(t)\| + \|\varepsilon^{5/4} \langle v\rangle^{-\gamma} \nabla_x f^\varepsilon(t)\| \nonumber\\
			&\qquad\qquad\;\;\;\; + \|\phi_R^\varepsilon(t)\|_{H^1_x} + \varepsilon^2 \left\| w \frac{\sqrt{\mu} f^\varepsilon(t)}{\sqrt{\mathcal{M}}} \right\|_{W^{1,  \infty}_{x,v}} \bigg\}\\ \nonumber
			\lesssim& 1 + \|f^\varepsilon(0)\| + \|\sqrt{\varepsilon} \nabla_{x,v} f^\varepsilon(0)\| + \|\sqrt{\varepsilon} \langle v\rangle^{2-2\gamma} f^\varepsilon(0)\| + \|\varepsilon^{5/4} \langle v\rangle^{-\gamma} \nabla_x f^\varepsilon(0)\| \\\nonumber
            &\qquad\qquad\;\;\;\; + \|\phi_R^\varepsilon(0)\|_{H^1_x} + \varepsilon^2 \left\| w \frac{\sqrt{\mu} f^\varepsilon(0)}{\sqrt{\mathcal{M}}} \right\|_{W^{1,  \infty}_{x,v}}.
		\end{align}
	\end{theorem}

	\begin{theorem}[\bf Hard potentials $0\le \gamma \le 1$]\label{t2}
		Under the same assumptions as in Theorem~\ref{t1},    the remainder $F_R^\varepsilon = \sqrt{\mu} f^\varepsilon$ defined in \eqref{F_R} satisfies   the following bound:
		\begin{align}\label{t2h}
			&\sup_{0\le t \le \varepsilon^{-1/2}} \Big\{ \|f^\varepsilon(t)\| + \|\sqrt{\varepsilon} \nabla_{x,v} f^\varepsilon(t)\| + \|\sqrt{\varepsilon} \langle v\rangle^2 f^\varepsilon(t)\| + \|\phi_R^\varepsilon(t)\|_{H^1_x}  + \varepsilon^2 \left\| w \frac{\sqrt{\mu} f^\varepsilon(t)}{\sqrt{\mathcal{M}}} \right\|_{W^{1,  \infty}_{x,v}} \Big\} \nonumber\\
			\lesssim& 1 + \|f^\varepsilon(0)\| + \|\sqrt{\varepsilon} \nabla_{x,v} f^\varepsilon(0)\| + \|\sqrt{\varepsilon} \langle v\rangle^2 f^\varepsilon(0)\| + \|\phi_R^\varepsilon(0)\|_{H^1_x} + \varepsilon^2 \left\| w \frac{\sqrt{\mu} f^\varepsilon(0)}{\sqrt{\mathcal{M}}} \right\|_{W^{1,  \infty}_{x,v}}.
		\end{align}
	\end{theorem}
	\begin{remark}
		By employing a Hilbert expansion of order $k\ge 4$ together with a weighted $H^1_{x, v}$--$W^{1,\infty}_{x, v}$ framework,
		we establish the ionic Euler--Poisson limit of \eqref{vpb} on the time interval $0 \le t \le \varepsilon^{-1/2}$
		for the full range of collision potentials $-3<\gamma\le 1$. Indeed,
		Theorems \ref{t1} and \ref{t2} lead to
		\begin{equation}\label{conveges-to-EP}
			\sup_{0\le t \le \varepsilon^{-1/2}}\big\{\|F^\e-F_0\|+\big\|\phi^\varepsilon-\phi_0\big\|_{H^1}\big\}=O(\e),
		\end{equation}
		implying that the solution $(F^\e,\phi^\varepsilon)$ of the IVPB system \eqref{vpb} converges to the solution of the compressible ionic Euler--Poisson system \eqref{ep}.
	\end{remark}
	\begin{remark}
		In the classical Euler--Poisson limit of the electron VPB system,
		Guo and Jang~\cite{GuoCMP2010} treated the hard-sphere case $\gamma=1$ using Hilbert expansions of order $k\ge 6$ in the $L^2_{x,v}$--$W^{1,\infty}_{x,v}$ framework
		on time interval $0 \le t \le \varepsilon^{-m}$ with
		$0 <m \le \frac12 \frac{2k-3}{2k-2}< \frac12$. Then Li and Wang~\cite{lisiam2023} extended the analysis to soft potentials $-3<\gamma<0$ under the condition
		$0 < m \le \frac{-\gamma}{(2-\gamma)(1-\gamma)}< \frac12 \frac{2k-3}{2k-2}$.
		\par
		Shortly before the submission of this work, a related preprint~\cite{LiWang2026_GlobalHilbertExpansionIonicVPB} appeared,
		in which Li and Wang adopted the classical $L^2_{x,v}$--$W^{1,\infty}_{x,v}$ framework introduced by \cite{GuoCMP2010},
		to establish the ionic Euler--Poisson limit of \eqref{vpb} for the hard-sphere case $\gamma=1$. While the hard-sphere case $\gamma=1$ is also covered in our present work, the analytical approach and the resulting estimates
		differ from those  in Theorem~\ref{t2} for $\gamma=1$.
	\end{remark}
	
	\begin{remark}
		We adopt a weighted $H^1_{x,   v}$--$W^{1,  \infty}_{x,   v}$ framework instead of the classical $L^2_{x,   v}$--$W^{1,  \infty}_{x,   v}$ approach. This choice allows a direct control of $\|\nabla h^\varepsilon\|_\infty$ as in \eqref{jh},   avoiding integration by parts in the Jacobian determinant. In particular,   we have
		\begin{align}\label{1h}
			&\iint_{\{|v'|\le 2N,  \,  |v''|\le 3N\}}
			|\nabla h^{\varepsilon}(s_{1},  X(s_{1}),  v'')| \,  \mathrm{d}v'\mathrm{d}v''
			\le\;  \frac{C_{N,  \eta}}{\varepsilon^{3/2}} \|f^{\varepsilon}(s_{1})\|_{H^1}.
		\end{align}
		This is also why an expansion order $k\ge 4$ suffices,   whereas the $L^2_{x,v}$--$W^{1,  \infty}_{x,v}$ framework in \cite{GuoCMP2010,   lisiam2023} requires $k\ge 6$ to close $\varepsilon^k \|h^\varepsilon\|_{W^{1,  \infty}}\le \sqrt{\varepsilon}$. Our approach yields
		\(
		\sup_{0\le s\le \varepsilon^{-1/2}}\varepsilon^2\|\nabla h^\varepsilon(s)\|_{W^{1,  \infty}}\lesssim 1,
		\)
		so that $\varepsilon^k \|h^\varepsilon\|_{W^{1,  \infty}}\le \varepsilon$ for $k\ge 4$.
		\par
		Moreover,   this framework efficiently handles the exponential nonlinearity $e^{\phi^\varepsilon}$ in \eqref{vpb}$_2$. For instance,   the $v$-growth term
		\(
		\iint (v-u_0) (\sqrt{\mu} f^\varepsilon) \cdot \nabla_x \phi_R^\varepsilon\,  \d v\,  \d x
		\) in \eqref{e1}
		is controlled via the Poisson energy estimate:
		\begin{align}\label{1x}
			\frac{1}{2}\frac{\d}{\d t}\|\phi_R^\varepsilon\|_{H^1}^2
			-\iint v (\sqrt{\mu} f^\varepsilon) \cdot \nabla_x \phi_R^\varepsilon\,  \d v\,  \d x
			\lesssim \varepsilon^{3/2}\|\nabla_x f^\varepsilon\|^2 + \text{other terms}.
		\end{align}
		In contrast to classical treatments~\cite{guo2002cpam,   lisiam2023}, where $
		\|\nabla_x \phi_R^\varepsilon\|_{C^{1,  \alpha}} \le C \|h^\varepsilon\|_{W^{1,  \infty}}
		$,   our estimates \eqref{phif} and \eqref{phiinfnity} give
		\[
		\|\nabla_x \phi_R^\varepsilon\|_{W^{1,  \infty}}
		\lesssim \|f^\varepsilon\|_{W^{1,  \infty}} + \|f^\varepsilon\|_{H^1_x} + \varepsilon.
		\]
The term $\|\nabla_x f^\varepsilon\|$ appearing above can be controlled in the $H^1_{x,v}$--$W^{1,\infty}_{x,v}$ framework.
\par
		Furthermore, the framework also can applies to the classical electron VPB system for all cutoff potentials  $-3<\gamma\le 1$.
	\end{remark}
	
	\begin{remark}
		The restriction $0 \le t \le \varepsilon^{-1/2}$ is dictated by the growth of the higher-order coefficients $H_n(t)$. In Lemma \ref{phiaf},
		$
		\| H_n(t)\|_{H^s} + \|\partial_t H_n(t)\|_{H^s} \lesssim (1+t)^{\,  n-2},
		$
		to ensure
		$
		\Big\|\sum_{n=2k}^{\infty} \varepsilon^{\,  n-k} \big(H_n(t) + \partial_t H_n(t)\big)\Big\|_{H^s} \lesssim \varepsilon
		$
		as in \eqref{bhe},   one must restrict $t \lesssim \varepsilon^{-1/2}$. We also note a difference in the definition of $F_n$ compared with \cite{GuoCMP2010,   lisiam2023}. Following \cite{LeiLiuXiaoZhao2023},   for $0 \le n \le 2k-2$,    $F_n$ involves
		$
		\sum_{i+j=n+1,  \,  i,   j\ge 0} Q(F_i,   F_j),
		$
		whereas in $F_{2k-1}$ the sum is taken as $
		\sum_{i+j=2k,  \,  i,   j\ge 1} Q(F_i,   F_j),
		$ thus excluding the term $Q(F_0,   F_{2k})$. As a result,   the source term $A$ in \eqref{F_R} differs from that in \cite{GuoCMP2010,   lisiam2023},   which allows the time scale to be extended to $t\sim\varepsilon^{-1/2}$; see \eqref{s3}. Moreover, we obtain an integral upper bound for \(\widetilde{\nu}\) by adjusting the range of \(\sigma\), rather than by shortening the time interval as in \cite{lisiam2023}; see Lemma~\ref{nuinter} for more details.
		
		\par
		In fact,   the weight function $w = e^{\widetilde{{\vartheta}}|v|^2}$ is primarily introduced to handle the  case $-3<\gamma<0$,   allowing one to obtain the positive lower bound of $\frac{\widetilde{\nu}}{\e}$ in \eqref{eq:nu:1-1}. For the sake of uniformity,   although different $\sigma$ are used in $w$ for the hard and soft potentials,   we adopt the  same form of weight function for the case $0 \le \gamma \le 1$ in Lemma~\ref{nuw}.
		Of course,   for $0 \le \gamma \le 1$,   alternative weight functions could be used,   exploiting the uniform lower bound $\nu \ge \nu_0 > 0$ in performing the $W^{1,  \infty}_{x,v}$ estimates; this aspect will not be discussed here.
	\end{remark}
	\subsection{Difficulties and Innovations}
	The difficulties and innovations we encountered are as follows:
	\subsubsection{\bf Time Growth of $F_n$ and $\phi_n$}
	Since the remainder equation \eqref{F_R} involves the  coefficients $F_i$ and $\phi_i$ for $i=0,  \dots,   2k-1$,   it is essential to estimate it. Our approach combines mathematical induction with energy and elliptic estimates. A key difficulty arises from the fact that \(|A_0|\sim 1\), so the term \(\langle A_0 V_1, U_1\rangle\) in \eqref{1. 24} cannot be directly controlled by a Gr\"onwall-type argument to yield the bound \(\|U_1\|^2+\|V_1\|^2\le C\).
	To overcome this,   we exploit integration by parts together with elliptic estimates and the Poisson equation for the $\varepsilon^1$-order term in \eqref{coeff},   which gives
	\begin{align*}
		\langle A_0 V_1, \,  U_1 \rangle = \frac{\mathrm{d}}{\mathrm{d} t} \widetilde{\mathcal{E}_1}^2 - \mathcal{U}_1,
	\end{align*}
	where $\widetilde{\mathcal{E}_1}^2 \sim \|\phi_1\|_{H^1}^2$ and $\mathcal{U}_1 \lesssim (1+t)^{-16/15} \big(\|U_1\|^2 + \|\phi_1\|_{H^1}^2\big)$; see \eqref{avu} for details.
	Moreover,   the term $\frac{Q(F_i,   F_j)}{\sqrt{\mu}}$ in $\overline{A}$ is treated using the macroscopic-microscopic decomposition, which allows us to obtain   $\| w_\lambda \overline{A} \|_{H^1_{x,   v}} \lesssim (1+t)^{i+j-2},  $  see \eqref{oa} for details.
	\subsubsection{\bf Growth in $|v|$}
	We identify three representative terms exhibiting velocity growth,   denoted by $\mathcal{H}_i$,   $i=1,   2,   3$.
	First,
	$
	\mathcal{H}_1
	= \left\langle \frac{ v\cdot\nabla_x\sqrt{\mu}}{\sqrt{\mu}}\,   f^\varepsilon, \,  f^\varepsilon \right\rangle.
	$
	By decomposing the velocity into high and low frequency  and applying $L^\infty$ estimates,   we obtain
	\begin{align*}
		\mathcal{H}_1
		\lesssim\;&
		\Big\|\langle v\rangle^3 f^\varepsilon \mathbf 1_{\{\langle v\rangle^{6-\gamma} \ge \kappa^2/\varepsilon\}}\Big\|_\infty \|f^\varepsilon\|  + (1+t)^{-\frac{16}{15}}
		\left(
		\|\langle v\rangle^3 \mathbb{P} f^\varepsilon\|
		+
		\|\langle v\rangle^{3-\gamma/2} \mathbb{P}^{\bot} f^\varepsilon
		\mathbf 1_{\{\langle v\rangle^{6-\gamma}\le \kappa^2/\varepsilon\}}\|_{\nu}
		\right)
		\|f^\varepsilon\|,
	\end{align*}
	see the estimates of $\mathcal{S}_1$,   $\mathcal{S}_5$,   $\mathcal{S}_8$,   etc.\ in Section~3.
	Next, $\mathcal{H}_2$ contains $\mathcal{H}_2^{1}$ and $\mathcal{H}_2^{2}$
	arising from \eqref{e1} and \eqref{e2}, respectively, and
	\[
	\mathcal{H}_2^1
	= \int \Big(\int (v-u_0)\sqrt{\mu} f^\varepsilon\,   \d v \Big)\cdot\nabla_x \phi_R^\varepsilon \,   \d x,
	\qquad
	\mathcal{H}_2^2
	= -\varepsilon \int \Big(\int (v-u_0)\sqrt{\mu} f^\varepsilon\,   \d v \Big)\cdot\nabla_x^3 \phi_R^\varepsilon \,   \d x.
	\]
	Although $\|\phi_R^\varepsilon\|_{H^3}$ is controlled by $\|f^\varepsilon\|_{H^1}$ (see Lemma~\ref{phiaf}),
the estimates for $\mathcal{H}_2^1$ and $\mathcal{H}_2^2$ are too large to allow closure of a Gr\"onwall-type argument.
	To overcome this difficulty,   we derive energy estimates for the Poisson equation $\eqref{F_R}_2$,   see Lemma~\ref{pv},  to eliminate $\mathcal{H}_2^1$   in \eqref{es} and  $\mathcal{H}_2^2$ in \eqref{e20}.
	Finally,   arising from \eqref{ev},   we consider
	\[
	\mathcal{H}_3
	= \langle \nabla_x \mathbb{P}^{\bot} f^\varepsilon,   \varepsilon \nabla_v \mathbb{P}^{\bot} f^\varepsilon \rangle,
	\qquad -3<\gamma<0.
	\]
Note that the energy functional is taken as $\|\sqrt{\varepsilon}\nabla f^\varepsilon\|^2$.
	Since the $\varepsilon$-coefficient is insufficient,   $\mathcal{H}_3$ cannot be directly absorbed into the energy.
	Moreover,   directly estimating $\|\langle v\rangle^{-\gamma}\nabla_x \mathbb{P}^{\bot} f^\varepsilon\|_\nu^2$
	would require  $L^2_{x,   v}$ estimate to differentiating $\nabla_x \eqref{ipgg}$,   leading to uncontrollable $\nabla_x^2 f^\varepsilon$ terms in
	$\|\nabla_x([[\mathbb{P},  \mathcal{A}_{\phi^\varepsilon}]]f^\varepsilon)\|$.
	Thus,   we use
	\[
	\mathcal{H}_3
	\le
	\eta \|\nabla_v \mathbb{P}^{\bot} f^\varepsilon\|_\nu^2
	+ C_\eta \varepsilon^2 \|\langle v\rangle^{-\gamma} \nabla_x f^\varepsilon\|_\nu^2
	+ C_\eta \varepsilon^2 \|f^\varepsilon\|_{H^1}^2.
	\]
	This requires performing an $L^2_{x,   v}$ estimate on $\langle v\rangle^{-\gamma}\nabla_x f^\varepsilon$.
	In particular,   under the local Maxwellian background (cf.\ \eqref{wxl1}),   we obtain
	\begin{align}\label{lwx}
		\frac{1}{\varepsilon}
		\big\langle \nabla_x (L f^\varepsilon),  \,
		\varepsilon^{5/2} \langle v\rangle^{-2\gamma} \nabla_x f^\varepsilon \big\rangle
		\ge\;&
		\frac{\delta_1 \varepsilon^{3/2}}{2}
		\|\langle v\rangle^{-\gamma} \nabla_x f^\varepsilon\|_\nu^2  - C_\eta \sqrt{\varepsilon}\|\sqrt{\varepsilon} f^\varepsilon\|_{H^1}^2
		- \text{other  terms}.
	\end{align}
	The appearance of the remainder term
	$\sqrt{\varepsilon}\|\sqrt{\varepsilon} f^\varepsilon\|_{H^1}^2$
	explains why the energy estimates are formulated in terms of
	$\varepsilon^{5/4}\langle v\rangle^{-\gamma}\nabla_x f^\varepsilon$.
	\subsubsection{\bf Local Maxwellian background $L(x,   v)$}
	The term $\nabla(L(x,   v)f^\varepsilon)$ creates substantial analytical difficulties.
	To overcome this issue,   as  Lemma~\ref{LLMa},   we decompose
	\[
	L(x,   v)=L_M(v)+\text{other terms},
	\]
	The remaining terms are controlled by separating the velocity variable into low  and high frequency.
	Combined with the coercivity of $L_M(v)$ (cf.~Lemma~\ref{lm}):
	\begin{align}\label{lmqq}
		\big\langle \langle v\rangle^{2\lambda_1}\partial_v^\beta (L_M f),  \,  \partial_v^\beta f \big\rangle
		\ge
		\|\langle v\rangle^{\lambda_1} \partial_v^\beta f\|_\nu^2
		- C_\eta \|f\|_\nu^2
		- \text{other terms},
	\end{align}
	which yields the desired dissipation.   However,   the energy functional  is
	$
	\| f^\varepsilon\|^2 + \|\sqrt{\varepsilon}\nabla f^\varepsilon\|^2
	$  in this work.
	Hence,  only for $\varepsilon^{n}\nabla(L f^\varepsilon)$ with $n\geq \tfrac{3}{2}$,   the dissipation estimate \eqref{lmqq} applies directly (cf.~\eqref{lwx}).
	\par
	To estimate $\nabla_x(L f^\varepsilon)$,   we employ the macro-micro decomposition.	Using the identity $\mathbb{P}\mathbb{P}^\bot \nabla_x f^\varepsilon = 0$ (cf.~\eqref{xln}),   we derive
	\begin{align*}
		\frac{1}{\varepsilon}
		\big\langle \nabla_x(L f^\varepsilon),  \,   \varepsilon \theta_0 \nabla_x f^\varepsilon \big\rangle
		\ge
		\delta_1 \theta_M \|\nabla_x \mathbb{P}^\bot f^\varepsilon\|_\nu^2
		- \mathcal{H}_4 - \mathcal{H}_5
		- \text{other terms}.
	\end{align*}
  $\mathcal{H}_4$ is estimated as
	\begin{align*}
		\mathcal{H}_4
		&=
		\left\langle  \sqrt{\mu}\nabla_x\!\left(\frac{1}{\sqrt{\mu}}\right)  \Gamma  \left(\sqrt{\mu},    \mathbb{P}^{\bot} f^\e \right)
		,  \,\theta _0 \nabla_x (\mathbb{P}^{\bot} f^\e)\right\rangle \lesssim
		\delta_0 \|\langle v\rangle^2 \mathbb{P}^\bot f^\varepsilon\|_\nu^2
		+ \delta_0 \|\nabla_x( \mathbb{P}^\bot f^\varepsilon)\|_\nu^2,
	\end{align*}
		see \eqref{n31} for more details.
	\begin{align*}
		\mathcal{H}_5=\left\langle  \sqrt{\mu}\nabla_x\!\left(\frac{1}{\sqrt{\mu}}\right)  \Gamma  \left(\sqrt{\mu},    \mathbb{P}^{\bot} f^\e \right)
		, \, \theta _0\nabla_x(\mathbb{P} f^\e) \right\rangle,
	\end{align*}
	which produces $\|\nabla_x (\mathbb{P} f^\varepsilon)\|_{\nu}$ via \eqref{Gf}.
	This term cannot be controlled by the energy functional  $\|f^\varepsilon\| +\|\sqrt{\e}\nabla_xf^\varepsilon\|$,  and  the dissipation 	$
	\frac{\delta_1 \varepsilon^{3/2}}{2} \|\langle v \rangle^{-\gamma} \nabla_x f^\varepsilon\|_\nu^2
	$ in  \eqref{lwx}.
	Therefore,   we   apply integration by parts in $x$ and use \eqref{Gf} to obtain:
	\begin{align*}
		\mathcal{H}_5=&-\left\langle \theta_0 \sqrt{\mu}\nabla_x\!\left(\frac{1}{\sqrt{\mu}}\right)  \Gamma  \left(\sqrt{\mu},    \nabla_x (\mathbb{P}^{\bot} f^\e) \right)
		,    \mathbb{P} f^\e \right\rangle- \text{l.o.t}\\\lesssim&
	(1+t)^{-16/15}\|f^\e\|^2+\delta_0\| \nabla_x (\mathbb{P}^{\bot} f^\e)\|_{\nu}^2+\delta_0\|   (\mathbb{P}^{\bot} f^\e)\|_{\nu}^2,
	\end{align*}
		see \eqref{n41} for more details.
	The presence of \((1+t)^{-16/15}\|f^\varepsilon\|^2\) explains why it is necessary to take $L^2_{x,v}-$ estimate  of \(\sqrt{\e} \|\nabla_x f^\varepsilon\| \).
	Other terms can be controlled by similar methods.
	\par
	Finally,   for the estimate of $\|\nabla_v f^\varepsilon\|$,   we note the equivalence
	\[
	\|f^\varepsilon\|+\|\sqrt{\varepsilon}\nabla_v f^\varepsilon\|
	\sim
	\|f^\varepsilon\|
	+\|\sqrt{\varepsilon}\nabla_v(\mathbb{P}^{\bot}f^\varepsilon)\|,
	\]
	which motivates us to estimate
	$\|\sqrt{\varepsilon}\nabla_v (\mathbb{P}^{\bot}f^\varepsilon)\| $ rather than $\|\sqrt{\varepsilon}\nabla_v f^\varepsilon\| $.
This approach
 yields the required dissipation $\|\nabla_v(\mathbb{P}^{\bot}f^\varepsilon)\|_{\nu}$, while the lower-order term  $C_{\eta} \|\mathbb{P}^{\bot}f^\varepsilon\|_{\nu}$ can be absorbed (see \eqref{vl}). Moreover, we can eliminate the   growing velocity term through the fact $
 \mathbb{P}^{\bot}\!\left( \frac{v-u_0}{\theta_0}\cdot \nabla_x \phi_R^\varepsilon\sqrt{\mu}\right)=0.
 $

	\subsection{Relevant literature}
	
	First,   we provide some necessary background on ion dynamics. Plasma,   composed of electrons and ions,   is a prevalent state of matter in both natural and laboratory environments,   such as lightning,   the ionosphere,   stellar coronae,   accretion disks,   and the interstellar medium \cite{Sentis2014,   Tan2024,   TrivelpieceKrall1973}. Its macroscopic behavior is primarily governed by electromagnetic forces. For example,   stellar atmosphere plasmas consist of electrons,   protons,   helium ions,   and a small fraction of heavier ions \cite{Phillips2013}. From a kinetic perspective,   plasma is described by statistical mechanics,   which bridges single-particle dynamics and magnetohydrodynamics. Charged particles interact predominantly through Coulomb collisions within a sphere of radius equal to the Debye length. The collision frequencies for electron-electron,   electron-ion,   and ion-ion interactions can be expressed explicitly; see \cite{Tan2024}.
	Owing to the large ion-to-electron mass ratio,   electron-electron and electron-ion collisions occur on much faster timescales than ion-ion collisions. Consequently,   on the ion timescale,   electrons can be regarded as being in thermal equilibrium,   with their density following a Boltzmann distribution,   see \cite{BardosGolseNguyenSentis2018,   FlynnGuo2024,   Guo2002}. Substituting this relation into Poisson's equation leads to the semilinear Poisson--Poincar\'e equation $\Delta \phi^\e = e^{\phi^\e} - \rho$ in $\eqref{vpb}_3$. Interested readers may refer to \cite{LiWang2025} for a more detailed derivation.
	
	Next,   we state the ionic Vlasov--Poisson
	system,   which models the dynamics of ions from a kinetic perspective. The system is given by:
	\begin{equation}\label{vp}
		\partial_t f + v \cdot \nabla_x f - \nabla \phi[f] \cdot \nabla_v f = 0,   \qquad
		-\Delta \phi[f] = \int_{\mathbb{R}^d} f \,   \d v - e ^{\phi[f]},
	\end{equation}
	where $f(t,   x,   v)$ is the ion distribution function and $\phi[f]$ is the self-consistent potential determined by the Poisson--Boltzmann relation.
	The global existence of weak solutions for the ionic Vlasov--Poisson system was first established by Bouchut \cite{Bouchut1991}. Subsequent works have investigated well-posedness,   propagation of moments,   and stability in the Wasserstein distance \cite{BardosGolseNguyenSentis2018,   Guo2002,   FlynnGuo2024},   with a comprehensive survey available in \cite{GriffinPickering}. Bardos et al. \cite{BardosGolseNguyenSentis2018} rigorously derived the ionic Vlasov--Poisson system as the massless-electron limit of the coupled ion-electron Vlasov--Poisson equations,   incorporating electron collisions via Boltzmann  or BGK-type operators. The mean-field limit of the system was justified deterministically in \cite{GriffinPickeringIacobelli2020},   and subsequently a probabilistic formulation was developed in \cite{GriffinPickering2024}. In a complementary direction,   Han-Kwan \cite{HanKwan2011} analyzed the quasineutral limit,   formally deriving macroscopic models including the compressible Euler and shallow-water equations. More recently,   global-in-time existence results for Lagrangian and renormalized solutions have been obtained under initial data with $L^{1+}$ integrability \cite{ChoiKooSong2025},   which also established the well-posedness of the Poisson--Boltzmann equation $-\Delta \phi = \rho - e^\phi$ for density functions $\rho \in L^p(\mathbb{T}^d)$ with $p>1$.
	
	To account for collisional effects,   we introduce a Boltzmann-type or Landau-type collision operator $Q(f,   f)$ on the right-hand side in \eqref{vp}. Flynn and Guo \cite{FlynnGuo2024} formally derived the Vlasov--Poisson--Landau system from a kinetic ion-electron model via the massless-electron limit. Flynn \cite{Flynn2025} proved the local well-posedness of the ionic Vlasov--Poisson--Landau system in $\mathbb{T}^3$ for large initial data. Li-Yang-Zhong \cite{LiYangZhong2016} constructed the existence and uniqueness of the strong solution to the ionic Vlasov--Poisson--Boltzmann system \eqref{vpb} in $\mathbb{R}^3$. In a related development,   Li and Wang \cite{LiWang2025} established the global existence and exponential decay of classical solutions to \eqref{vpb} in $\mathbb{T}^3$.
	
	Furthermore,   we briefly review classical results on the compressible Euler--Poisson limit of the Boltzmann equation obtained via the Hilbert expansion method. This approach,   originally introduced by Hilbert in 1912,   is based on seeking a formal asymptotic expansion of solutions to the scaled Boltzmann equation in powers of the Knudsen number. A rigorous justification of the formal Hilbert expansion was first achieved through truncated asymptotic expansions. In the whole-space or periodic setting,   Caflisch \cite{Caflisch1980} derived the compressible Euler limit by means of a truncated Hilbert expansion,   under the restrictive assumption that the initial remainder vanishes. This assumption was later removed by Lachowicz \cite{Lachowicz1987},   who incorporated initial layers into the expansion,   thereby allowing for more general initial data. Guo et al. \cite{GuoCMP2010} also highlighted the restrictive assumption on the initial remainder in Caflisch's framework. By employing the $L^2$--$L^\infty$ method originally developed in \cite{Guo2010Decay},   they revised Caflisch's proofs to obtain the compressible Euler--Poisson limit for hard-sphere case,   and Jiang et al.\cite{Jiang2025EulerPoissonFD} extended this to the Vlasov--Poisson--Boltzmann system with Fermi--Dirac statistics.
	Using a similar methodology  as \cite{Guo2010Decay},   Li et al. \cite{lisiam2023} extended the approach to soft potentials.
	Once boundary effects are considered,   deriving the compressible Euler limit requires accounting for boundary layers in the half-space problem. Sone \cite{Sone2007} highlighted the need for viscous and Knudsen layers,   while Guo,   Huang,   and Wang \cite{GuoHuangWang2021},   and Jiang,   Luo,   and Tang \cite{JiangLuoTang2021a,   JiangLuoTang2021b} treated specular,   Maxwell,   and diffuse reflection conditions. These studies mainly focus on hard-sphere kernels.
	\par
	Up to now, due to the analytical challenges posed by the nonlinear Poisson equation,
	the   Euler--Poisson limit of the ionic Vlasov--Poisson--Boltzmann system \eqref{vpb} for the full range of cutoff potentials $-3 < \gamma \le 1$  had not been investigated. In this paper, we completely resolve this problem.
	\bigskip

	\noindent \textbf{Notation.}\;Throughout the paper,
	$c$ and
	$C$ denote a generic positive constant independent of~$\varepsilon$.
	For simplicity of notation,   we write $X \lesssim Y$ if $X \leq C Y$ for some constant $C>0$ independent of $X$ and $Y$.
	The notation $X \ll 1$ means that $X$ is sufficiently small.
	We use $\langle \cdot,   \cdot \rangle_{L^2_v}$ to denote the $L^2(\mathbb{R}^3_v)$ inner product,   with corresponding norm $\|\cdot\|_{L^2_v}$.
	The notation $\langle \cdot,   \cdot \rangle$ stands for the $L^2$ inner product either on $\mathbb{R}^3_x$ or on $\mathbb{R}^3_x \times \mathbb{R}^3_v$  with corresponding norm $\|\cdot\|_{L^2_{x}}$ or $\|\cdot\|_{L^2_{x,   v}}$.
	For $1 \le q \le \infty$,   without causing ambiguity,   we abbreviate the norms $\|\cdot\|_{L^q_x}$ and $\|\cdot\|_{L^q_{x,   v}}$ as $\|\cdot\|_q$,   and $\|\cdot\|_2$ is simply denoted by $\|\cdot\|$.
	For any integer $m \ge 1$,   we write $\|\cdot\|_{W^{m,   q}_{x}}$  or  $\|\cdot\|_{W^{m,   q}_{x,   v}}$   as $\|\cdot\|_{W^{m,   q}}$.
	When $q=2$,   we write $\|\cdot\|_{H^m}$ instead of $\|\cdot\|_{W^{m,   2}}$ and $\|\cdot\|_{H^0}=\|\cdot\|$. In particular, $\|g\|_{H^1_x}=\|g\| +\|\nabla_x g\| $.
	The weighted norms are defined by
	\[
	\|g\|_{\nu}^2 = \iint_{\mathbb{R}^3_x \times \mathbb{R}^3_v} \nu(v)\,   g^2(x,   v)\,   \d x\,   \d v,
	\qquad
	\|g\|_{H^1(\nu)}^2 = \|g\|_{\nu}^2 + \|\nabla _x g\|_{\nu}^2 + \|\nabla _v g\|_{\nu}^2,
	\]
	where $\nu(v) = \int_{\mathbb{R}^3} |v - u| ^{\gamma}\,   \mu(u) \,   \d u.	$ For convenience,   the product norm on $X$ is defined by
	$
	\|(f,   g)\|_{X} := \|f\|_{X} + \|g\|_{X}.
	$ $\langle v \rangle=(1+|v|^2)^{1/2}$.
	The operators $\nabla_x = (\partial_{x_1},  \partial_{x_2},  \partial_{x_3})$ and
	$\nabla_v = (\partial_{v_1},  \partial_{v_2},  \partial_{v_3})$ denote the spatial and velocity gradients,   respectively.
	We also write $\nabla g = \nabla_x g + \nabla_v g$ for brevity. For $s \in \mathbb{N}:=\{0,1,2,\cdots\}$,   the notation $\nabla^{s}$ denotes any partial derivative $\partial^{\beta}$ with multi-index $\beta$ satisfying $|\beta|=s$.	
	The symbol $a\cdot b$ denotes the standard inner product (dot product) of vectors $a$ and $b$, and is also used for contraction between vectors and higher-order tensors in formulas.
	\par

	\bigskip
	
	\noindent\textbf{The structure.}\;	The remainder of this paper is organized as follows.
	Section~2 is devoted to several new properties that are specific to the present problem and apply uniformly to both hard and soft potentials.
	Section~3 establishes weighted $H^1_{x, v}$ energy estimates for the remainder system \eqref{F_R} in the soft potential case.
	Section~4 further derives $W^{1,\infty}_{x, v}$ estimates with suitable time--velocity weights for soft potentials.
	Section~5 is devoted to the proofs of Theorems~\ref{t1} and~\ref{t2}.
	Finally, Appendix~A collects lemmas establishing time-dependent estimates for $\rho_i, u_i$, $\phi_i$, and $F_i$, which play a crucial role in the $H^1_{x,v}$ and $W^{1,\infty}_{x,v}$ analyses. Appendix~B gathers several preparatory lemmas used throughout the paper for both hard and soft potentials.
	\section{New Lemmas for  Hard and Soft Potentials}
	With \eqref{nunum},  for $s=0,1$, it follows  that
	\begin{align}\label{h}
		|\nabla^s f^\varepsilon|
		\lesssim e^{-\frac{\widetilde{\vartheta}}{2}|v|^2}\sum_{j=0}^{s}\|\nabla^j h^\varepsilon\|_\infty.
	\end{align}
	\par
	Due to the factor $e^{\varepsilon^k \phi_R^{\varepsilon}}$ in $\eqref{F_R}_2$,   the relation between $\phi_R^\varepsilon$ and $f^\varepsilon$ is obtained via the  a  $\mathit{priori}$  assumption
	\begin{align}\label{pri}
		\e^{k} \sup_{0 \leq t \leq \e^{-1/2}} \big( \|\phi_R^\varepsilon\|_{H^2} + \|h^\varepsilon\|_{W^{1,  \infty}} \big) \leq \varepsilon.
	\end{align}
	Rewrite $\eqref{F_R}_2$ as
	\begin{equation}\label{1eq:remainder}
		-\phi_R^\varepsilon+\Delta \phi_R^\varepsilon =		R_0,
	\end{equation}
	where
	$R_0=e^{\phi_0}\frac{e^{\varepsilon^k \phi_R^\varepsilon} - 1}{\varepsilon^k}	-\phi_R^\varepsilon+G 	- \int_{\mathbb{R}^3}F_R^ \varepsilon\,   \d v$,   $F_R^ \varepsilon=\sqrt{\mu}f^ \varepsilon$.	 	
	
	\begin{lemma}\label{phiaf}
		Suppose  $0
		\leq t\leq {\varepsilon^{- 1/2}}$,   the remainder \( \phi_R^\varepsilon \)  satisfies
		\begin{align}\label{phif}
			&	\|   \phi_R^\varepsilon \|_{H^{2}}
			\lesssim  \|  f^\varepsilon \| +\varepsilon,  \\
			&\label{phi1} \|   \nabla_x\phi_R^\varepsilon \|_{H^{2}}
			\lesssim   (1+t)^{-16/15} \|f^\e\|+ \|\nabla_xf^\e\|+\e,  \\
			&\label{phiftt}	\| \partial_t \phi_R^\varepsilon \|_{H^2} \lesssim   (1+t)^{-16/15} \|f^\varepsilon\|_{ }+\|\nabla_xf^\varepsilon \| +\e,
		\end{align}
		and
		\begin{align}\label{phiinfnity}
			\|\nabla_x^2 \phi_R^\varepsilon\|_{\infty} \lesssim \|h^{\varepsilon}\|_{W^{1,  \infty} }+\|f^{\varepsilon}\|_{H^{1}   }	 +\varepsilon 	.
		\end{align}
		Moreover,  the  $\phi^\e$
	defined in \eqref{exp}	satisfies
		\begin{equation}\label{phi}
			\|\nabla_x \phi ^\varepsilon\|\lesssim	{\delta_0},  	 \quad\|\nabla_x \phi ^\varepsilon\|_{W^{1,  \infty}}\lesssim \delta_0	(1+t)^{-16/15}.
		\end{equation}
	\end{lemma}
	\begin{proof}
		By elliptic regularity for \eqref{1eq:remainder},  we  have
		\begin{equation}\label{phi31}
			\| \phi_R^\varepsilon \|_{H^{2}}\lesssim		\| R_0 \|,  	\quad 	\|\nabla_x \phi_R^\varepsilon \|_{H^{2}}\lesssim		\| \nabla_xR_0 \|,   \quad 	 \| \partial_{t}\phi_R^\varepsilon \|_{ H^2}\lesssim		\| \partial_{t}R_0 \|_{ }	.
		\end{equation}
		To estimate the terms on the right-hand side of \eqref{phi31}, we first bound $G$.
		For the integer $s\geq 0$,   by \eqref{fir2},  we have
		\begin{align*}
			\|H_n(t)\|_{ H^s} \le& \sum_{m=2}^{n} \frac{1}{m!} \sum_{\substack{n_1 + \cdots + n_m = n \\ 1 \le n_i \le 2k - 1}} \left\| \prod_{i=1}^{m} \phi_{n_i}(t) \right\|_{ H^s} \leq \sum_{m=2}^{n} \frac{1}{m!} \binom{n - 1}{m - 1} (1 + t)^{\sum_{i=1}^{m}(n_i - 1)}  \leq     4^n  (1 + t)^{n -2},
		\end{align*}
		and
		\begin{align*}
			\|\partial_{t}H_n(t)\|_{ H^s}
			\lesssim&\sum_{m=2}^{n} \frac{1}{m!} \sum_{\substack{n_1 + \cdots + n_m = n}}  \sum_{j=1}^m \|\partial_t \phi_{n_j}(t)\|_{ H^2}  \prod_{i \ne j} \|\phi_{n_i}(t)\|_{ H^{\max\{s,2\}}}
			\lesssim 4^n (1+t)^{\,  n-2}.
		\end{align*}
		Using the same method as above, it holds that
		\begin{equation*}
			\|B_n(t)\|_{H^2} +	\|\partial_{t}B_n(t)\|_{H^2} \lesssim   (1 + t)^{n-1}.
		\end{equation*}
		With $0\leq t\leq {\varepsilon^{- 1/2}}$,  and $\sum_{n=2k}^{\infty} \varepsilon^{n-k}4^n(1+t)^{n-2}\leq  16 \varepsilon^{2-k} \sum_{n=2k}^{\infty}\left(4\e(1+t) \right) ^{n-2}
		\lesssim \e^{k}(1+t)^{2k-2} \lesssim \e$,   we have
		\begin{align}\label{bhe}
			&	\left\|\sum_{n=2k}^{\infty} \varepsilon^{n-k} H_n(t)\right\|_{ H^s}+	\left\|\sum_{n=2k}^{\infty} \varepsilon^{n-k} \partial_{t}H_n(t)\right\|_{ H^s}	  \lesssim \e
			,  \\\nonumber
			&\left\| \sum_{n=2k}^\infty \varepsilon^n H_n(t) \right\|_{ H^s}+ 	 \left\| \sum_{n=2k}^\infty \varepsilon^n \partial_{t}H_n(t) \right\|_{ H^s}\lesssim \e^{k+1},  \\\nonumber
			&\left\|\sum_{n=1}^{2k-1} \varepsilon^n B_n \right\|_{ H^s}+	\left\|\sum_{n=1}^{2k-1} \varepsilon^n \partial_{t} B_n \right\|_{ H^s}\lesssim \e.
		\end{align}
		Applying  Taylor expansion,  the Sobolev embedding $ H^2(\mathbb{R}^3)\hookrightarrow L^\infty(\mathbb{R}^3)$,  \eqref{pri},  \eqref{bhe},  \eqref{fir2}, \eqref{decay} and    \eqref{ttphi0},  for $i=0,1,2$, we have
		\begin{equation}\label{gh}
			\|\nabla_x^iG\| \lesssim \left\| \sum_{n=2k}^\infty \varepsilon^{n-k} H_n(t) \right\|_{H^i} +\left( 	\left\|\sum_{n=1}^{2k-1} \varepsilon^n B_n \right\|_{ H^{2+i}}+\left\| \sum_{n=2k}^\infty \varepsilon^n H_n(t) \right\|_{ H^{2+i}}\right)\|\phi_R^\varepsilon\|_{H^i} \lesssim \varepsilon \|\phi_R^\varepsilon\|_{H^i} +\varepsilon,
		\end{equation}	
		and
		\begin{align}\label{gt}
			&	\|\partial_{t}G\|_{ }	  \lesssim  \|\partial_{t}\phi_0\|_{\infty}\left\| \sum_{n=2k}^\infty \varepsilon^{n-k} H_n(t) \right\|+ \left\| \sum_{n=2k}^\infty \varepsilon^{n-k} \partial_{t}H_n(t) \right\|   \\\nonumber
			&+\left\{\|\partial_{t}\phi_0\|_{\infty}\left( 	\left\|\sum_{n=1}^{2k-1} \varepsilon^n B_n \right\|_{ H^2}+\left\| \sum_{n=2k}^\infty \varepsilon^n H_n(t) \right\|_{ H^2}\right)+\left\|\sum_{n=1}^{2k-1} \varepsilon^n \partial_{t}B_n \right\|_{ H^2}+\left\| \sum_{n=2k}^\infty \varepsilon^n \partial_{t}H_n(t) \right\|_{ H^2}\right\}\|\phi_R^\varepsilon\|\\\nonumber
			&+\left( 	\left\|\sum_{n=1}^{2k-1} \varepsilon^n B_n \right\|_{ H^2}+\left\| \sum_{n=2k}^\infty \varepsilon^n H_n(t) \right\|_{ H^2}\right) \|\partial_{t}\phi_R^\varepsilon\|
			\lesssim  \varepsilon \|\phi_R^\varepsilon\|_{ }+\varepsilon \|\partial_{t}\phi_R^\varepsilon\|_{ }+ \varepsilon.
		\end{align}
		\par	
		Applying the Taylor expansion,   the   Sobolev embedding $ H^2(\mathbb{R}^3)\hookrightarrow L^\infty(\mathbb{R}^3)$,   and \eqref{pri},   we write
		\[
		e^{\phi_0}\frac{e^{\varepsilon^k \phi_R^\varepsilon}-1}{\varepsilon^k} - \phi_R^\varepsilon = (e^{\phi_0}-1)\phi_R^\varepsilon +  {O}(\varepsilon^k (\phi_R^\varepsilon)^2).
		\]
		Using \(\|e^{\phi_0}-1\|_\infty \lesssim \|\phi_0\|_\infty \lesssim \delta_0 (1+t)^{-16/15}\) from \eqref{decay}
		and \eqref{gh},  it holds that
		\begin{align*}
			\|R_0\| \lesssim&	 \left(\| e^{\phi_0} -1\|_{\infty}	+ \e^k\|\phi_R^\e  \|_{ H^2} \right)\| \phi_R^\varepsilon\|_{ }+\|G\|  +\left\|\int \sqrt{\mu}f^\e \d v\right\|  \\			\lesssim&\left(  \delta_0(1+t)^{-\frac{16}{15}}+ {\e}\right) \|\phi_R^\varepsilon\|+\|f^\e\|+\e.
		\end{align*}
		Combining \eqref{phi31},  and $\left(  \delta_0(1+t)^{-\frac{16}{15}}+ {\e}\right) \|\phi_R^\varepsilon\|$  can be absorbed by $\|\phi_R^\varepsilon\|_{H^2}$,   we prove \eqref{phif}.
		\par
		Differentiating $R_0$ with respect to $x$,  using \eqref{pri},
		the expansion $e^{\varepsilon^k \phi_R^\varepsilon}=1+ {O}(\varepsilon^k \phi_R^\varepsilon)$,
		\eqref{decay},
		the Sobolev embedding $ H^2(\mathbb{R}^3)\hookrightarrow L^\infty(\mathbb{R}^3)$,
		together with \eqref{gh} and \eqref{phif},   we obtain
		\begin{align*}
			\| \nabla_x R_0 \|
			\lesssim{}& \|\nabla_x \phi_0\|_{\infty} \|\phi_R^\varepsilon\|
			+ \Big( \|e^{\phi_0}-1\|_{\infty} + \varepsilon^k \|\phi_R^\varepsilon\|_{\infty} \Big)\|\nabla_x \phi_R^\varepsilon\|
			+ \|\nabla_x G\|
			+ \left\| \int \nabla_x(\sqrt{\mu} f^\varepsilon)\,   \d v \right\| \\
			\lesssim{}& \big( (1+t)^{-16/15} + \varepsilon \big)\|\phi_R^\varepsilon\|_{H^1}
			+  (1+t)^{-16/15}\|f^\varepsilon\|
			+ \|\nabla_x f^\varepsilon\|
			+ \varepsilon \\
			\lesssim{}& \big(  (1+t)^{-16/15} + \varepsilon \big)\|f^\varepsilon\|
			+ \|\nabla_x f^\varepsilon\|
			+ \varepsilon.
		\end{align*}
		Combining \eqref{phi31} and the time-scale assumption $0\le t\le \varepsilon^{-1/2}$,
		which yields $\varepsilon \lesssim (1+t)^{-16/15}$,
		we deduce \eqref{phi1}.
		\par
		Using  $\eqref{F_R}_1$,   \eqref{decay}, \eqref{ttphi0}, \eqref{gt} and \eqref{phif},   we have
		\begin{align}\label{tr}
			\|	\partial_{t}R_0 \|_{ }	
			\lesssim&\|\partial_t\phi_0\|_{\infty}	
			\|\phi_R^\varepsilon\|_{ } +\left( \| e^{\phi_0} -1\|_\infty e^{\e^k \|\phi_R^\varepsilon\|_{H^2}}+\|e^{\varepsilon^k \phi_R^\varepsilon}-1\|_{\infty}\right)
			\|\partial_{t}\phi_R^\varepsilon\|_{ } +\|\partial_{t} G\|_{ }+\left\| \int v \cdot \nabla_x ( \sqrt{\mu} f^\varepsilon ) \,   \d v \right\|_{ }	\nonumber \\
			\lesssim&\left(\delta_0 (1+t)^{-16/15}+{\e}\right)\|\partial_{t}\phi_R^\varepsilon\|_{ }+   (1+t)^{-16/15}  \|f^\varepsilon\|_{ }+\|\nabla_xf^\varepsilon \| +\e.
		\end{align}	
		Substituting \eqref{tr} into \eqref{phi31},  with $ 0<\delta_0\ll 1$,     yields \eqref{phiftt}.
		\par
		Applying the Gagliardo--Nirenberg  inequality and  Sobolev embedding inequality  $W^{1,   p}(\mathbb R^3)\hookrightarrow C^{0,   1-n/p}(\mathbb R^3)$ for $p>3$,  we have
		\begin{equation*}
			\|\nabla_x^2 \phi_R^\varepsilon\|_{L ^\infty}
			\lesssim
			\|\nabla_x^2 \phi_R^\varepsilon\|_{L ^2}^{\tfrac{1}{4}}
			\big[ \nabla ^2 \phi   _R^\varepsilon \big]^{\tfrac{3}{4}}_{C^{0,   1/2}(\mathbb{R}^3)}\lesssim
			\|\nabla_x^2 \phi_R^\varepsilon\|_{L ^2}^{\tfrac{1}{4}}
			\|\nabla_x^2 \phi_R^\varepsilon\|_{W^{1,   6} }^{\tfrac{3}{4}}\lesssim\|\nabla_x^2 \phi_R^\varepsilon\|_{L ^2}+
			\|\nabla_x^2 \phi_R^\varepsilon\|_{W^{1,   6} }.
		\end{equation*}
		Recall $\eqref{F_R}_2$,  by elliptic regularity,  \eqref{pri},  Sobolev embedding $H^2(\mathbb R^3)\hookrightarrow W^{1,   6}(\mathbb R^3)$,  Gagliardo--Nirenberg  inequality,  \eqref{gh},    \eqref{phif},  and \eqref{h},   we have
		\begin{align*}
			&	\|\nabla_x^2 \phi_R^\varepsilon\|_{W^{1,   6} }=  \left\|e^{\phi_0} \frac{(e^{\varepsilon^k \phi_R^{\varepsilon}}-1
				)}{{\varepsilon}^k}+ G-\int_{\mathbb{R}^3} F_R^\varepsilon\,   \d v\right\|_{W^{1,   6} }
			\lesssim \|  \phi_R^{\varepsilon} \|_{W^{1,   6} }+\|G\|_{W^{1,   6} }+\|f^{\varepsilon}\|_{W^{1,   6}   }\\
			\lesssim&\|  \phi_R^{\varepsilon} \|_{ H^2}+\|G\|_{ H^2}+\|f^{\varepsilon}\|_{W^{1,  \infty}   }+\|f^{\varepsilon}\|_{H^{1} }
			\lesssim \|h^{\varepsilon}\|_{W^{1,  \infty}   }+ \|f^{\varepsilon}\|_{H^{1}   }+\e,
		\end{align*}
		which gives \eqref{phiinfnity}. Substituting \eqref{decay}, \eqref{ttphi0},   \eqref{fir2},   \eqref{phiftt},   \eqref{phiinfnity},   and \eqref{pri} into \eqref{exp} yields \eqref{phi}.
	\end{proof}
	
	The result in Lemma~\ref{lm} is stated for  $L_M(v)$ defined in \eqref{delm} and does not directly apply   $L(x,   v)$  defined in \eqref{deln}. 	To address this problem,   we employ the following decomposition:
	\begin{lemma}\label{LLMa}
		Suppose $-3 < \gamma\leq 1$,   it holds that
		\begin{align}\label{LLM}
			Lg= 	L_Mg-\mathcal{I}(g)=
			L_Mg-\sum_{i=1}^{4} \mathcal{I}_i (g),
		\end{align}
		where
		\begin{align*}
			\mathcal{I}_1 (g)=&\frac{\sqrt{\mu}-\sqrt{\mathcal{M}}}{\sqrt{\mu }}	 \Gamma_M\left( \sqrt{\mathcal{M}},  g\right),    \quad
			\mathcal{I}_2 (g)=\Gamma\left(\frac{\mathcal{M}}{\sqrt{\mu}},  \frac{\left(\sqrt{\mu}-\sqrt{\mathcal{M}}\right)g}{\sqrt{\mu}}\right),  \quad 	\mathcal{I}_3 (g)= \Gamma\left(\frac{\mu-\mathcal{M}}{\sqrt{\mu}},  g\right),   \\
			\mathcal{I}_4 (g)=&\frac{\sqrt{\mu}-\sqrt{\mathcal{M}}}{\sqrt{\mu }}	 \Gamma_M\left(g,  \sqrt{\mathcal{M}}\right)+\Gamma\left(\frac{\left(\sqrt{\mu}-\sqrt{\mathcal{M}}\right)g}{\sqrt{\mu}},  \frac{\mathcal{M}}{\sqrt{\mu}}\right) +\Gamma\left(g,  \frac{\mu-\mathcal{M}}{\sqrt{\mu}}\right).
		\end{align*}
		Let ${\overline{w} }=\langle v \rangle^{\lambda_1}e^{\lambda_2|v|^2}, \,  0\leq \lambda_1, \,  0\leq \lambda_2 \ll 1	$, it holds that
		\begin{align}\label{lg}
			\langle \mathcal{I} (g), \, {\overline{w} }^2f\rangle	\lesssim \left( \kappa^2+\delta_0 \right)\|{\overline{w} } g\|_{\nu}    \|{\overline{w} } f\|_{\nu},  	
		\end{align}
		and
		\begin{align}\label{lgv}
			\langle \nabla_v\mathcal{I} (g), \, {\overline{w} }^2 f\rangle	\lesssim\left(\delta_0+\kappa^2\right)\left(\|{\overline{w} } g\|_{\nu}^2 +\| \langle v \rangle {\overline{w} } g\|_{\nu}^2 +\|{\overline{w} } \nabla_vg\|_{\nu}^2 \right) +\left(\delta_0+\kappa^2\right)\|{\overline{w} }f\|_{\nu}^2.  	
		\end{align}
		
	\end{lemma}
	\begin{proof}
		Owing to the bilinearity of the collision operator $Q$ and the definitions of $L$,   $\Gamma$,   $L_M$,   and $\Gamma_M$ in \eqref{deln} and \eqref{delm},   the decomposition \eqref{LLM} follows directly.
		By first-order Taylor expansion in the parameters \((\rho_0,  {\theta }_0,   u_0)=(n_0,  {\theta }_M,   0)\) and  \eqref{decay},  there exists $q_1\in(3/4,   1)$, it holds that
		\begin{align}\label{3. 1}
			| {\mu}	- {\mathcal{M}}	 |\lesssim& \| \left(
			|\rho_0-n_0|
			+|{\theta }_0-{\theta }_M|
			+ |u_0|
			\right)\|_{\infty} 		
			(1+|v|^2)\mathcal{M}\lesssim\delta_0 (1+t)^{-16/15}\mathcal{M}^{q_1}.
		\end{align}
		By $c\mathcal{M}\leq \mu  $ and \eqref{3. 1},  we have	\begin{align}\label{3. 2}
			|\sqrt{\mu}	-\sqrt{\mathcal{M}}	 |\le \frac{|\mu-\mathcal{M}|}{\sqrt\mu+\sqrt{\mathcal{M}}}\lesssim&\delta_0 (1+t)^{-16/15}\mathcal{M}^{q_1-1/2}.
		\end{align}
		Using the mean value theorem,   $c\mathcal{M}\leq \mu  $ and \eqref{3. 1},  there exists $\xi $   between $\mu$ and $\mathcal{M}$,    thus
		\begin{align}\label{mu14}
			\big|\mu^{1/4}-\mathcal{M}^{1/4}\big|=\frac{1}{4}\xi^{-3/4}|\mu-\mathcal{M}|
			\lesssim\mathcal{M}^{-3/4}|\mu-\mathcal{M}|
			\lesssim  \delta_0(1+t)^{-16/15}\mathcal{M}^{q_1-3/4}.
		\end{align}
		\par
		To obtain a small coefficient at large velocities,    we decompose the velocity space into high and low velocity regions.
		Let $\eta \in \left(0,   \frac{1}{16\theta_M}\right)$ be a small constant,   and define
		$
		C_\kappa = \sqrt{\frac{1}{\frac{1}{8\theta_M}-\eta} \ln\left( \frac{(2\pi \theta_M)^{-3/4}}{\kappa^2} \right)},
		$
		so that
		$
		\left\| e^{\eta|v|^2} {\overline{w}} \sqrt{\mathcal{M}} \mathbf{1}_{\{|v|\geq C_\kappa\}} \right\|_\infty \lesssim \kappa^2.
		$
		Then,   	treating $ \frac{\sqrt{\mu}-\sqrt{\mathcal{M}}}{\sqrt{\mu}} $  as a weight function,
		applying \eqref{gw} together with the uniform bound
		$\left|\frac{\sqrt{\mu}-\sqrt{\mathcal{M}}}{\sqrt{\mu}}\right|\lesssim 1$,		
		combining \eqref{3. 2},   we have
		\begin{align}\label{2. 59}
			\left\langle \mathcal{I}_{1}(g),\,   {\overline{w}}^2 f \right\rangle
			\lesssim & \Bigg( \left\| e^{\eta|v|^2} {\overline{w}}  \frac{\sqrt{\mu}-\sqrt{\mathcal{M}}}{\sqrt{\mu}}  \sqrt{\mathcal{M}} \mathbf{1}_{\{|v|\geq C_\kappa\}} \right\|_\infty
			+ \left\| e^{\eta|v|^2} {\overline{w}}  \frac{\sqrt{\mu}-\sqrt{\mathcal{M}}}{\sqrt{\mu}}  \sqrt{\mathcal{M}} \mathbf{1}_{\{|v|\leq C_\kappa\}} \right\|_\infty \Bigg) \nonumber\\
			&\quad \times \|{\overline{w}} g\|_\nu \|{\overline{w}} f\|_\nu \nonumber\\
			\lesssim & \Big( \left\| e^{\eta|v|^2} {\overline{w}} \sqrt{\mathcal{M}} \mathbf{1}_{\{|v|\geq C_\kappa\}} \right\|_\infty + \|\sqrt{\mu} -\sqrt{\mathcal{M}} \|_\infty \Big) \|{\overline{w}} g\|_\nu \|{\overline{w}} f\|_\nu \nonumber\\
			\lesssim & (\kappa^2 + \delta_0) \|{\overline{w}} g\|_\nu \|{\overline{w}} f\|_\nu.
		\end{align}
		Using the same method as \eqref{2. 59},  we have
		$
		\langle \mathcal{I}_2(g),   {\overline{w} }^2 f \rangle\lesssim \left( \kappa^2+\delta_0 \right)\|{{\overline{w} } }g\|_{\nu}\| {\overline{w} }f\|_{\nu}.  	
		$
		From $c\mathcal{M}\leq \mu$,   $\sqrt{\mu}-\sqrt{\mathcal{M}}=(\mu^{1/4}+\mathcal{M}^{1/4})(\mu^{1/4}-\mathcal{M}^{1/4})
		\le 4\mu^{1/4}\big|\mu^{1/4}-\mathcal{M}^{1/4}\big|$.
		Using \eqref{Gf},     taking  $\eta \in (0,  \frac{1}{32{\theta_M}})$,  and by \eqref{mu14},  we have
		\begin{align}\label{0i3}
			\langle \mathcal{I}_3(g),    {\overline{w} }^2f \rangle
			\lesssim&\left\|e^{\eta|v|^2}{\overline{w} }  \frac{\mu-\mathcal{M}}{\sqrt{\mu}}\right\|_{\infty}\| {\overline{w} } g\|_{\nu}\| {\overline{w} }f\|_{\nu} \lesssim \left\|\mu^{1/4}-\mathcal{M}^{1/4} \right\|_{\infty}\| {\overline{w} }g\|_{\nu}\| {\overline{w} }f\|_{\nu}
			\lesssim \delta_0 \| {\overline{w} }g\|_{\nu}\| {\overline{w} }f\|_{\nu}.
		\end{align}
		By symmetry,   $ \langle \mathcal{I}_4(g),    f \rangle$   admits the same upper bound as $ \langle \sum_{i=1}^{3}\mathcal{I}_i(g),    f \rangle$.
		Combining above estimates of $\mathcal{I}_1(g)\sim\mathcal{I}_4(g)$ and using Young's inequality,   we prove \eqref{lg}.
		\par
		Taking the velocity derivative of each term $\mathcal{I}_i(g)$ and applying the same high-low velocity decomposition and weighted estimates as employed for $\mathcal{I}_1(g)$ and $\mathcal{I}_3(g)$,
		we obtain  \eqref{lgv}.
	\end{proof}
	In the $W^{1,  \infty}_{x,v}$ estimates,   the dissipation is insufficient for $-3<\gamma<1$. To solve this problem, we define
	$
	\frac{\widetilde{\nu} }{\varepsilon}:=\frac{\nu}{\e}-	 \nabla _{x}\phi ^{\varepsilon }\cdot \frac{w\nabla_v \left( \frac{\sqrt{\mathcal{M}}}{w}  \right)}{\sqrt{\mathcal{M}}}  -\frac{\partial_t w}{w}.
	$ The properties of $\widetilde{\nu}$ are as follows:
	\begin{lemma}\label{nuinter}
		Let $-3<\gamma\leq 1$,     suppose  $\sigma \geq \frac{1}{3}$,   $0\leq T_0 \ll 1 $ be a small  fixed number and  $0\le s \le t\leq   T_0$,     it holds that
		\begin{align}\label{muall}
			\frac{1}{\e}\widetilde{\nu}(v)
			\gtrsim\;&\frac{1}{\e}{\nu}(v)+\frac{\langle v\rangle^2}{(1+t)^{1+\sigma}},
		\end{align}
		which implies  $ \widetilde{\nu}\gtrsim \nu$.
		
		Let $a=   { {\e^ \frac{2}{ \gamma-2}}}$.	If $0\leq\gamma\leq 1$,  denote $		
		\varrho =\frac{\sigma \gamma+2}{2-\gamma} >1$,  for  $\sigma \geq \frac{1}{3}$,
		we obtain
		\begin{equation}\label{nuh}
			\int_0^{t}\exp\{-\frac{1}{\varepsilon}\int_s^t \widetilde{\nu} (\tau) d\tau
			\}	\frac{1}{\varepsilon} \widetilde{\nu}(s) e^{-as^{\varrho}}ds\lesssim  e^{{-a t^\varrho   }},   	
		\end{equation}
		and
		\begin{align}\label{dnuh}
			\int_0^t
			\exp\Big\{-\frac{1}{\varepsilon}\int_s^t \widetilde{\nu}(\tau)\,  d\tau \Big\}
			\left| \nabla \left( \frac{\widetilde{\nu}}{\varepsilon} \right) \right|   e^{-as^{\varrho}}\,   ds\lesssim  e^{{-a t^\varrho   }}.
		\end{align}
		In particular,   if $-3<\gamma<0$,   denote $		
		\varrho =\frac{\sigma \gamma+2}{2-\gamma} \in (0,   1)$,  for   $\sigma\in(\frac{1}{3},  \frac{2}{3})$,
		it holds that
		\begin{equation}\label{nus}
			\int_0^{t}\exp\{-\frac{1}{\varepsilon}\int_s^t \widetilde{\nu}(\tau) d\tau
			\}	\frac{1}{\varepsilon} \widetilde{\nu}(s) e^{-as^{\varrho}}ds\lesssim  e^{{-a t^\varrho   }},   	
		\end{equation}
		and
		\begin{align}\label{dnus}
			\int_0^t
			\exp\Big\{-\frac{1}{\varepsilon}\int_s^t \widetilde{\nu}(\tau)\,  d\tau \Big\}
			\big| \nabla \left( \frac{\widetilde{\nu}}{\varepsilon} \right) \big|   e^{-as^{\varrho}}\,   ds\lesssim  e^{{-a t^\varrho   }}.
		\end{align}
		
	\end{lemma}
	
	\begin{proof}
		With  \eqref{nunum}, we have
		\begin{align}\label{nusimnum}
			\nu_M
			\sim \int_{\mathbb{R}^3} |u-v|^{\gamma}\mathcal{M}(u)\d u
			\lesssim \int_{\mathbb{R}^3} |u-v|^{\gamma}\mu(u)\d u
			\sim \nu
			\lesssim \int_{\mathbb{R}^3} |u-v|^{\gamma}\mathcal{M}^{\varpi}(u)\d u
			\sim \langle v\rangle^{\gamma}
			\sim \nu_M.
		\end{align}
		\par
		Using the  Young's  inequality   and \eqref{nusimnum},  we obtain
		\begin{align}\label{2.29.1}
			|v\cdot \nabla_x\phi^\e|\leq \langle v \rangle \|\nabla_x\phi^\e\|_{ \infty }
			\leq  \Big[\frac{1}{\e }\nu \Big]^{\frac{1}{2-\gamma}}
			\Big[\e^\frac{1}{1-\gamma}\langle v \rangle ^2 \|\nabla_x\phi^\e\|_{ {\infty}}^{\frac{2-\gamma}{1-\gamma}}\Big]^{\frac{1-\gamma}{2-\gamma}}.
		\end{align}
		For $-3<\gamma<1$ and  the constant  $0<\eta\ll 1,  \sigma>1/3$,  with  \eqref{2.29.1} and  \eqref{phi},   we have
		\begin{align}\label{eq:nu:1-1}
			\frac{1}{\e }\widetilde{\nu}(v):=&  \frac{\nu}{\varepsilon}
			+\nabla_x\phi^\varepsilon\cdot v\Big(\frac{1}{2{\theta }_M}+2\widetilde\vartheta(t)\Big)
			+\vartheta\sigma\frac{|v|^2}{(1+t)^{1+\sigma}}\\\nonumber
			\gtrsim&\frac{\nu}{\varepsilon}
			- \Big(\frac{1}{2{\theta }_M}+2\widetilde\vartheta(t)\Big)\left(\eta \frac{1}{\e  }\nu+\e^\frac{1}{1-\gamma}\langle v \rangle ^2 \|\nabla_x\phi^\e\|_{ {\infty}}^{\frac{2-\gamma}{1-\gamma}}\right)
			+\vartheta\sigma\frac{|v|^2}{(1+t)^{ 1+\sigma}} \\\nonumber
			\gtrsim&\frac{1}{\e}\left( 1-\eta\right)\nu+\Big[\vartheta\sigma-\Big(\frac{1}{2{\theta }_M}+2\widetilde\vartheta(t)\Big)\e^{\frac{1}{4}}(1+t)^{(1+\sigma-\frac{16}{15} \times\frac{5}{4})} \Big]
			\frac{\langle v\rangle^2}{(1+t)^{1+\sigma}}\\\nonumber
			\gtrsim&\frac{1}{2\e} \nu+
			\frac{\langle v\rangle^2}{(1+t)^{1+\sigma}}.
		\end{align}
		If $\gamma=1$,  by \eqref{phi} and \eqref{nusimnum},  we have
		$
		|v\cdot \nabla_x\phi^\e|\leq \langle v \rangle \|\nabla_x\phi^\e\|_{ {\infty}}\lesssim\delta_0 (1+t)^{-16/15}\langle v \rangle  <\frac{1}{2\e}\nu,  $ then
		\begin{align*}
			\frac{1}{\e }\widetilde{\nu}(v) \gtrsim \frac{\nu}{\varepsilon}-\frac{1}{2\e}\nu+\vartheta\sigma\frac{|v|^2}{(1+t)^{1+\sigma}}	\gtrsim \frac{1}{2\e} \nu+
			\frac{\langle v\rangle^2}{(1+t)^{1+\sigma}}.
		\end{align*}	
		Thus  \eqref{muall} holds for $-3<\gamma \leq 1$.
		\par
		Using  \eqref{muall},  \eqref{nusimnum} and the Young's inequality,    we have
		\begin{align}\label{vqw}
			\frac{1}{\e }\widetilde{\nu}
			\gtrsim\Big( \frac{1}{\e }	\langle v\rangle^{\gamma}\Big)^{\frac{2}{2-\gamma}}
			\Big[\frac{\langle v\rangle^2}{(1+t)^{1+\sigma}}\Big]^{\frac{-\gamma}{2-\gamma}}
			=\;& {\e^{-\frac{2}{2-\gamma}}}(1+t)^{\frac{(1+\sigma)\gamma}{2-\gamma}}= a(1+t)^{ \varrho-1 },
		\end{align}
		where $a=   { {\e^ \frac{2}{ \gamma-2}}}$ and  $\varrho:=\frac{\sigma \gamma+2}{2-\gamma}$.
		If 	$0\leq\gamma\leq 1$,   $\varrho >1$.  If  $-3<\gamma<0$,  let $\sigma<2/3$ to obtain $\varrho> 0$,  we have $\varrho\in (0,   1)$ from $\sigma \in [ 1/3,   2/3)$. By first-order Taylor expansion,   we have $(1+t)^\varrho - (1+s)^\varrho = \varrho (t-s) + {O}((t-s)^2) \le C_\varrho (t-s)$. Then,  with \eqref{vqw},  it holds that
		\begin{align}\label{nuexp}
			\exp\left\{-\frac{1}{2\varepsilon}\int_s^t \widetilde{\nu} (\tau) d\tau
			\right\}	\lesssim e^{-a\left( (1+t)^{\varrho}- (1+s)^{\varrho}\right)}\leq  C_{\varrho} e^{-at^{\varrho}+as^{\varrho}} \in(0,   1).
		\end{align}
		Applying  \eqref{nuexp},   we have
		\begin{align*}	 	 	
			&\int_0^{t}\exp\left\{-\frac{1}{\varepsilon}\int_s^t \widetilde{\nu} (\tau) \d \tau
			\right	\}	\frac{1}{\varepsilon} \widetilde{\nu}(s) e^{-as^{\varrho}}\d s
			\lesssim\int_0^{t} \exp\left\{-\frac{1}{2\varepsilon}\int_s^t \widetilde{\nu} (\tau) d\tau
			\right\}	\frac{1}{\varepsilon} \widetilde{\nu}(s) e^{-a\left( t^{\varrho}- s^{\varrho}\right)}e^{-as^{\varrho}}\d s\lesssim e^{{-a t^\varrho   }},
		\end{align*}
		which proves \eqref{nuh} and \eqref{nus}.
		\par	 	
		Moreover,   taking derivatives of   \eqref{nuw}
		with respect to
		x and
		v,   using \eqref{nusimnum} and \eqref{muall}
		to obtain $  \nabla_x\nu \lesssim \langle v \rangle^{\gamma}\lesssim \widetilde{\nu},  \,|  \nabla_v\nu| \lesssim   \widetilde{\nu},  $
		with  \eqref{phi} and $\sigma\geq1/3>1/15$,  we have
		\begin{align}\label{tnu}
			&	\left|\nabla \left(\frac{\widetilde{\nu}}{\e}(v) \right) \right|\lesssim \frac{\widetilde{\nu}}{\e}+\|\nabla_x^2 \phi^\e\|_{ \infty} 	\left| \Big(\frac{1}{2{\theta }_M}+2\widetilde{\vartheta}(t)\Big)v\right|+\|\nabla_x  \phi^\e\|_{ \infty} 	\left| \Big(\frac{1}{2{\theta }_M}+2\widetilde{\vartheta}(t)\Big) \right|+	\left|\frac{2\vartheta\sigma}{(1+t)^{1+\sigma }} v\right| \\\nonumber
			\lesssim&\frac{\widetilde{\nu}}{\e}+\frac{\langle v \rangle }{(1+t)^{16/15 }} +\frac{1 }{(1+t)^{16/15 }} +\frac{2\vartheta\sigma \langle v \rangle }{(1+t)^{1+\sigma }}
			\lesssim \frac{\widetilde{\nu}}{\e}+\frac{\langle v \rangle^2}{(1+t)^{1+\sigma }}.
		\end{align}	
		By   \eqref{muall},  we have $\frac{\langle   v \rangle^2 }{(1+t)^{1+\sigma }}\lesssim \frac{\widetilde{\nu}}{\e}$. Then,  \eqref{tnu}  together with   \eqref{nuh} or \eqref{nus}  yields \eqref{dnuh} or \eqref{dnus} respectively.

	\end{proof}	
	
	\section{\texorpdfstring{Weighted $H^{1}_{x,v}$ Estimates}{Weighted H1 Estimates} for the Remainder with Soft Potentials}
	In this section we derive the weighted $H^{1}_{x,   v}$ energy estimates of  $F_R^\varepsilon$  for $-3 < \gamma<0$. A main difficulty lies in handling   the linearized operator $L(x,   v)$. Another challenge is to control the   growth in   $ |v|$.
	We begin with the analysis for
	$L(x,   v)$.
	\begin{lemma}\label{pl2}
		Suppose $-3<\gamma <0$,
		there exists a constant $\delta_1\in(0,   1)$ such that
		\begin{align}\label{0l}
			&\frac{1}{\e}	\langle L f^\e,  {\theta }_0f^\e \rangle \geq  \frac{\delta_1 {\theta }_M }{\e}\Vert  \mathbb{P}^{\bot}f^{\varepsilon} \Vert_{\nu}^2,
		\end{align}
		\begin{align}\label{xl}
			&\frac{1}{\e}\langle \nabla_x (Lf^\e), \, \e \theta _0\nabla_x f^\e\rangle	
			\geq \frac{ \delta_1{\theta_M}}{2}\|  \nabla_x(\mathbb{P}^{\bot} f^\e)	\|_{\nu}-C\{(1+t)^{-16/15}\|f^\e\|^2 +\delta_0\|  \mathbb{P}^{\bot}f^\e\|_{\nu}^2+\delta_0  \|\langle v\rangle^2\mathbb{P}^{\bot} f^\e\|_{\nu}^2 \},
		\end{align}
		\begin{align}\label{vl}
			&\frac{1}{\e}\langle \nabla_v (L\mathbb{P}^{\bot} f^\e),\,  \e  \nabla_v(\mathbb{P}^{\bot} f^\e)\rangle	
			\geq  \frac{\delta_1}{2}\|   \nabla_v (\mathbb{P}^{\bot} f^\e)\|_{\nu}^2-C\left(C_\eta+\delta_0+\kappa^2 \right)\| \mathbb{P}^{\bot} f^\e\|_{\nu}^2-C\left(\delta_0+\kappa^2\right)\| \langle v \rangle\mathbb{P}^{\bot} f^\e\|_{\nu}^2,
		\end{align}
		\begin{align}\label{wxl1}
			\frac{1}{\e}	\langle \nabla_x(L f^\e),\,   \e^{5/2} \langle v \rangle ^{-2\gamma}\nabla_x  f^\e\rangle \geq& \frac{\delta_1 \e^{3/2}}{2}\|\langle v \rangle ^{-\gamma}\nabla_x  f^\e\|_{\nu}^2- C_{\eta} \sqrt{\e}\|\sqrt{\e} f^\e\|_{H^1}^2
			\\\nonumber&-C\e\|\langle v \rangle^{2-2\gamma} \mathbb{P}^{\bot}f^\e\|_{\nu}^2-C\e\|  \nabla_x (\mathbb{P}^{\bot} f^\e)\|_{\nu}^2,
		\end{align}
		and
		\begin{align}\label{wl}
			\frac{1}{\e}\langle L \mathbb{P}^{\bot} f,\,   {\e} \langle v\rangle^{4-4\gamma}\mathbb{P}^{\bot} f \rangle	\geq \frac{\delta_1}{2}	\|\langle v\rangle^{2-2\gamma}\mathbb{P}^{\bot} f\|_{\nu}- C_{\eta}	\| \mathbb{P}^{\bot} f\|_{\nu}.
		\end{align}
		
	\end{lemma}
	\begin{proof}
		It is well known that there exists a constant $\delta_1 >0$ such that \eqref{0l}.
		\par
		Decompose $\nabla_x f^\e=\nabla_x (\mathbb{P}f^\e)+ \nabla_x (\mathbb{P}^{\bot} f^\e)$,
		$\nabla_x(\mathbb{P}f^\e)=\mathbb{P}(\nabla_xf^\e)+[\nabla_x,  \mathbb{P}]f^\e$,  and   $$[\nabla_x,   \mathbb{P}]f^\e :=  \sum_{i=0}^{4}\left(\langle f^\e,  \nabla_x\chi_i\rangle\chi_i+ \langle f^\e,  \chi_i\rangle\nabla_x\chi_i\right).$$
		Since $\langle \theta _0 L( \nabla_x (\mathbb{P}^{\bot} f^\e)),  \mathbb{P}(\nabla_x  f^\e)\rangle=0$,  it holds that
		\begin{align*}
			\frac{1}{\e}\langle \nabla_x (Lf^\e), \, \e \theta _0\nabla_x f^\e\rangle
			=&\langle L( \nabla_x (\mathbb{P}^{\bot} f^\e)), \, \theta _0 \nabla_x (\mathbb{P}^{\bot} f^\e)\rangle+\langle L( \nabla_x (\mathbb{P}^{\bot} f^\e)), \, \theta _0[\nabla_x,   \mathbb{P}]f^\e \rangle\\
			&+\langle [\nabla_x,   L]\mathbb{P}^{\bot} f^\e, \, \theta _0 \nabla_x (\mathbb{P}^{\bot} f^\e)\rangle+\langle [\nabla_x,   L]\mathbb{P}^{\bot} f^\e, \, \theta _0\nabla_x(\mathbb{P} f^\e)\rangle,
		\end{align*}
		where
		\begin{align}\label{nxl}
			[\nabla_x,   L] (\mathbb{P}^{\bot} f^\e)
			=& -\left(\sqrt{\mu}\nabla_x\!\left(\frac{1}{\sqrt{\mu}}\right)\right)\left[\Gamma  \left(\sqrt{\mu},    \mathbb{P}^{\bot} f^\e \right)+ \Gamma  \left(  \mathbb{P}^{\bot} f^\e,   \sqrt{\mu} \right)\right]-\Gamma  \left(\sqrt{\mu},   \frac{(\nabla_x \sqrt{\mu})}{\sqrt{\mu}}\,  \mathbb{P}^{\bot} f^\e \right)
			\\\nonumber
			& -\Gamma  \left( \frac{(\nabla_x \sqrt{\mu})}{\sqrt{\mu}}\,   \mathbb{P}^{\bot} f^\e,   \sqrt{\mu} \right) -\Gamma  \left(    \frac{\nabla_x \mu}{\sqrt{\mu}},  \mathbb{P}^{\bot} f^\e\right)-\Gamma  \left(   \mathbb{P}^{\bot} f^\e,   \frac{\nabla_x \mu}{\sqrt{\mu}} \right).
		\end{align}
		Rewrite  $	\nabla_x\big(\mathbb{P}^{\bot}f^\e\big)$ as
		\begin{align*}
			\nabla_x\big(\mathbb{P}^{\bot}f^\e\big)
			=\nabla_x f^\e-\nabla_x(\mathbb{P} f^\e)=\nabla_x f^\e-\mathbb{P}(\nabla_x f^\e)-[\nabla_x,  \mathbb{P}]f^\e=\mathbb{P}^{\bot} \nabla_xf^\e-[\nabla_x,  \mathbb{P}]f^\e.
		\end{align*}
		By \eqref{0l} and  $ \mathbb{P}\mathbb{P}^{\bot}  \nabla_xf^\e  =0$,  it holds that
		\begin{align*}
			\left\langle L( \nabla_x (\mathbb{P}^{\bot} f^\e)),  \theta _0 \nabla_x (\mathbb{P}^{\bot} f^\e)\right\rangle\geq\delta_1{\theta_M}\|\mathbb{P}^{\bot} \left( \nabla_x(\mathbb{P}^{\bot} f^\e)\right)	\|_{\nu}
			= \delta_1{\theta_M}\|  \nabla_x(\mathbb{P}^{\bot} f^\e)	\|_{\nu}+
			\delta_1{\theta_M}\|   \mathbb{P}( [\nabla_x,  \mathbb{P}]f^\e)	\|_{\nu}.
		\end{align*}
		Then,   we obtain
		\begin{align}\label{xln}
			\frac{1}{\e}\langle \nabla_x (Lf^\e),  \e \theta _0\nabla_x f^\e\rangle	 \geq  \delta_1{\theta_M}\|  \nabla_x(\mathbb{P}^{\bot} f^\e)	\|_{\nu}-\sum_{i=1}^{4}
			\mathcal{N}_i,
		\end{align}
		where
		\begin{align*}
			\mathcal{N}_1= -\delta_1{\theta_M}\|   P ([\nabla_x,  \mathbb{P}]f^\e)	\|_{\nu},  \quad\mathcal{N}_2=\sum_{i=1}^{2}\mathcal{N}_2^{i},  \quad\mathcal{N}_3=\sum_{i=1}^{4}\mathcal{N}_3^{i},  \quad\mathcal{N}_4=\sum_{i=1}^{4}\mathcal{N}_4^{i},
		\end{align*}
		\begin{align*}
			\mathcal{N}_2^1=&\left\langle	 \Gamma(\sqrt{\mu},  \nabla_x(\mathbb{P}^{\bot} f^\e)),  \,  \theta _0[\nabla_x,   \mathbb{P}]f^\e \right\rangle,   \quad
			\mathcal{N}_2^2=\left\langle	 \Gamma(\nabla_x(\mathbb{P}^{\bot} f^\e),  \sqrt{\mu}),  \,  \theta _0[\nabla_x,   \mathbb{P}]f^\e\right\rangle,  \\	
			\mathcal{N}_3^1=&\left\langle  \sqrt{\mu}\nabla_x\!\left(\frac{1}{\sqrt{\mu}}\right)  \Gamma  \left(\sqrt{\mu},    \mathbb{P}^{\bot} f^\e \right)
			, \, \theta _0 \nabla_x (\mathbb{P}^{\bot} f^\e)\right\rangle,  \\
			\mathcal{N}_3^2=&\left\langle \Gamma  \left(\sqrt{\mu},   \frac{(\nabla_x \sqrt{\mu})}{\sqrt{\mu}}\,   \mathbb{P}^{\bot} f^\e \right)	
			,  \theta _0 \nabla_x (\mathbb{P}^{\bot} f^\e)\right\rangle,   \quad
			\mathcal{N}_3^3= \left\langle   \Gamma  \left(    \frac{\nabla_x \mu}{\sqrt{\mu}},  \mathbb{P}^{\bot} f^\e \right)
			,\,  \theta _0 \nabla_x (\mathbb{P}^{\bot} f^\e)\right\rangle,  \\
			\mathcal{N}_3^4
			= &\left\langle
			\sqrt{\mu}\,  \nabla_x\!\left(\frac{1}{\sqrt{\mu}}\right)
			\Gamma\big( \mathbb{P}^{\bot}f^\varepsilon,  \sqrt{\mu}\big)
			+ \Gamma\!\left(\frac{\nabla_x\sqrt{\mu}}{\sqrt{\mu}}\,   \mathbb{P}^{\bot}f^\varepsilon,  \sqrt{\mu}\right)
			+ \Gamma\!\left( \mathbb{P}^{\bot}f^\varepsilon,  \frac{\nabla_x\mu}{\sqrt{\mu}}\right),  \;
			\theta _0 \nabla_x (\mathbb{P}^{\bot} f^\e)
			\right\rangle,  \\
			\mathcal{N}_4^1=&\left\langle  \sqrt{\mu}\nabla_x\!\left(\frac{1}{\sqrt{\mu}}\right)  \Gamma  \left(\sqrt{\mu},    \mathbb{P}^{\bot} f^\e \right)
			, \, \theta _0\nabla_x(\mathbb{P} f^\e) \right\rangle, \;\mathcal{N}_4^2=\left\langle  \sqrt{\mu}\nabla_x\!\left(\frac{1}{\sqrt{\mu}}\right)  \Gamma  \left(    \mathbb{P}^{\bot} f^\e, \sqrt{\mu} \right)
			, \, \theta _0\nabla_x(\mathbb{P} f^\e) \right\rangle \\
			\mathcal{N}_4^3=&\left\langle \Gamma  \left(\sqrt{\mu},   \frac{(\nabla_x \sqrt{\mu})}{\sqrt{\mu}}\,   \mathbb{P}^{\bot} f^\e \right)	+\Gamma  \left(   \frac{(\nabla_x \sqrt{\mu})}{\sqrt{\mu}}\,   \mathbb{P}^{\bot} f^\e,\sqrt{\mu} \right)	
			,\,  \theta _0[\nabla_x,\mathbb{P} ]f^\e\right\rangle,  \\
			\mathcal{N}_4^4=&\left\langle \Gamma  \left(    \frac{\nabla_x \mu}{\sqrt{\mu}},  \mathbb{P}^{\bot} f^\e \right)+\Gamma  \left(    \frac{\nabla_x \mu}{\sqrt{\mu}},  \mathbb{P}^{\bot} f^\e \right)
			, \, \theta _0[\nabla_x,\mathbb{P} ]f^\e\right\rangle.
		\end{align*}
		Next,   we estimate $\mathcal{N}_i$ term by term.
		Using the definition of $\mathbb{P}f$ in \eqref{F_1},  and \eqref{decay},    we have
		\begin{align}
			\mathcal{N}_1
			\lesssim& \delta_1{\theta_M}	\left\|\sum_{j=0}^{4}\left\langle\sum_{i=0}^{4}\left(\langle f^\e,  \nabla_x\chi_i\rangle\chi_i+\langle f^\e, \, \chi_i\rangle\nabla_x\chi_i\right),  \chi_j  \right\rangle\chi_j  \langle v \rangle^{\gamma/2}\right\|
			\lesssim(1+t)^{-16/15}\|f^\e\|^2.
		\end{align}
		Using   \eqref{Gf}     and \eqref{decay},  we have
		\begin{align*}
			\mathcal{N}_2^1
			\lesssim&\|e^{\eta|v|^2}\sqrt{\mu}\|_{\infty}\|\nabla_x(\mathbb{P}^{\bot} f^\e)\|_{\nu}\left\| \theta _0[\nabla_x,   \mathbb{P}]f^\e \langle v \rangle^{\gamma/2}\right\| 			\lesssim  (1+t)^{-16/15}\|f\|^2+\delta_0 \|\nabla_x(\mathbb{P}^{\bot} f^\e)\|_{\nu}^2.
		\end{align*}
		By symmetry,     $	\mathcal{N}_2^2  $ admits the same upper bound as $	\mathcal{N}_2^1  $,   thus,
		\begin{align}
			\mathcal{N}_2
			\lesssim (1+t)^{-16/15}\|f\|^2+\delta_0 \|\nabla_x(\mathbb{P}^{\bot} f^\e)\|_{\nu}^2.
		\end{align}
		Using \eqref{gw} and \eqref{decay},  we have
		\begin{align}\label{n31}
			\mathcal{N}_3^1	\lesssim&\|\nabla_x(\rho_0,u_0)\|_{\infty}\left|\left\langle    \Gamma  \left(\sqrt{\mu},    \mathbb{P}^{\bot} f^\e \right)
			,   \langle v\rangle^2 \nabla_x (\mathbb{P}^{\bot} f^\e)\right\rangle	\right|\\\nonumber
			\lesssim&\delta_0(1+t)^{-16/15}\|e^{\eta|v|^2}\langle v\rangle^2\sqrt{\mu}\|_{\infty}\|\langle v\rangle^2\mathbb{P}^{\bot} f^\e\|_{\nu}\|\nabla_x(\mathbb{P}^{\bot} f^\e)\|_{\nu}\\\nonumber
			\lesssim&\delta_0  \|\langle v\rangle^2\mathbb{P}^{\bot} f^\e\|_{\nu}^2+\delta_0  \|\nabla_x(\mathbb{P}^{\bot} f^\e)\|_{\nu}^2,
		\end{align}
		and
		\begin{align*}
			\mathcal{N}_3^2
			\lesssim&\|\nabla_x(\rho_0,u_0)\|_{\infty}\left|\left\langle \Gamma  \left(\sqrt{\mu},   \langle v \rangle^2 \mathbb{P}^{\bot} f^\e \right)	
			,   \theta _0 \nabla_x (\mathbb{P}^{\bot} f^\e)\right\rangle\right|\\\nonumber
			\lesssim&\delta_0(1+t)^{-16/15}\|e^{\eta|v|^2}\langle v \rangle^2\sqrt{\mu}\|_{\infty}\|\langle v\rangle^2\mathbb{P}^{\bot} f^\e\|_{\nu}\|\nabla_x(\mathbb{P}^{\bot} f^\e)\|_{\nu}\\\nonumber
			\lesssim&\delta_0  \|\langle v\rangle^2\mathbb{P}^{\bot} f^\e\|_{\nu}^2+\delta_0  \|\nabla_x(\mathbb{P}^{\bot} f^\e)\|_{\nu}^2.
		\end{align*}
		Apply the same argument as for $	\mathcal{N}_3^2 $,  we have   $	\mathcal{N}_3^3 \lesssim \delta_0  \| \mathbb{P}^{\bot} f^\e\|_{\nu}^2+\delta_0  \|\nabla_x(\mathbb{P}^{\bot} f^\e)\|_{\nu}^2. $
		By symmetry,    $	\mathcal{N}_3^4  $ admits the same upper bound as $\sum_{i=1}^{3}	\mathcal{N}_3^i  $.   Thus,   we have
		\begin{align}
			\mathcal{N}_3
			\lesssim\delta_0  \| \mathbb{P}^{\bot} f^\e\|_{\nu}^2+\delta_0  \|\langle v\rangle^2\mathbb{P}^{\bot} f^\e\|_{\nu}^2+\delta_0  \|\nabla_x(\mathbb{P}^{\bot} f^\e)\|_{\nu}^2.
		\end{align}
		Using integration by parts,   we have
		\begin{align*}
			\mathcal{N}_4^1
			=&-\left\langle \nabla_x\left[\theta_0 \sqrt{\mu}\nabla_x\!\left(\frac{1}{\sqrt{\mu}}\right)  \Gamma  \left(\sqrt{\mu},    \mathbb{P}^{\bot} f^\e \right)\right]
			,    \mathbb{P} f^\e \right\rangle= {\mathcal{N}_{4,   1}^1}+{\mathcal{N}_{4,   2}^1},
		\end{align*}
		where ${\mathcal{N}_{4,   1}^1}=-\left\langle \theta _0 \sqrt{\mu}\nabla_x\!\left(\frac{1}{\sqrt{\mu}}\right)  \Gamma  \left(\sqrt{\mu},     \nabla_x(\mathbb{P}^{\bot} f^\e) \right)
		,   \mathbb{P} f^\e  \right\rangle$,  and ${\mathcal{N}_{4,   2}^1}$ represents the remaining terms.
		Applying  the definition of $\mathbb{P} f^\varepsilon$ and \eqref{Gf},  we obtain
		\begin{align}\label{n41}
			\left|{\mathcal{N}_{4,   1}^1}\right|
			\lesssim&\|\nabla_x(\rho_0,u_0)\|_{\infty}	\left|	\Big\langle
			\Gamma  \left(\sqrt{\mu},    \nabla_x(\mathbb{P}^{\bot} f^\e) \right),  \langle v \rangle^2  \mathbb{P} f^\e
			\Big\rangle\right|\\\nonumber
			\lesssim&\delta_0(1+t)^{-16/15}\|e^{\eta|v|^2} \sqrt{\mu}\|_{\infty}\left\|  \nabla_x(\mathbb{P}^{\bot} f^\e)  \right\|_{\nu}\left\|\langle v\rangle^2  \mathbb{P} f^\e \right\|_{\nu}\\\nonumber
			\lesssim&(1+t)^{-16/15}\|f^\e\|^2+\delta_0\| \nabla_x (\mathbb{P}^{\bot} f^\e)\|_{\nu}^2.
		\end{align}
		Using a similar   method as $\mathcal{N}_{4,   1}^1$,   we have
		$
		{\mathcal{N}_{4,   2}^1} \lesssim (1+t)^{-16/15}\|f^\e\|^2+\delta_0\|  \mathbb{P}^{\bot}f^\e \|_{\nu}^2.
		$ By symmetry, $	{\mathcal{N}_4^2}$ admit the same bound as  $	{\mathcal{N}_4^1}$.
		Thus,
		\begin{align*}
			{\mathcal{N}_4^1}+{\mathcal{N}_4^2} \lesssim&(1+t)^{-16/15}\|f^\e\|^2+\delta_0\|  \mathbb{P}^{\bot}f^\e \|_{\nu}^2+\delta_0\| \nabla_x (\mathbb{P}^{\bot} f^\e) \|_{\nu}^2.
		\end{align*}
	Applying \eqref{Gf}, we have
	\begin{align*}
|	\mathcal{N}_4^3+	\mathcal{N}_4^4|	\lesssim & \|\nabla_x(\rho_0,u_0)\|_{\infty} \left(\|e^{\eta|v|^2}\sqrt{\mu}\|_{\infty}+\|e^{\eta|v|^2}\langle v \rangle^2\sqrt{\mu}\|_{\infty}\right)\left( \|\langle v \rangle^2\mathbb{P}^{\bot} f^\e\|_{\nu}+\| \mathbb{P}^{\bot} f^\e\|_{\nu}\right)\|[ \nabla_x,\mathbb{P} ] f^\e\|_{\nu}\nonumber\\
\lesssim&(1+t)^{-16/15}\|f^\e\|^2+\delta_0\| \mathbb{P}^{\bot} f^\e\|_{\nu}^2+\delta_0\| \langle v \rangle^2\mathbb{P}^{\bot} f^\e\|_{\nu}^2.
	\end{align*}
	 Then,
	\begin{align}
		\mathcal{N}_4
		\lesssim(1+t)^{-16/15}\|f^\e\|^2+\delta_0\|  \mathbb{P}^{\bot}f^\e\|_{\nu}^2+\delta_0\| \nabla_x (\mathbb{P}^{\bot} f^\e) \|_{\nu}^2.
	\end{align}
 		Substituting the estimates for $\mathcal{N}_1 \sim\mathcal{N}_4$ into \eqref{xln} and with $\delta_0 \ll 1$,   we obtain \eqref{xl}.
		\par
		Applying \eqref{LLM}, Lemma \ref{lm},  and \eqref{lgv},
		there exists a constant $\delta_{1}\in(0,   1)$ such that
		\begin{align*}
			&\frac{1}{\e}\langle \nabla_v (L\mathbb{P}^{\bot} f^\e),  \e   \nabla_v (\mathbb{P}^{\bot} f^\e)\rangle\\
			=&\left\langle \nabla_v (L_M\mathbb{P}^{\bot} f^\e),      \nabla_v (\mathbb{P}^{\bot} f^\e)\right\rangle-\left\langle \nabla_v \left(\mathcal{I} \left(\mathbb{P}^{\bot} f^\e\right)\right),      \nabla_v (\mathbb{P}^{\bot} f^\e)\right\rangle\\
			\geq& \|   \nabla_v (\mathbb{P}^{\bot} f^\e)\|_{\nu}^2-\eta \|   \nabla_v (\mathbb{P}^{\bot} f^\e)\|_{\nu}^2-C_{\eta} \|    \mathbb{P}^{\bot} f^\e\|_{\nu}^2 -C\left(\delta_0+\kappa^2\right)\left(\| \mathbb{P}^{\bot} f^\e\|_{\nu}^2 +\| \langle v \rangle\mathbb{P}^{\bot} f^\e\|_{\nu}^2 +\| \nabla_v(\mathbb{P}^{\bot} f^\e)\|_{\nu}^2 \right)  \\
			\geq&\frac{\delta_1}{2}\|   \nabla_v (\mathbb{P}^{\bot} f^\e)\|_{\nu}^2-C(C_{\eta}+\delta_0+\kappa^2)\| \mathbb{P}^{\bot} f^\e\|_{\nu}^2-C\left(\delta_0+\kappa^2\right)\| \langle v \rangle\mathbb{P}^{\bot} f^\e\|_{\nu}^2,
		\end{align*}
		which yields \eqref{vl}.
		\par
		By \eqref{LLM},  Lemma \ref {lm},  \eqref{lg} and $C_{\eta}\e^{3/2}\|\nabla_x   f^\e\|_{\nu}^2 \lesssim C_{\eta}\e^{3/2}\| f^\e\|_{H^1}^2+C_{\eta}\e^{3/2}\|\nabla_x  \mathbb{P}^{\bot}f^\e\|_{\nu}^2$,   we have
		\begin{align}\label{wxl}
			&	\frac{1}{\e}	\langle \nabla_x(L f^\e),   \e^{5/2} \langle v \rangle ^{-2\gamma}\nabla_x  f^\e\rangle =\e^{3/2}\langle L(\nabla_x f^\e),     \langle v \rangle ^{-2\gamma}\nabla_x  f^\e\rangle+\e^{3/2}\langle [\nabla_x,   L]  f^\e,     \langle v \rangle ^{-2\gamma}\nabla_x  f^\e\rangle\nonumber\\
			=&\e^{3/2}\langle L_M(\nabla_x f^\e),     \langle v \rangle ^{-2\gamma}\nabla_x f^\e\rangle-\e^{3/2}\left\langle  \mathcal{I} (\nabla_x f^\e),     \langle v \rangle ^{-2\gamma}\nabla_x f^\e\right\rangle+\e^{3/2}\langle [\nabla_x,   L]  f^\e,     \langle v \rangle ^{-2\gamma}\nabla_x  f^\e\rangle\\\nonumber
			\geq&\e^{3/2}\| \langle v \rangle ^{- \gamma}\nabla_x  f^\e\|_{\nu}^2 -C_{\eta}\e^{3/2}\|  f^\e\|_{H^1}^2-C_{\eta}\e^{3/2}\|\nabla_x  \mathbb{P}^{\bot}f^\e\|_{\nu}^2 -C(\eta+ \kappa^2+\delta_0)\e^{3/2}\|\langle v \rangle ^{- \gamma}\nabla_x  f^\e\|_{\nu}^2-\sum_{i=1}^{4}\mathcal{N}_{5}^i,
		\end{align}
		where,   similarly to \eqref{nxl},
		\begin{align*}
			\mathcal{N}_5^1=&\e^{3/2}\left\langle  \sqrt{\mu}\nabla_x\!\left(\frac{1}{\sqrt{\mu}}\right)  \Gamma  \left(\sqrt{\mu},     f^\e \right)
			,  \langle v \rangle ^{-2 \gamma}\nabla_x  f^\e\right\rangle,  \\
			\mathcal{N}_5^2=&\e^{3/2}\left\langle \Gamma  \left(\sqrt{\mu},   \frac{(\nabla_x \sqrt{\mu})}{\sqrt{\mu}}\,    f^\e \right)	
			,  \langle v \rangle ^{-2 \gamma}\nabla_x   f^\e\right\rangle,    \quad
			\mathcal{N}_5^3= \e^{3/2}\left\langle   \Gamma  \left(    \frac{\nabla_x \mu}{\sqrt{\mu}},    f^\e \right)
			,  \langle v \rangle ^{-2 \gamma}\nabla_x   f^\e\right\rangle,  \\
			\mathcal{N}_5^4
			= &\e^{3/2}\left\langle
			\sqrt{\mu}\,  \nabla_x\!\left(\frac{1}{\sqrt{\mu}}\right)
			\Gamma\big(  f^\varepsilon,  \sqrt{\mu}\big)
			+ \Gamma\!\left(\frac{\nabla_x\sqrt{\mu}}{\sqrt{\mu}}\,    f^\varepsilon,  \sqrt{\mu}\right)
			+ \Gamma\!\left(  f^\varepsilon,  \frac{\nabla_x\mu}{\sqrt{\mu}}\right)\;,  \;
			\langle v \rangle ^{-2 \gamma}\nabla_x   f^\e
			\right\rangle.
		\end{align*}
		Using the similar method as $	\mathcal{N}_3^1$ in \eqref{n31},  we have
		\begin{align*}
			\sum_{i=1}^{3}	 \mathcal{N}_5^i\lesssim&
			\e^{3/2}\delta_0(1+t)^{-16/15}\left(\|e^{\eta|v|^2}\langle v \rangle ^{2-2 \gamma}\sqrt{\mu}\|_{\infty}\|\langle v \rangle ^{2-2 \gamma}  f^\e\|_{\nu}\|
			+\|e^{\eta|v|^2}\langle v \rangle ^{ -2 \gamma}\sqrt{\mu}\|_{\infty}\|\langle v \rangle ^{ 2-2 \gamma}  f^\e\|_{\nu}\right.\\
			&\left.+\|e^{\eta|v|^2}\langle v \rangle ^{2 -2 \gamma}\sqrt{\mu}\|_{\infty}\|\langle v \rangle ^{ -2 \gamma}  f^\e\|_{\nu}
			\right)\nabla_x  f^\e\|_{\nu}	\\
			\lesssim&\sqrt{\e}\|\sqrt{\e}f^\e\|_{H^1}^2+\e\|\langle v \rangle ^{2-2 \gamma}  \mathbb{P}^{\bot}f^\e\|_{\nu}^2+\e \|\nabla_x\mathbb{P}^{\bot}f^\e\|_{\nu}^2.
		\end{align*}
		By symmetry,    $	\mathcal{N}_5^4  $ admits the same upper bound as $\sum_{i=1}^{3}	\mathcal{N}_5^i  $.   Thus,   we have
		\begin{align}\label{n5}
			\mathcal{N}_5
			\lesssim \sqrt{\e}\|\sqrt{\e}f^\e\|_{H^1}^2+\e\|\langle v \rangle ^{2-2 \gamma}  \mathbb{P}^{\bot}f^\e\|_{\nu}^2+\e \|\nabla_x\mathbb{P}^{\bot}f^\e\|_{\nu}^2.
		\end{align}
		Combining \eqref{n5},  \eqref{wxl} and choosing $\eta+\kappa^2+\delta_0$ small enough,   we prove \eqref{wxl1}.
		\par
		By \eqref{LLM} and \eqref{lg},  we have
		\begin{align*}
			&\frac{1}{\e}\langle L \mathbb{P}^{\bot} f^\e,   {\e}\langle v\rangle^{4-4\gamma}\mathbb{P}^{\bot} f^\e \rangle\\
			= &	\langle L_M \mathbb{P}^{\bot} f^\e,  \langle v\rangle^{4-4\gamma}\mathbb{P}^{\bot} f ^\e \rangle	  -	\left\langle   \mathcal{I}  \left(\mathbb{P}^{\bot} f^\e\right),  \langle v\rangle^{4-4\gamma}\mathbb{P}^{\bot} f^\e \right\rangle\\
			\geq& 	\|\langle v\rangle^{2-2\gamma}\mathbb{P}^{\bot} f^\e\|_{\nu}-\eta\|\langle v\rangle^{2-2\gamma}\mathbb{P}^{\bot} f^\e\|_{\nu}^2 - C_\eta	\| \mathbb{P}^{\bot} f^\e\|_{\nu}^2-\left( \kappa^2+\delta_0 \right)\|\langle v\rangle^{2-2\gamma}\mathbb{P}^{\bot} f^\e\|_{\nu}^2.
		\end{align*}
		Choosing  $\eta,  \kappa$ and $\delta_0$ is small enough,   we prove \eqref{wl}
	\end{proof}
	Next, we perform an energy estimate for the Poisson equation \(\eqref{F_R}_2\)   to eliminate the terms
	$
	\int \Bigl( \int v \sqrt{\mu}\, f^\varepsilon \,\mathrm{d}v \Bigr)\cdot \nabla_x \phi_R^\varepsilon \,\mathrm{d}x \:
	${and}$\:-\varepsilon \int \Bigl( \int v \sqrt{\mu}\, f^\varepsilon \,\mathrm{d}v \Bigr)\cdot \nabla_x^3 \phi_R^\varepsilon \,\mathrm{d}x,
	$
	which will appear later in the \(L^2\)-estimate of \(f^\varepsilon\) in \eqref{e1} and the \(L^2\)-estimate of \(\nabla_x f^\varepsilon\) in \eqref{e20}.
	\begin{lemma}\label{pv}
	It holds that
		\begin{align}\label{r1}
			&\frac{1}{2} \frac{\d}{\d t} \| \nabla \phi_R^\varepsilon \|_{ }^2+\frac{1}{2} \frac{\d}{\d t} \| \sqrt{e^{\phi_0}  e^{\varepsilon^k \phi_R^\varepsilon} } \phi_R^\varepsilon \|_{ }^2 	 - \int \Bigl( \int v \sqrt{\mu}\, f^\varepsilon \,\mathrm{d}v \Bigr)\cdot \nabla_x \phi_R^\varepsilon \,\mathrm{d}x=R_1,
		\end{align}
		and
		\begin{align}\label{r2}
			\varepsilon \int \Bigl( \int v \sqrt{\mu}\, f^\varepsilon \,\mathrm{d}v \Bigr)\cdot \nabla_x^3 \phi_R^\varepsilon \,\mathrm{d}x=R_2.
		\end{align}
		Moreover,  	suppose  $0
		\leq t\leq {\varepsilon^{- 1/2}} $,   we have
		\begin{equation}\label{r}
			|R_i|\lesssim \left( (1+t)^{-16/15}+\sqrt{\e}\right)\|f^\varepsilon\|_{ }^2
			+ \sqrt{\e}\|\sqrt{\e}\nabla_xf^\varepsilon\|_{ }^2
			+\e, \;i=1,   2.
		\end{equation}
	\end{lemma}
	
	\begin{proof}
		Taking  \( \partial_t \) and the \(L^2\)-inner product with $ \phi_R^\varepsilon$ to $\eqref{F_R}_2$,  there is
		\begin{equation} \label{1eq:dt-energy-inner-product}
			\begin{aligned}
				- \left\langle \partial_t \Delta \phi_R^\varepsilon,     \phi_R^\varepsilon \right\rangle
				&=- \left\langle e^{\phi_0}  e^{\varepsilon^k \phi_R^\varepsilon}  \partial_t \phi_R^\varepsilon,    \phi_R^\varepsilon \right\rangle - \left\langle \partial_{t}\phi_0 e^{\phi_0}  \frac{e^{\varepsilon^k \phi_R^\varepsilon}-1 }{\e^k},      \phi_R^\varepsilon \right\rangle- \left\langle \partial_t G,     \phi_R^\varepsilon \right\rangle + \left\langle \int_{\mathbb{R}^3} \partial_t F_R^\varepsilon \,   \d v,     \phi_R^\varepsilon \right\rangle.
			\end{aligned}
		\end{equation}
		By integration by parts,  $	- \left\langle \partial_t \Delta \phi_R^\varepsilon,     \phi_R^\varepsilon \right\rangle=\frac{1}{2} \frac{\d}{\d t} \| \nabla \phi_R^\varepsilon \|_{ }^2$. 	Rewrite
		\begin{align*}
			&- \left\langle e^{\phi_0}  e^{\varepsilon^k \phi_R^\varepsilon}  \partial_t \phi_R^\varepsilon,    \phi_R^\varepsilon \right\rangle\\\nonumber
			=&- \frac{1}{2}\frac{\d}{\d t}\left\langle e^{\phi_0}  e^{\varepsilon^k \phi_R^\varepsilon},    |\phi_R^\varepsilon|^2 \right\rangle+ \frac{1}{2}\left\langle \partial_t \phi_0 e^{\phi_0}  e^{\varepsilon^k \phi_R^\varepsilon},      |\phi_R^\varepsilon|^2 \right\rangle+ \frac{1}{2}\left\langle e^{\phi_0} \varepsilon^k\partial_t  \phi_R^\varepsilon e^{\varepsilon^k \phi_R^\varepsilon},        |\phi_R^\varepsilon|^2 \right\rangle.
		\end{align*}
		Applying integration by parts,   using    $\eqref{F_R}_1$ and \eqref{f},  we have
		$$-\left\langle \int_{\mathbb{R}^3} \partial_t F_R^\varepsilon \,   \d v,   \phi_R^\varepsilon \right\rangle=\left\langle \int v \cdot\nabla_x (\sqrt{\mu} f^\varepsilon)   \,   \d v,   \phi_R^\varepsilon \right\rangle=-	\int \Bigl( \int v \sqrt{\mu}\, f^\varepsilon \,\mathrm{d}v \Bigr)\cdot \nabla_x \phi_R^\varepsilon \,\mathrm{d}x.
		$$
		Then \eqref{1eq:dt-energy-inner-product} can be rewritten as,
		\begin{align}\label{epir1}
			&\frac{1}{2} \frac{\d}{\d t} \| \nabla \phi_R^\varepsilon \|_{ }^2+\frac{1}{2} \frac{\d}{\d t} \| \sqrt{e^{\phi_0}  e^{\varepsilon^k \phi_R^\varepsilon} } \phi_R^\varepsilon \|_{ }^2
			-\iint v   \sqrt{\mu} f^\varepsilon  \cdot \nabla_x \phi_R^\varepsilon \,   \d v \d x\\ \nonumber
			=&\frac{1}{2}\left\langle \partial_t \phi_0 e^{\phi_0}  e^{\varepsilon^k \phi_R^\varepsilon},      |\phi_R^\varepsilon|^2 \right\rangle+ \frac{1}{2}\left\langle e^{\phi_0} \varepsilon^k\partial_t  \phi_R^\varepsilon e^{\varepsilon^k \phi_R^\varepsilon},        |\phi_R^\varepsilon|^2 \right\rangle   - \left\langle \partial_{t}\phi_0 e^{\phi_0}  \frac{e^{\varepsilon^k \phi_R^\varepsilon}-1 }{\e^k},      \phi_R^\varepsilon \right\rangle- \left\langle \partial_{t}{G},     \phi_R^\varepsilon \right\rangle  \\ \nonumber
			=&R_1.
		\end{align}
		Using  a  $\mathit{priori}$  assumption \eqref{pri},  \eqref{phif},  \eqref{gt} and \eqref{phiftt},  we have
		\begin{align}\label{r1detail}
			|R_1|\lesssim& \|\partial_{t} \phi_0\|_{\infty}		 \|\phi_R^\varepsilon\|_{ }^2+\varepsilon^k\|\phi_R^\varepsilon\|_{ H^2}	 \|\phi_R^\varepsilon\|_{ } \|\partial_{t}\phi_R^\varepsilon\|_{ }  + \|\partial_{t}G\|_{ } \|\phi_R^\varepsilon\|_{ }\nonumber\\
			\lesssim& \delta_0 (1+t)^{-16/15} \|\phi_R^\varepsilon\|_{ }^2+  \e	 \|\phi_R^\varepsilon\|_{ } \|\partial_{t}\phi_R^\varepsilon\|_{ }   +(\varepsilon +\varepsilon \|\phi_R^\varepsilon\|_{ }+\varepsilon \|\partial_{t}\phi_R^\varepsilon\|_{ })\|\phi_R^\varepsilon\|_{ }\nonumber\\
			\lesssim& \left( (1+t)^{-16/15}+\sqrt{\e}\right)\|f^\varepsilon\|_{ }^2
			+\e^{3/2}\|\nabla_xf^\varepsilon\|_{ }^2
			+\e.
		\end{align}
		\par
		Taking the time derivative \( \partial_t \) and the \(L^2\)-inner product with $\e \nabla^2 \phi_R^\varepsilon$ to $\eqref{F_R}_2$.	Applying integration by parts with  $\eqref{F_R}_1$ and \eqref{f},  we have
		$
		\e	\left\langle \int_{\mathbb{R}^3} \partial_t F_R^\varepsilon \,   \d v,    \nabla_x ^2 \phi_R^\varepsilon \right\rangle =\varepsilon \int \Bigl( \int v \sqrt{\mu}\, f^\varepsilon \,\mathrm{d}v \Bigr)\cdot \nabla_x^3 \phi_R^\varepsilon \,\mathrm{d}x
		.
		$ Then,   we obtain
		\begin{align}
			\varepsilon \int \Bigl( \int v \sqrt{\mu}\, f^\varepsilon \,\mathrm{d}v \Bigr)\cdot \nabla_x^3 \phi_R^\varepsilon \,\mathrm{d}x=R_2,
		\end{align}
		where
		\begin{align*}
			R_2=- \e\left\langle \partial_t \Delta \phi_R^\varepsilon,     \nabla^2 \phi_R^\varepsilon \right\rangle+\e\left\langle e^{\phi_0}  e^{\varepsilon^k \phi_R^\varepsilon}  \partial_t \phi_R^\varepsilon,    \nabla^2 \phi_R^\varepsilon \right\rangle +\e \left\langle \partial_{t}\phi_0 e^{\phi_0}  \frac{e^{\varepsilon^k \phi_R^\varepsilon}-1 }{\e^k},      \nabla^2 \phi_R^\varepsilon \right\rangle+ \e\left\langle \partial_t G,     \nabla^2 \phi_R^\varepsilon \right\rangle.
		\end{align*}
		Applying a  $\mathit{priori}$  assumption \eqref{pri},  \eqref{gt},     \eqref{phif},  and \eqref{phiftt}
		,   we have
		\begin{align}
			|R_2|\lesssim&\e \|\partial_{t}\phi_R^\e\|_{H^2}\|\phi_R^\e\|_{H^2}
			+\e\|\partial_{t} \phi_0\|_{\infty}	 \| \phi_R^\varepsilon\|  \|\nabla_x^2\phi_R^\varepsilon\|_{ } +\e \|\partial_{t} G\|_{ } \|\nabla^2\phi_R^\varepsilon\|_{ }\nonumber\\
			\lesssim& \left( (1+t)^{-16/15}+\sqrt{\e}\right)\|f^\varepsilon\|_{ }^2
			+\sqrt{\e} \|\sqrt{\e}\nabla_xf^\varepsilon\|_{ }^2
			+\e.
		\end{align}
	\end{proof}
	The following is the weighted $H^{1}_{x, v}$ energy estimates for the remainder terms in the soft potentials case.
	\begin{proposition}\label{L2}
		Let $-3<\gamma <0$,   suppose $0\leq t \leq \e^{-1/2} $,   it holds that
		\begin{align}\label{L2a}
			&\frac{\d}{\d t}\left\{ \| \sqrt{{\theta }_0} f^{\varepsilon} \|^2   +\|\sqrt{\e{\theta }_0} \nabla_x f^\varepsilon\|^2 + \| \sqrt{\e}  \nabla_v\left(\mathbb{P}^{\bot}f^\e\right)\|_{ }^2+\| \sqrt{\e}  \langle v\rangle^{2-2\gamma}\mathbb{P}^{\bot}f^\e\|_{ }^2+\| \e^{5/4} \sqrt{\theta_0} \langle v\rangle^{-\gamma}\nabla_x f^\e)\|_{ }^2  \right. \nonumber\\
			&\qquad    \left.  \| \nabla \phi_R^\varepsilon \|_{}^2 +\| \sqrt{e^{\phi_0}  e^{\varepsilon^k \phi_R^\varepsilon}} \phi_R^\varepsilon \|_{}^2   \right\}\\\nonumber
			& +\frac{\delta_1 {\theta }_M}{\varepsilon}   \Vert \ \mathbb{P}^{\bot} f^{\varepsilon} \Vert_{\nu}^2 + { {\delta_1}  {\theta }_M}  \|\nabla_x  \left(\mathbb{P}^{\bot}f^\varepsilon \right) \|_\nu^2+ {\delta_1} \|   \nabla_v(\mathbb{P}^{\bot}f^\e)\|_{ \nu}^2+{\delta_1} \| \langle v\rangle^{2-2\gamma}   \mathbb{P}^{\bot}f^\e\|_{ \nu}^2 +{\delta_1 \e^{3/2}} \| \langle v\rangle^{-\gamma} \nabla_x   f^\e\|_{ \nu}^2   \\	\nonumber
			\lesssim&\varepsilon^4\left\|h^\varepsilon\right\|_{W^{1,  \infty}} \Vert f^{\varepsilon }\Vert_{H^1}    \\\nonumber
			& +	   \left((1+t)^{-16/15}+\sqrt{\e}\right)\left(\Vert f^{\varepsilon }\Vert ^{2} +\Vert\sqrt{\e} \nabla f^{\varepsilon }\Vert ^{2} +\|\sqrt{\e} \langle v \rangle^{2-2 \gamma }  \mathbb{P}^{\bot} f^\e\|^2+\|\e^{5/4}\langle v \rangle^{- \gamma }\nabla_x  f^\e\|^2 \right) +\sqrt{\e}
			.
		\end{align}

	\end{proposition}
	\begin{proof}
		By \eqref{f},    \eqref{F_R} can be rewritten as
		\begin{align}\label{fe}
			& \partial _{t}f^{\varepsilon }+v\cdot \nabla _{x}f^{\varepsilon }  +\frac{v-u_{0}}{{\theta } _{0}}
			\sqrt{\mu }\cdot \nabla _{x}\phi _{R}^{\varepsilon }+\frac{1}{\varepsilon
			}L f^{\varepsilon } \\\nonumber
			=& -\frac{\{\partial _{t}+v\cdot \nabla _{x}-\nabla _{x}\phi _{0}\cdot
				\nabla _{v}\}\sqrt{\mu }}{\sqrt{\mu }}f^{\varepsilon }-\sum_{i=1}^{2k-1}\varepsilon ^{i}\nabla _{x}\phi _{i}\cdot \frac{
				v-u_{0}}{2{\theta } _{0}}f^{\varepsilon }-\varepsilon ^{k}\nabla _{x}\phi _{R}^{\varepsilon
			}\cdot \frac{v-u_{0}}{2{\theta } _{0}}f^{\varepsilon
			}
			\\\nonumber
			& \quad+\varepsilon
			^{k-1}\Gamma (f^{\varepsilon },  f^{\varepsilon
			})+\sum_{i=1}^{2k-1}\varepsilon ^{i-1}\left\{\Gamma \left(\frac{F_{i}}{\sqrt{\mu }}
			,  f^{\varepsilon }\right)+\Gamma \left(f^{\varepsilon },  \frac{F_{i}}{\sqrt{\mu }}\right)\right\}   \\\nonumber
			& \quad+\nabla _{x}\phi  ^{\varepsilon }\cdot \nabla
			_{v}f^{\varepsilon }+\sum_{i=1}^{2k-1}\varepsilon ^{i}\nabla _{x}\phi _{R}^{\varepsilon }\cdot \frac{\nabla
				_{v}F_{i}}{\sqrt{\mu }} +\varepsilon ^{k-1}\overline{A},  	
		\end{align}
		where $\overline{A} = \frac{A}{\sqrt{\mu}}$  and  $\nabla_v \mu = -\frac{v - u_0}{{\theta }_0} \mu$.
		\\
		\emph{\bf Step 1. $L_{x,   v}^2$-estimate of $f^\e$.}\;
		\par
		Taking the $L^2_{x,v}$ inner product of \eqref{fe} with ${\theta}_0 f^{\varepsilon}$,     applying integration by parts, combining   \eqref{0l},   we obtain
		\begin{equation}
			\begin{split}
				& \frac{1}{2} \frac{\d}{\d t} \left\Vert \sqrt{{\theta }_0} f^{\varepsilon} \right\Vert^2
				+ \int \left( \int (v - u_0) \sqrt{\mu} f^{\varepsilon} \,   \d v \right) \cdot \nabla_x \phi_R^{\varepsilon} \,   \d x
				+\frac{\delta_1}{\varepsilon}   {\theta }_M \left\Vert  \mathbb{P}^{\bot}f^{\varepsilon}\right \Vert_{\nu}^2 \leq\sum_{i=1}^{3}\mathcal{S}_i,
			\end{split}
			\label{e1}
		\end{equation}
		where
		\begin{align*}
			\mathcal{S}_1&=\frac{1}{2} \langle (\partial_t + v \cdot \nabla_x){\theta }_0 f^{\varepsilon},   f^{\varepsilon} \rangle
			- \left\langle  \frac{(\partial_t + v \cdot \nabla_x + \nabla_x \phi_0 \cdot \nabla_v) \sqrt{\mu}}{\sqrt{\mu}} f^{\varepsilon},   {\theta }_0f^{\varepsilon} \right\rangle\\
			& \quad
			- \left\langle \sum_{i=1}^{2k-1} \varepsilon^i \nabla_x \phi_i \cdot \frac{v - u_0}{2} f^{\varepsilon},   f^{\varepsilon} \right\rangle- \varepsilon^k \left\langle \nabla_x \phi_R^{\varepsilon} \cdot \frac{v - u_0}{2} f^{\varepsilon},   f^{\varepsilon} \right\rangle,  \\
			\mathcal{S}_2&=\varepsilon^{k-1} \langle  \Gamma(f^{\varepsilon},   f^{\varepsilon}),   {\theta }_0f^{\varepsilon} \rangle
			+ \left\langle \sum_{i=1}^{2k-1} \varepsilon^{i-1} \left\{ \Gamma\left( \frac{F_i}{\sqrt{\mu}},   f^{\varepsilon} \right) + \Gamma\left( f^{\varepsilon},   \frac{F_i}{\sqrt{\mu}} \right) \right\},    {\theta }_0f^{\varepsilon} \right\rangle,  \\
			\mathcal{S}_3&=\left\langle  \sum_{i=1}^{2k-1} \varepsilon^i \nabla_x \phi_R^{\varepsilon} \cdot \frac{\nabla_v F_i}{\sqrt{\mu}},   {\theta }_0f^{\varepsilon} \right\rangle
			+ \varepsilon^{k-1} \langle  \overline{A},  {\theta }_0 f^{\varepsilon} \rangle.
		\end{align*}
		Adding \eqref{e1} and \eqref{r1} yields
		\begin{align}\label{es}
			&\frac{1}{2} \frac{\d}{\d t} \left\{
			\left\| \sqrt{{\theta }_0} f^{\varepsilon}\right \|^2  +  \| \nabla \phi_R^\varepsilon \| ^2 +\left\| \sqrt{e^{\phi_0}  e^{\varepsilon^k \phi_R^\varepsilon}} \phi_R^\varepsilon \right\| ^2 \right\}  +
			\frac{\delta_1}{\varepsilon}   {\theta }_M\left\|  \mathbb{P}^{\bot}f^{\varepsilon} \right\|_\nu^2
			\\\nonumber
			\leq&R_1
			+ \int u_0 \left( \int \sqrt{\mu} f^{\varepsilon} \,   \d v \right) \cdot \nabla_x \phi_R^{\varepsilon} \,   \d x +\sum_{i=1}^{3}\mathcal{S}_i.
		\end{align}
		Using \eqref{decay} and \eqref{phif},   we obtain
		\begin{align}\label{s01}
			\int u_0 \left( \int \sqrt{\mu} f^{\varepsilon} \,   \d v \right) \cdot \nabla_x \phi_R^{\varepsilon} \,   \d x\lesssim& \|u_0\|_{\infty}	\|f^{\varepsilon}\|_{ }\|\nabla_x\phi_R^{\varepsilon} \|_{}\lesssim(1+t)^{-\frac{16}{15}}\|f^{\varepsilon}\|_{}^2+\e\|f^{\varepsilon}\|.
		\end{align}
		By \eqref{h},     the a  $\mathit{priori}$  assumption $  \varepsilon^k\| h^\varepsilon\|_{\infty}\le \e $ in \eqref{pri},   \eqref{phif} and Young's inequality,   we have
		\begin{align}\label{vg3}
			&\varepsilon^k \left\langle    \nabla_x \phi_R^\varepsilon \cdot \frac{v - u_0}{2{\theta }_0} f^\varepsilon,   {\theta }_0 f^\varepsilon \right\rangle
			\lesssim  \varepsilon^k\|  \langle v \rangle f^\varepsilon\|_{\infty} 	\|\nabla_x \phi_R^\varepsilon\| 	\| f^\e\|
			\lesssim  \varepsilon^k\|  h^\varepsilon\|_{\infty} (	\| f^\e\|+\e)	\| f^\e\|   \lesssim \e \| f^\e\|^2+\e.
		\end{align}
		Using   \eqref{decay},   \eqref{fir2},  the Sobolev embedding $ H^2(\mathbb{R}^3)\hookrightarrow L^\infty(\mathbb{R}^3)$,    \eqref{h} and combining \eqref{vg3},  we have
		\begin{align}\label{vg11}
			\mathcal{S}_1
			\lesssim\;&
			\Big(\|\nabla(\rho_0,   u_0,
			\theta _0)\|
			+\sum_{i=1}^{2k-1}\varepsilon^i\|\nabla_x\phi_i\|\Big)
			\left\|\langle v\rangle^3 f^\varepsilon
			\mathbf 1_{\{\langle v\rangle^{6-\gamma}\ge \kappa^2/\varepsilon\}}\right\|_\infty
			\left\|f^\varepsilon\right\| \\\nonumber
			&+
			\Big(\|\nabla(\rho_0,   u_0,  \theta _0)\|_\infty
			+\sum_{i=1}^{2k-1}\varepsilon^i\|\nabla_x\phi_i\|_\infty\Big)
			\left\|\langle v\rangle^3 f^\varepsilon
			\mathbf 1_{\{\langle v\rangle^{6-\gamma}\le \kappa^2/\varepsilon\}}\right\|
			\left\|f^\varepsilon\right\| +\e \| f^\e\|^2+\e \\\nonumber
			\lesssim\;&\left(\delta_0 +\e\right)
			\left\|\langle v\rangle^{3}f^\varepsilon
			\mathbf 1_{\{\langle v\rangle^{6-\gamma}\ge \kappa^2/\varepsilon\}}\right\|_\infty
			\left\|f^\varepsilon\right\| \\\nonumber
			&+
			\Big((1+t)^{-16/15}+ \e\Big)
			\left\|\Big(\left\|	\langle v\rangle^3\mathbb{P} f^\varepsilon\right\|+\|\langle v\rangle^{3-\gamma/2} \mathbb{P}^{\bot} f^\varepsilon\mathbf 1_{\{\langle v\rangle^{6-\gamma}\le \kappa^2/\varepsilon\}}\|_{\nu}
			\Big)
			\right\|
			\left\|f^\varepsilon\right\|+\e \| f^\e\|^2+\e \\\nonumber
			\lesssim\;&
			\left\|\langle v \rangle^{4\gamma-24}\langle v \rangle^{27-4\gamma}e^{-\frac{\widetilde{\vartheta}|v|^2}{2}}h^\e	\mathbf 1_{\{\langle v\rangle^{6-\gamma}\ge \kappa^2/\varepsilon\}}\right\|_{\infty}\left\|f^\varepsilon\right\|\\\nonumber
			&
			+ (1+t)^{-16/15} \left\|f^\varepsilon\right\|^2  +
			\left\|\langle v\rangle^{3-\gamma/2}\mathbb{P}^{\bot} f^\varepsilon
			\mathbf 1_{\{\langle v\rangle^{6-\gamma}\le \kappa^2/\varepsilon\}}\right\|_\nu^2  +\e\\\nonumber
			\lesssim\;&
			\varepsilon^4\left\|h^\varepsilon\right\|_\infty\left\|f^\varepsilon\right\|
			+ (1+t)^{-16/15} \left\|f^\varepsilon\right\|^2  +\frac{\kappa^2}{\e}
			\left\| \mathbb{P}^{\bot} f^\varepsilon
			\right  \|_\nu^2+\e.
		\end{align}		
		Applying \eqref{Gf},   \eqref{h},  \eqref{fir3},  a  $\mathit{priori}$  assumption in \eqref{pri},   and Young's inequality,   we have
		\begin{align}\label{s2}
			\mathcal{S}_2
			\lesssim& \e^{k-1}\|e^{\eta |v|^2} f^\e\|_{\infty}\|f^\e\|_{\nu}\|\mathbb{P}^{\bot} f^{\varepsilon}\|_{\nu}+\sum_{i=1}^{2k-1} \varepsilon^{i-1} \left\|e^{\eta |v|^2} \frac{F_i}{\sqrt{\mu}}\right\|_{\infty}\|f\|_{\nu}\|\mathbb{P}^{\bot} f^{\varepsilon}\|_{\nu}
			\\\nonumber
			\lesssim&\left(\e^{k-1}\|h^\e\|_{\infty}+\sum_{i=1}^{2k-1}\left(\e(1+t)\right)^{i-1}\right)\left( \e\|f^\e\|^2+\frac{\kappa^2}{\e}\|\mathbb{P}^{\bot} f^{\varepsilon}\|_{\nu}^2\right) \\\nonumber
			\lesssim&  \e\|f^\e\|^2+\frac{\kappa^2}{\e}\|\mathbb{P}^{\bot} f^{\varepsilon}\|_{\nu}^2.
		\end{align}	
		By \eqref{fir3},     Lemma \ref{oa},  \eqref{phif},   and Young's inequality,   we have
		\begin{align}\label{s3}
			\mathcal{S}_3\lesssim&\sum_{i=1}^{2k-1} \varepsilon^i \left\|  \left( \int \frac{|\nabla_v F_i|^2}{\mu} \,   \d v \right)^{1/2} \right\|_{L_x^\infty} \| \nabla \phi_R^{\varepsilon} \|_{L^2_x} \|f^{\varepsilon}\|_{L^2_{x,v}}+\e^{k-1}\|\overline{A}\|_{L^2_{x,v}}\Vert f^{\varepsilon }\Vert_{L^2_{x,v}} \\\nonumber
			\lesssim&   \e\sum_{i=1}^{2k-1} [\varepsilon(1 + t)]^{i - 1}   \| \nabla \phi_R^{\varepsilon} \|    \|f^{\varepsilon}\|  +\varepsilon ^{k-1}\sum_{2k+1\leq i+j\leq 4k-2}{\varepsilon
				^{i+j-2k}(1+t)^{i+j-2}}\,  \|f^{\varepsilon }\|\\\nonumber
			\lesssim&\sqrt{\e}\|f^{\varepsilon }\|	^2 +\sqrt{\e} 	.
		\end{align}
		Substituting \eqref{r}, \eqref{s01} and \eqref{vg11}--\eqref{s3}   into \eqref{e1},    we have
		\begin{align}\label{l2}
			&\frac{1}{2}\frac{\d}{\d t}\left\{
			\left\| \sqrt{{\theta }_0} f^{\varepsilon} \right\|^2  +  \| \nabla \phi_R^\varepsilon \|_{}^2 +\left\| \sqrt{e^{\phi_0}  e^{\varepsilon^k \phi_R^\varepsilon}} \phi_R^\varepsilon\right \|_{}^2 \right\}  +  \frac{{\delta_1}     {\theta }_M}{\e}\Vert  \mathbb{P}^{\bot}f^{\varepsilon }\Vert _{\nu }^{2} \\\nonumber
			\lesssim &  \varepsilon ^{4}\Vert h^{\varepsilon }\Vert _{\infty }\Vert
			f^{\varepsilon }\Vert
			+\left((1+t)^{-16/15}+\sqrt{\e}\right)\Vert f^{\varepsilon }\Vert ^{2} +  \sqrt{\e}\|\sqrt{\e}\nabla_xf^\varepsilon\|_{ }^2 +\frac{\kappa^2}{\e}\|\mathbb{P}^{\bot} f^{\varepsilon}\|_{\nu}^2  +\sqrt{\e}.
		\end{align}
		{\bf \emph{Step 2.} $L_{x,   v}^2$-estimate of $\sqrt{\e}  \nabla_xf^\e$. }
		\par
		Applying $\nabla_x$ to \eqref{fe}, taking the $L^2_{x,   v} $ inner product with $\e{\theta }_0 \nabla_xf^\e$,   using integration by parts,
		and combining \eqref{xl},  we  obtain
		\begin{align}\label{e20}
			&\frac{1}{2}\frac{\d}{\d t}\left\|\sqrt{\e{\theta }_0} \nabla_x f^\varepsilon\right\|^2-\varepsilon \int \Bigl( \int v \sqrt{\mu}\, f^\varepsilon \,\mathrm{d}v \Bigr)\cdot \nabla_x^3 \phi_R^\varepsilon \,\mathrm{d}x+\frac{ {\delta_1}  {\theta }_M}{2}\| \nabla_x( \mathbb{P}^{\bot}f^{\varepsilon}) \|_\nu^2 \\	\nonumber
			\leq&C\left\{(1+t)^{-16/15}\|f^\e\|^2 +\delta_0\|  \mathbb{P}^{\bot}f^\e\|_{\nu}^2+\delta_0  \|\langle v\rangle^2\mathbb{P}^{\bot} f^\e\|_{\nu}^2 \right\}+\sum_{i=4}^{7}  \mathcal{S}_i,
		\end{align}
		where
		\begin{align*}
			\mathcal{S}_4=&	\e \left\langle \nabla_x \left(  {(v - u_0)}  \sqrt{\mu}  \right)\cdot \nabla_x^2 \phi_R^\varepsilon,  \; f^\varepsilon \right\rangle -\e\left\langle \nabla_x \left( \frac{v - u_0}{{\theta }_0} \sqrt{\mu} \right)\cdot \nabla_x \phi_R^\varepsilon,   {\theta }_0 \nabla_x f^\varepsilon \right\rangle,      \\
			\mathcal{S}_5= &\frac{1}{2}  \left\langle (\partial_t +v\cdot \nabla_x) {\theta }_0,  \e   |\nabla_x f^\varepsilon|^2 \right\rangle-\left\langle  \nabla_x \left( \frac{ (\partial_t + v \cdot \nabla_x - \nabla_x \phi_0 \cdot \nabla_v) \sqrt{\mu} }{ \sqrt{\mu} } f^\varepsilon \right),   \e{\theta }_0 \nabla_x f^\varepsilon \right\rangle\\
			&-   \sum_{i=1}^{2k-1} \varepsilon^i \left\langle \nabla_x \left( \nabla_x \phi_i \cdot \frac{v - u_0}{2{\theta }_0} f^\varepsilon \right),   \e{\theta }_0\nabla_x f^\varepsilon \right\rangle	-   \varepsilon^k \left\langle \nabla_x \left( \nabla_x \phi_R^\varepsilon \cdot \frac{v - u_0}{2{\theta }_0} f^\varepsilon \right),   \e{\theta }_0 \nabla_x f^\varepsilon \right\rangle,  \\
			\mathcal{S}_6=&\varepsilon^{k-1}\left\langle  \nabla_x \Gamma(f^\varepsilon,   f^\varepsilon),       \e{\theta }_0 \nabla_x f^\varepsilon \right\rangle +
			\left\langle  \sum_{i=1}^{2k-1} \varepsilon^{i-1}\left\{ \nabla_x \Gamma \left( \frac{F_i}{\sqrt{\mu}},   f^\varepsilon \right)+ \nabla_x \Gamma \left( f^\varepsilon,   \frac{F_i}{\sqrt{\mu}} \right)\right\},   \e{\theta }_0 \nabla_x f^\varepsilon \right\rangle
			,  \\
			\mathcal{S}_7=& \left\langle \nabla _x\left(   \nabla _x \phi ^\varepsilon \cdot \nabla_v f^\varepsilon\right),   \e{\theta }_0\nabla_x f^\varepsilon \right\rangle+ \sum_{i=1}^{2k-1} \varepsilon^i
			\left\langle  \nabla_x\left(\nabla_x \phi_R^\varepsilon \cdot \frac{ \nabla_v F_i }{ \sqrt{\mu} } \right ),  \e{\theta }_0 \nabla_x f^\varepsilon \right\rangle
			+   \varepsilon^{k-1} \left\langle \nabla_x \overline{A},   \e{\theta }_0\nabla_x f^\varepsilon \right\rangle.
		\end{align*}
		Adding \eqref{r2} to \eqref{e20}, we obtain
		\begin{align}\label{e2}
			&\frac{1}{2}\frac{\d}{\d t}\{\|\sqrt{\e{\theta }_0} \nabla_x f^\varepsilon\|^2 + \frac{{\delta_1}  {\theta }_M}{2} \|\nabla_x ( \mathbb{P}^{\bot}f^{\varepsilon}) \|_\nu^2 \\	\nonumber
			\leq &C\{(1+t)^{-16/15}\|f^\e\|^2 +\delta_0\|  \mathbb{P}^{\bot}f^\e\|_{\nu}^2+\delta_0  \|\langle v\rangle^2\mathbb{P}^{\bot} f^\e\|_{\nu}^2 \}+R_2-\varepsilon \int u_0  \Bigl( \int \sqrt{\mu}\, f^\varepsilon \,\mathrm{d}v \Bigr)\cdot \nabla_x^3 \phi_R^\varepsilon \,\mathrm{d}x+\sum_{i=4}^{7}  \mathcal{S}_i.
		\end{align}
		By \eqref{decay} and \eqref{phif},  it holds that
		\begin{align}\label{r2u}
			&\left|-\varepsilon \int u_0  \Bigl( \int \sqrt{\mu}\, f^\varepsilon \,\mathrm{d}v \Bigr)\cdot \nabla_x^3 \phi_R^\varepsilon \,\mathrm{d}x \right|\lesssim
			\e	 \|f^\e\|_{  }\|\nabla_x^3\phi_R^\e\|_{  }
			\lesssim  \sqrt{\e}\|f^\e\|^2+\sqrt{\e}\|\sqrt{\e}\nabla_xf^\e\|^2+\e.
		\end{align}
		Applying    \eqref{decay},  and \eqref{phif},   we have
		\begin{align}\label{s4}
			\mathcal{S}_4	\lesssim &\e\|\nabla_x\phi_R^\e\|_{H^1}\|f^\e\|_{H^1_x}
			\lesssim \sqrt{\e}\|f^\e\|^2+ \sqrt{\e}\|\sqrt{\e} f^\e\|_{H^1_x} ^2+\e.
		\end{align}
		Using the  same method as $	\mathcal{S}_1 $ in \eqref{vg11},    we have	 		
		\begin{align}\label{vg0031}
			\mathcal{S}_5
			\lesssim\;
			\varepsilon^4\left\|h^\varepsilon\right\|_{W^{1,  \infty}}\left\|\nabla_xf^\varepsilon\right\|
			+ (1+t)^{-16/15}  \left\|\sqrt{\e}f^\varepsilon\right\|_{H^1_x}^2+ {\kappa^2}
			\left\| \mathbb{P}^{\bot} f^\varepsilon
			\right\|_{H^1_x(\nu)} ^2+\e.
		\end{align}		 		 		
		$\nabla_x f^\e=\mathbb{P}(\nabla_x f^\e)+[\nabla_x,  \mathbb{P}]f^\e+\nabla_x (\mathbb{P}^{\bot}f^\e)$,  and $\langle \Gamma(f^\e,   f^\e),  \mathbb{P}(\nabla_x f^\e)\rangle=0$,   using  the  similar method as $	\mathcal{S}_2 $ in \eqref{vg11},    we have	 		
		\begin{align}
			\mathcal{S}_6
			\lesssim\; &
			\left(
			\varepsilon^{k}\|h^\varepsilon\|_{W^{1,  \infty}}
			+ \varepsilon \sum_{i=1}^{2k-1} [\varepsilon(1+t)]^{i-1}
			\right)\left( \|f^\e\|_{\nu}+\|\nabla_xf^\e\|_{\nu}\right)\left( \|f^\e\| +\|\nabla_x(\mathbb{P}^{\bot}f^\e)\|_{ \nu}\right)\nonumber\\
			\lesssim\; &\sqrt{\e}\|f^\e\|^2+
			\sqrt{\varepsilon}\,  \|\sqrt{\varepsilon}\,  \nabla_xf^\varepsilon\|^2	+ \sqrt{\e}\|\mathbb{P}^{\bot}f^\e\|_{\nu}^2
			+ \sqrt{\e}\|\nabla_x(\mathbb{P}^{\bot}f^\e)\|_{\nu}^2.
		\end{align}
		By \eqref{phi},  we have
		\begin{align*}
			\left\langle \nabla _x\left(   \nabla _x \phi ^\varepsilon \cdot \nabla_v f^\varepsilon\right),   \e{\theta }_0\nabla_x f^\varepsilon \right\rangle\lesssim\e\|\nabla_x^2 \phi^\e\|_{\infty}\|\nabla_v f^\e \|	\|\nabla_x f^\varepsilon \|\lesssim (1+t)^{-16/15}  \left(\|\sqrt{\e} \nabla_x  f^\e \| ^2+\|\sqrt{\e} \nabla_v  f^\e \| ^2\right).
		\end{align*}	
		The other terms in $\mathcal{S}_7$ are handled in the same way as $\mathcal{S}_3$ in \eqref{s3},  then
		\begin{align}\label{s7}
			\mathcal{S}_7
			\lesssim& (1+t)^{-16/15}  \left(\|\sqrt{\e} \nabla_x  f^\e \| ^2+\|\sqrt{\e} \nabla_v  f^\e \| ^2\right)+    \e.
		\end{align}
		Substituting \eqref{r},  \eqref{r2u}--\eqref{s7},  into \eqref{e20},    we have
		\begin{align} \label{l2x}
			&\frac{1}{2}\frac{\d}{\d t} \|\sqrt{\e{\theta }_0} \nabla_x f^\varepsilon\|^2 +\frac{ { {\delta_1}  {\theta }_M} }{2} \|\nabla_x  \left(\mathbb{P}^{\bot}f^\varepsilon \right) \|_\nu^2 \\	\nonumber
			\lesssim&\varepsilon^4\left\|h^\varepsilon\right\|_{W^{1,  \infty}}\left\|\nabla_xf^\varepsilon\right\| +\left( (1+t)^{-16/15}+\sqrt{\e}\right)\left( \|f^\e\|^2+\|\sqrt{\e} \nabla_x  f^\e \| ^2+\|\sqrt{\e} \nabla_v  f^\e \| ^2\right)\\\nonumber
			&+(\delta_0+\kappa^2+\sqrt{\e})
			\left\| \mathbb{P}^{\bot} f^\varepsilon
			\right\|_{H_x^1(\nu)} ^2
			+ \delta_0  \|\langle v\rangle^2\mathbb{P}^{\bot} f^\e\|_{\nu}^2	  +\e.
		\end{align}
		{\bf \emph{Step 3.} $L_{x,   v}^2$-estimate of $\sqrt{\e}  \nabla_v\left(\mathbb{P}^{\bot}f^\e\right)$.\;}
		Denote  	$\mathcal{A}_{\phi^\e}:= v\cdot \nabla_x -  \nabla_x \phi^\e\cdot \nabla_v+\frac{v-u_0}{2{\theta }_{0}}  \cdot\nabla_x \phi^\e $,
		applying the microscopic projection $\mathbb{P}^{\bot}$ to both sides of \eqref{fe}	and noting that $\mathbb{P}^{\bot}( \frac{v-u_0}{{\theta }_0}\cdot \nabla_x \phi_R^\e\sqrt{\mu})=0$,   then we  obtain
		\begin{align}\label{ipgg}
			&\;\pt_t\mathbb{P}^{\bot}f^\e+\ v\cdot  \nabla_x (\mathbb{P}^{\bot} f^\e)+\frac{1}{\e}L(\mathbb{P}^{\bot}f^\e)\\\nonumber
			=&\; \nabla_x \phi^\e\cdot  \nabla_v (\mathbb{P}^{\bot} f^\e)-\frac{v-u_0}{2{\theta }_{0}} \cdot\nabla_x \phi^\e\mathbb{P}^{\bot} f^\e-\mathbb{P}^{\bot}\left(\frac{\{\partial_t+v\cdot\nabla_x\}\sqrt\mu}{\sqrt\mu} f^\varepsilon
			\right)\\\nonumber
			& \;+\sum_{i=1}^{2k-1}\e^{i}\mathbb{P}^{\bot}\left( \frac{1}{\sqrt{\mu}}\nabla_x\phi_R^\e\cdot\nabla_v F_i\right)+\varepsilon
			^{k-1}\Gamma (f^{\varepsilon },  f^{\varepsilon
			})+\sum_{i=1}^{2k-1}\varepsilon ^{i-1}\left\{\Gamma \left(\frac{F_{i}}{\sqrt{\mu }}
			,  f^{\varepsilon }\right)+\Gamma \left(f^{\varepsilon },  \frac{F_{i}}{\sqrt{\mu }}\right)\right\}\\\nonumber
			&\; +\varepsilon ^{k-1}\mathbb{P}^{\bot} (\overline{A} )+[[\mathbb{P},  \mathcal{A}_{\phi^\e}]]f^\e\nonumber,
		\end{align}
		where $[[\mathbb{P},  \mathcal{A}_{\phi^\e}]]:=\mathbb{P}\mathcal{A}_{\phi^\e}-\mathcal{A}_{\phi^\e}\mathbb{P}$.
		Applying   $  \nabla_v$ to \eqref{ipgg},    taking the $L^2 $-inner product of the differentiated equation with $ \e \nabla_v(\mathbb{P}^{\bot}f^\varepsilon)$,    using integration by parts,   and combining \eqref{vl},
		we deduce
		\begin{align}\label{ev}
			&\; \frac{1}{2}\frac{d }{d t} \| \sqrt{\e}  \nabla_v\left(\mathbb{P}^{\bot}f^\e\right)\|_{ }^2+\langle  \nabla_x (\mathbb{P}^{\bot} f^\e),   \e\nabla_v\left(\mathbb{P}^{\bot}f^\e\right) \rangle+ \frac{{\delta_1}}{2} \|   \nabla_v\left(\mathbb{P}^{\bot}f^\e\right)\|_{ \nu}^2 \\\nonumber
			\le&C\left(C_\eta+\delta_0+\kappa^2 \right)\| \mathbb{P}^{\bot} f^\e\|_{\nu}^2+C\left(\delta_0+\kappa^2\right)\| \langle v \rangle \mathbb{P}^{\bot} f^\e\|_{\nu}^2
			+\sum_{i=8}^{10}  \mathcal{S}_i,
		\end{align}
		where
		\begin{align}
			\mathcal{S}_8&:=-\left\langle \frac{v-u_0}{2{\theta }_0} \cdot \nabla_x \phi^\varepsilon \,     \nabla_v(\mathbb{P}^{\bot}f^\varepsilon) 	+   \nabla_v\left\{\mathbb{P}^{\bot}\left(\frac{\{\partial_t+v\cdot\nabla_x\}\sqrt\mu}{\sqrt\mu} f^\varepsilon\right)\right\},  \;  \e  \nabla_v (\mathbb{P}^{\bot}f^\varepsilon) \right\rangle,  \nonumber \\
			\mathcal{S}_9&:=\Big\langle   \varepsilon^{k-1}   \nabla_v \Gamma(f^\varepsilon,   f^\varepsilon)+ \sum_{i=1}^{2k-1} \varepsilon^{i-1}   \nabla_v \left\{
			\Gamma\!\left( \frac{F_i}{\sqrt{\mu}},   f^\varepsilon \right)
			+ \Gamma\!\left( f^\varepsilon,   \frac{F_i}{\sqrt{\mu}} \right) \right\},     \e \nabla_v\left(\mathbb{P}^{\bot}f^\e\right)\Big \rangle,  \\\nonumber
			\mathcal{S}_{10}&:=\Big\langle  \sum_{i=1}^{2k-1} \varepsilon^{i}   \nabla_v\left\{ \mathbb{P}^{\bot}
			\left( \frac{1}{\sqrt{\mu}} \nabla_x \phi_R^\varepsilon \cdot \nabla_v F_i \right)\right\}+\varepsilon^{k-1}   \nabla_v \{\mathbb{P}^{\bot}(\overline{A}) \} +   \nabla_v \left([[\mathbb{P},  \mathcal{A}_{\phi^\varepsilon}]] f^\varepsilon \right),       \e \nabla_v\left(\mathbb{P}^{\bot}f^\e\right)\Big \rangle.
		\end{align}
		By young's inequality,   we have
		\begin{align}\label{dis1}
			\langle  \nabla_x (\mathbb{P}^{\bot} f^\e),   \e\nabla_v\left(\mathbb{P}^{\bot}f^\e\right) \rangle \leq & \eta  \| \nabla_v\left(\mathbb{P}^{\bot}f^\e\right)\|_{\nu}^2+C_{\eta}\e^2\| \langle v \rangle^{-\gamma} \nabla_x\mathbb{P}^{\bot}f^\e\|_{\nu}^2\\\nonumber
			\leq&\eta  \| \nabla_v\left(\mathbb{P}^{\bot}f^\e\right)\|_{\nu}^2+C_{\eta}\e^2\| \langle v \rangle^{-\gamma} \nabla_x f^\e\|_{\nu}^2+C_{\eta}\e^2\|  f^\e\|_{H^1}^2.
		\end{align}
		Using the same similar method as $	\mathcal{S}_{1}$ in \eqref{vg11},  we obtain
		\begin{align}\label{3. 53}
			&-\Bigg\langle \frac{v-u_0}{2{\theta }_0} \cdot \nabla_x \phi^\varepsilon \,     \nabla_v  f^\varepsilon	
			+   \nabla_v \Big(\frac{(\partial_t+v\cdot\nabla_x)\sqrt\mu}{\sqrt\mu} f^\varepsilon\Big),  \;  \e  \nabla_v (\mathbb{P}^{\bot}f^\varepsilon) \Bigg\rangle \nonumber\\
			\lesssim\;&
			\varepsilon^4 \|h^\varepsilon\|_{W^{1,  \infty}} \|\nabla_v (\mathbb{P}^{\bot}f^\varepsilon)\|
			+ (1+t)^{-16/15}\  \|\sqrt{\e}f^\varepsilon\|_{H^1}^2
			+ \kappa^2 \|  \mathbb{P}^{\bot} f^\varepsilon\|_{H^1(\nu)}^2.
		\end{align}
		By \eqref{decay},  we have
		\begin{equation}\label{pg}
			\left\|v^n  \nabla_v( \mathbb{P}[v^ng])\right\|=\left\|\sum_{i=0}^{4}\langle v^ng,  \chi_i\rangle(\nabla_v\chi_i) v^n\right\|
			\lesssim \|g\| 	,  n>- 4.
		\end{equation}
		Then, applying \eqref{pg},\eqref{phi} and \eqref{decay},  we obtain
		\begin{align}\label{3. 55}
			&\e\left\langle\frac{v-u_0}{2{\theta }_0} \cdot \nabla_x \phi^\varepsilon \,     \nabla_v  (\mathbb{P}f^\varepsilon )+  \nabla_v \left\{\mathbb{P} \left( \frac{\{\partial_t+v\cdot\nabla_x\}\sqrt\mu}{\sqrt\mu} f^\varepsilon\right)\right\}\langle v \rangle^{-\frac{\gamma}{2}},  \; \langle v \rangle^{\frac{\gamma}{2}}   \nabla_v (\mathbb{P}^{\bot}f^\varepsilon) \right\rangle \\\nonumber
			\lesssim&\e \|\nabla_x (\phi^\varepsilon,\rho_0,u_0) \|_{\infty}	 \|\langle v \rangle^{1-\frac{\gamma}{2}}\nabla_v\left(\mathbb{P}f^\e\right)+\langle v \rangle^{-\frac{\gamma}{2}}\nabla_v\mathbb{P}\left(\langle v \rangle^3f^\e\right)\|  \|\nabla_v(\mathbb{P}^{\bot}f^\varepsilon)\|_{\nu}\\\nonumber
			\lesssim&\e\|\sqrt{\e}f^\e\| ^2 + \sqrt{\e} \|\nabla_v(\mathbb{P}^{\bot}f^\varepsilon)\|_{\nu}^2.
		\end{align}
		Combining \eqref{3. 53} and \eqref{3. 55},  we have
		\begin{align}\label{s8}
			\mathcal{S}_8\lesssim	\varepsilon^4 \|h^\varepsilon\|_{W^{1,  \infty}} \|f^\varepsilon\|_{H^1}
			+ ((1+t)^{-16/15}+\sqrt{\e })\  \|\sqrt{\e}f^\varepsilon\|_{H^1}^2
			+ (\kappa^2+\sqrt{\e}) \|  \mathbb{P}^{\bot} f^\varepsilon\|_{H^1(\nu)}^2.
		\end{align}
		By Lemma \ref{gamma} and the a  $\mathit{priori}$  assumption in \eqref{pri},   we have
		\begin{align}\label{s9}
			\mathcal{S}_9 \lesssim &\,   \left(\varepsilon^{k }
			\Big\|
			e^{\eta |v|^2}
			f^\varepsilon
			\Big\|_{\infty}+\e
			\sum_{i=1}^{2k-1} \varepsilon^{i-1}
			\Big\|
			e^{\eta |v|^2}
			\left(\frac{F_i}{\sqrt{\mu}}+\nabla_v\left(\frac{F_i}{\sqrt{\mu}}\right)\right)\
			\Big\|_{\infty}\right)\|f^\e\|_{H^1(\nu)}\big\|
			\nabla_v(\mathbb{P}^{\bot}f^{\varepsilon})
			\big\|_{\nu}  \nonumber\\
			\lesssim &\,
			\sqrt{\e}
			\|\sqrt{\e}f^\e\|_{H^1}^2 +	\sqrt{\e}
			\|\mathbb{P}^{\bot} f^\e\|_{H^1(\nu)}^2
			.
		\end{align}
		Using \eqref{fir3}, Lemma \ref{oa}, \eqref{pg},   and \eqref{phi},  we obtain
		\begin{align}\label{s10}
			\mathcal{S}_{10} \lesssim&\e	\sum_{i=1}^{2k-1}
			\varepsilon^i
			\| \nabla_x\phi_R^\varepsilon\|
			\left(
			\int
			\left|\nabla_v\left\{\mathbb{P}^{\bot}\left(\frac{\nabla_v F_i}{\sqrt{\mu}} \right)\right\}
			\right| ^2
			\d v
			\right)^{1/2}
			\|	\nabla_v\left(\mathbb{P}^{\bot}f^\e\right) \|\nonumber\\
			&+\e^{k }\|  \nabla_v \mathbb{P}^{\bot}(\overline{A})    \|\|  \nabla_v\left(\mathbb{P}^{\bot}f^\e\right)\|+\e\| \nabla_v \left( [[\mathbb{P},  \mathcal{A}_{\phi^\varepsilon}]] f^\varepsilon \right) \langle v \rangle^{-\gamma/2}  \|\|  \nabla_v\left(\mathbb{P}^{\bot}f^\e\right)\|_{\nu}\nonumber\\
			\lesssim&\varepsilon^2\sum_{i=1}^{2k-1} \{\varepsilon  (1+t)\}^{i-1}\|\nabla_x\phi_R^\varepsilon\|\|  \nabla_v\left(\mathbb{P}^{\bot}f^\e\right) \|   +\varepsilon ^{k}\sum_{2k+1\leq i+j\leq 4k-2}{\varepsilon
				^{i+j-2k}(1+t)^{i+j-2}}\|  \nabla_v\left(\mathbb{P}^{\bot}f^\e\right)  \| \nonumber\\
			&+ \e \|f^\e\|_{H^1} \|\nabla_v\left(\mathbb{P}^{\bot}f^\e\right)\|_{\nu}\\\nonumber
			\lesssim&\e^2\left( \|f^\e\|+\e\right)\|\nabla_v\left(\mathbb{P}^{\bot}f^\e\right) \| +\e^{3/2}\|\nabla_v\left(\mathbb{P}^{\bot}f^\e\right) \|+ \e \|f^\e\|_{H^1} \|\nabla_v\left(\mathbb{P}^{\bot}f^\e\right)\|_{\nu} \\\nonumber
			\lesssim& \sqrt{\e}\|\sqrt{\e}f^\e\|_{H^1}^2+\sqrt{\e}\|\nabla_v\left(\mathbb{P}^{\bot}f^\e\right)\|_{\nu}^2+\e.
		\end{align}
		Substituting \eqref{dis1} and \eqref{s8}--\eqref{s10} into \eqref{ev},     we have
		\begin{align}\label{l2v}
			&\;  \frac{1}{2}\frac{d }{d t} \|  \sqrt{\e} \nabla_v\left(\mathbb{P}^{\bot}f^\e\right)\|_{ }^2+ {\delta_1} \|   \nabla_v\left(\mathbb{P}^{\bot}f^\e\right)\|_{ \nu}^2 \\\nonumber
			\lesssim &\varepsilon^4 \|h^\varepsilon\|_{W^{1,  \infty}} \|\nabla_v (\mathbb{P}^{\bot}f^\varepsilon)\| +\left( (1+t)^{-16/15}+\sqrt{\e}\right)\left( \|\sqrt{\e}f^\e\|_{H^1}^2\right)\\\nonumber
			&+C_\eta	\left\| \mathbb{P}^{\bot} f^\varepsilon
			\right\|_{\nu}+( \delta_0+\kappa^2 +\sqrt{\e}+\eta)
			\left\|\mathbb{P}^{\bot} f^\varepsilon
			\right\|_{H^1(\nu)} ^2
			+\left(\delta_0+\kappa^2\right)\| \langle v \rangle \mathbb{P}^{\bot} f^\e\|_{\nu}^2	 +C_{\eta}\e^2\| \langle v \rangle^{-\gamma} \nabla_x f^\e\|_{\nu}^2  +\e.
		\end{align}
		\par	To control  $C_{\eta}\varepsilon^2\|\langle v\rangle^{-\gamma}\nabla_x f^\varepsilon\|_{\nu}^2$ in \eqref{l2v},   we proceed as follows.
        \smallskip
		\par
		\noindent{\bf	\emph Step 4. $L_{x,   v}^2$-estimate of $\varepsilon^{5/4}\langle v\rangle^{-\gamma}\nabla_x f^\varepsilon$.}\;	  Applying $\nabla_x$ to \eqref{fe} and taking the $L^2_{x,   v} $ inner product with $\e^{5/2} \langle v \rangle^{-2\gamma}{\theta }_0 \nabla_xf^\e$ on both sides,   combining		
		\eqref{wxl1},  we  obtain
		\begin{align}\label{wgev}
			&\frac{1}{2}\frac{\d}{\d t}\|\e^{5/4}\sqrt{ {\theta }_0}\langle v \rangle^{-\gamma} \nabla_x f^\varepsilon\|^2 +\frac{\delta_1 \e^{3/2}}{2}\|\langle v \rangle ^{-\gamma}\nabla_x  f^\e\|_{\nu}^2 \\	\nonumber
			\leq&C_{\eta} \sqrt{\e}\|\sqrt{\e} f^\e\|_{H^1}^2
			+C\e\| \nabla_x(\mathbb{P} f^\e)\|_{\nu}^2+C\e\|\langle v \rangle^{2-2\gamma} \mathbb{P}f^\e\|_{\nu}^2+\sum_{i=11}^{14}  \mathcal{S}_i,
		\end{align}
		where
		\begin{align*}
			\mathcal{S}_{11} &= \left\langle \nabla_x \left( \frac{v - u_0}{{\theta}_0} \sqrt{\mu} \cdot \nabla_x \phi_R^\varepsilon \right),   \varepsilon^{5/2} \langle v \rangle^{-2\gamma} \theta_0 \nabla_x f^\varepsilon \right\rangle,   \\[2mm]
			\mathcal{S}_{12} &= \frac{1}{2}  \left\langle (\partial_t + v \cdot \nabla_x) \theta_0, \varepsilon^{5/2}  \langle v \rangle^{-2\gamma} |\nabla_x f^\varepsilon|^2 \right\rangle \\
			&\quad - \left\langle \nabla_x \left( \frac{ (\partial_t + v \cdot \nabla_x - \nabla_x \phi_0 \cdot \nabla_v) \sqrt{\mu} }{\sqrt{\mu}} f^\varepsilon \right),   \varepsilon^{5/2} \langle v \rangle^{-2\gamma} \theta_0 \nabla_x f^\varepsilon \right\rangle \\
			&\quad - \sum_{i=1}^{2k-1} \varepsilon^i \left\langle \nabla_x \left( \nabla_x \phi_i \cdot \frac{v - u_0}{2 \theta_0} f^\varepsilon \right),   \varepsilon^{5/2} \langle v \rangle^{-2\gamma} \theta_0 \nabla_x f^\varepsilon \right\rangle \\
			&\quad - \varepsilon^k \left\langle \nabla_x \left( \nabla_x \phi_R^\varepsilon \cdot \frac{v - u_0}{2 \theta_0} f^\varepsilon \right),   \varepsilon^{5/2} \langle v \rangle^{-2\gamma} \theta_0 \nabla_x f^\varepsilon \right\rangle,   \\[1mm]
			\mathcal{S}_{13} &= \varepsilon^{k-1} \left\langle \nabla_x \Gamma(f^\varepsilon,   f^\varepsilon),   \varepsilon^{5/2} \langle v \rangle^{-2\gamma} \theta_0 \nabla_x f^\varepsilon \right\rangle \\
			&\quad + \left\langle \sum_{i=1}^{2k-1} \varepsilon^{i-1} \left\{ \nabla_x \Gamma\left( \frac{F_i}{\sqrt{\mu}},   f^\varepsilon \right) + \nabla_x \Gamma\left( f^\varepsilon,   \frac{F_i}{\sqrt{\mu}} \right) \right\},   \varepsilon^{5/2} \langle v \rangle^{-2\gamma} \theta_0 \nabla_x f^\varepsilon \right\rangle,   \\[1mm]
			\mathcal{S}_{14} &= \left\langle \nabla_x \left( \nabla_x \phi^\varepsilon \cdot \nabla_v f^\varepsilon \right),   \varepsilon^{5/2} \langle v \rangle^{-2\gamma} \theta_0 \nabla_x f^\varepsilon \right\rangle \\
			&\quad + \sum_{i=1}^{2k-1} \varepsilon^i \left\langle \nabla_x \left( \nabla_x \phi_R^\varepsilon \cdot \frac{\nabla_v F_i}{\sqrt{\mu}} \right)+ \varepsilon^{k-1}\nabla_x \overline{A},   \varepsilon^{5/2} \langle v \rangle^{-2\gamma} \theta_0 \nabla_x f^\varepsilon \right\rangle.
		\end{align*}
		Since the power of $\varepsilon^{5/2}<\varepsilon^{3/2}$   in  $ \mathcal{S}_{11}$,    we first apply H\"older's  inequality in the $v$-variable,   then in the $x$-variable,   and finally invoke \eqref{phi1} to obtain
		\begin{align}\label{s41}
			\mathcal{S}_{11}	  \lesssim \e^{5/2}\|\nabla_x  \phi_R^\varepsilon \|_{H^1}\|\nabla_xf^\e\|\lesssim    \e^{5/2}\left( \|  f^\varepsilon \| +\varepsilon\right)\|f^\e\|_{H^1}
			\lesssim \sqrt{\e}\|\sqrt{\e}f^\e\|_{H^1}^2+\e.
		\end{align}
		Using the same method as $\mathcal{S}_5$ in \eqref{vg031},  $\langle v \rangle^{-2\gamma}$ does not create new difficulties,   we  have
		\begin{align}\label{vg031}
			\mathcal{S}_{12}
			\lesssim
			\varepsilon^4\left\|h^\varepsilon\right\|_{W^{1,  \infty}}\left\|\nabla_xf^\varepsilon\right\|
			+ \sqrt{\e} \left\|\sqrt{\e}f^\varepsilon\right\|_{H^1}^2+ \e^{3/2}
			\left\| \mathbb{P}^{\bot} f^\varepsilon
			\right\|_{H^1(\nu)} ^2+\e.
		\end{align}		 		 		
		Using \eqref{Gf},    \eqref{fir3},  and the a  $\mathit{priori}$  assumption   $\varepsilon^{k } \|h^\varepsilon\|_{W^{1,  \infty}}\leq \e$ in \eqref{pri},  we have	
		\begin{align}\label{s13. 3}
			\mathcal{S}_{13}\lesssim &  \left(\varepsilon^{k+3/2}
			\Big\|
			e^{\eta |v|^2}\langle v \rangle^{-\gamma}
			f^\varepsilon
			\Big\|_{W^{1,  \infty}}+\e^{5/2}
			\sum_{i=1}^{2k-1} \varepsilon^{i-1 }
			\Big\|
			e^{\eta |v|^2}
			\frac{F_i }{\sqrt{\mu}}
			\Big\|_{W^{1,  \infty} }\right)
			\|\langle v \rangle^{-\gamma}f^\varepsilon\|_{H_x^1(\nu)}
			\big\|\langle v \rangle^{-\gamma}
			\nabla_x f^{\varepsilon}
			\big\|_{\nu}    \nonumber\\
			\lesssim&\sqrt{\e}\|\sqrt{\e}f^\e\|^2+\e^{5/2}\|\langle v \rangle^{-\gamma}\mathbb{P}^{\bot}f^\e\|_{\nu}^2+\e^{5/2}\|\langle v \rangle^{-\gamma}\nabla_xf^\e\|_{\nu}^2		.
		\end{align}
		Applying integration by parts and using \eqref{phi},   we obtain
		\begin{align*}
			&\left\langle \nabla_x \left( \nabla_x \phi^\varepsilon \cdot \nabla_v f^\varepsilon \right),
			\varepsilon^{5/2} \langle v \rangle^{-2\gamma} \theta_0 \nabla_x f^\varepsilon \right\rangle \\
			=&-\frac{1}{2}\varepsilon^{5/2}\theta_0(-2\gamma)
			\left\langle \nabla_x \phi^\varepsilon \cdot \langle v \rangle^{-2\gamma-2}v,
			|\nabla_x f^\varepsilon|^2 \right\rangle
			+\left\langle \nabla_x^2 \phi^\varepsilon \cdot \nabla_v f^\varepsilon,
			\varepsilon^{5/2} \langle v \rangle^{-2\gamma} \theta_0 \nabla_x f^\varepsilon \right\rangle \\
			\lesssim\;&
			\varepsilon^{5/2}\|\nabla_x \phi^\varepsilon\|_{\infty}
			\|\langle v \rangle^{-\gamma} f^\varepsilon\|^2
			+\varepsilon^{5/2}
			\|\langle v \rangle^{-2\gamma}\nabla_v f^\varepsilon
			\mathbf{1}_{\{\langle v \rangle^{-5\gamma}\ge 1/\varepsilon\}}\|_{\infty}
			\|\nabla_x^2 \phi^\varepsilon\|\,  \|\nabla_x f^\varepsilon\| \\
			&+	\|\nabla_x^2 \phi^\varepsilon\|_{\infty}\varepsilon^{5/2}\Big(
			\|\langle v \rangle^{-2\gamma}\nabla_v \mathbb{P}f^\varepsilon\|
			+\|\langle v \rangle^{-2\gamma-\gamma/2}
			\nabla_v (\mathbb{P}^{\perp}f^\varepsilon)
			\mathbf{1}_{\{\langle v \rangle^{-5\gamma}\le 1/\varepsilon\}}\|_{\nu}
			\Big)\|\nabla_x f^\varepsilon\| \\
			\lesssim\;&
			(1+t)^{-16/15}\|\varepsilon^{5/4}\langle v \rangle^{-\gamma}f^\varepsilon\|^2
			+\varepsilon^{4}\|h^\varepsilon\|_{W^{1,  \infty}}\|\nabla_x f^\varepsilon\|
			+\sqrt{\varepsilon}\|\sqrt{\varepsilon}f^\varepsilon\|_{H^1}^2
			+\sqrt{\varepsilon}\|\nabla_v(\mathbb{P}^{\perp}f^\varepsilon)\|_{\nu}^2.
		\end{align*}
		The remaining terms in $\mathcal{S}_{14}$ can be treated by same method as  $\mathcal{S}_3$ in \eqref{s3}.
		Invoking Lemma~\ref{oa} to control $\|\nabla_x \overline{A}\langle v \rangle^{-2\gamma}\|$,   we conclude
		\begin{align}\label{s71}
			\mathcal{S}_{14}
			\lesssim&\e^{4}\|h^\e\|_{W^{1,  \infty}}\|\nabla_x f^\varepsilon\| +\sqrt{\e}\|\sqrt{\e}f^\e\|_{H^1}^2+(1+t)^{-16/15} \| \e^{5/4}\langle v \rangle^{-\gamma} f^\varepsilon \|^2+\sqrt{\e}\|\nabla_v(\mathbb{P}^{\bot}f^\varepsilon) \|_{\nu}^2 +\e.
		\end{align}
		Substituting   \eqref{s41}--\eqref{s71} into \eqref{wgev},    we have	
		\begin{align}\label{l2wxip}
			&\frac{1}{2}\frac{\d}{\d t}\|\e^{5/4}\sqrt{ {\theta }_0}\langle v \rangle^{-\gamma} \nabla_x f^\varepsilon\|^2 +\frac{\delta_1 \e^{3/2}}{2}\|\langle v \rangle ^{-\gamma}\nabla_x  f^\e\|_{\nu}^2 \\	\nonumber
			\lesssim&\e^{4}\|h^\e\|_{W^{1,  \infty}}\|\nabla_x f^\varepsilon\|+  \sqrt{\e}\|\sqrt{\e} f^\e\|_{H^1}^2+(1+t)^{-16/15} \| \e^{5/4}\langle v \rangle^{-\gamma} f^\varepsilon \|^2\\\nonumber
			& +\sqrt{\e}\| \mathbb{P}^{\bot}f^\varepsilon \|_{H^1(\nu)}^2+\e\|\langle v \rangle^{2-2\gamma} \mathbb{P}f^\e\|_{\nu}^2+\e^{5/2}\|\langle v \rangle^{-\gamma}\nabla_xf^\e\|_{\nu}^2	+\e.
		\end{align}
		
		Combining the above four steps,   we observe that the remaining dominant dissipation terms which cannot yet be closed is
		$
		\|\langle v\rangle^{2-2\gamma}\,  \mathbb{P}^{\bot} f^\varepsilon\|_{\nu}.
		$
		To control this term,   we proceed to the next estimate.
		
		\medskip
		\noindent{\bf	\emph Step 5. $L_{x,   v}^2$-estimate of $\sqrt{\varepsilon}\,  \langle v\rangle^{2-2\gamma}\mathbb{P}^{\bot} f^\varepsilon$.}\;	
		Taking $L^2$ inner product to \eqref{ipgg} with $\e {\langle v \rangle^{{4-4\gamma}}} \mathbb{P}^{\bot} f^\e $,   $ \langle \ v\cdot  \nabla_x (\mathbb{P}^{\bot} f^\e),  \e {\langle v \rangle^{4-4\gamma}} \mathbb{P}^{\bot} f^\e\rangle=0$ from  integration by parts and using \eqref{wl},   we have	
		\begin{align}\label{wev}
			&\; \frac{1}{2}\frac{d }{d t} \| \sqrt{\e}  \langle v\rangle^{2-2\gamma}\mathbb{P}^{\bot}f^\e\|_{ }^2+	  \frac{{\delta_1}}{2} \| \langle v\rangle^{2-2\gamma}   \mathbb{P}^{\bot}f^\e\|_{ \nu}^2
			\le C C_{\eta}	\| \mathbb{P}^{\bot} f\|_{\nu}
			+\sum_{i=15}^{17}  \mathcal{S}_i,
		\end{align}
		where
		\begin{align*}
			\mathcal{S}_{15}=&	\left\langle   \nabla_x \phi^\varepsilon \cdot \nabla_v (\mathbb{P}^{\bot}f^\varepsilon) - \frac{v-u_0}{2{\theta }_0} \cdot \nabla_x \phi^\varepsilon \,       \mathbb{P}^{\bot}f^\varepsilon	-   \mathbb{P}^{\bot}\left(\frac{\{\partial_t+v\cdot\nabla_x\}\sqrt\mu}{\sqrt\mu} f^\varepsilon\right),  \; \e   {\langle v \rangle^ {4-4\gamma}} \mathbb{P}^{\bot}f^\varepsilon \right\rangle,  \nonumber \\
			\mathcal{S}_{16}=&	 \left\langle  \varepsilon^{k-1}    \Gamma(f^\varepsilon,   f^\varepsilon)+ \sum_{i=1}^{2k-1} \varepsilon^{i-1}     \left\{
			\Gamma\!\left( \frac{F_i}{\sqrt{\mu}},   f^\varepsilon \right)
			+ \Gamma\!\left( f^\varepsilon,   \frac{F_i}{\sqrt{\mu}} \right) \right\},    \; \e   {\langle v \rangle^{4-4\gamma}} \mathbb{P}^{\bot}f^\varepsilon \right\rangle,  \nonumber \\
			\mathcal{S}_{17}
			=&\;
			\left\langle
			\sum_{i=1}^{2k-1} \varepsilon^{i}\,   \mathbb{P}^{\bot}
			\Big( \frac{1}{\sqrt{\mu}}\,   \nabla_x \phi_R^\varepsilon \cdot \nabla_v F_i \Big)
			+ \varepsilon^{k-1}\,   \mathbb{P}^{\bot}(\overline{A})
			+ [[\mathbb{P},  \mathcal{A}_{\phi^\varepsilon}]] f^\varepsilon,\;
			\varepsilon\,   \langle v \rangle^{4-4\gamma} \mathbb{P}^{\bot}f^\varepsilon
			\right\rangle.
		\end{align*}
		Using integration by parts,
		$ |\mathbb{P}^{\bot}f^\varepsilon|\leq |f^\e|$,
		\eqref{decay},  \eqref{phi},    \eqref{h},   and Young's inequality,   we have
		\begin{align}\label{s11}
			\mathcal{S}_{15}
			\lesssim&\e \|\langle v \rangle^{7-4\gamma } f^\e\mathbf{1}_{\{\langle v \rangle^{14-9\gamma }\geq \kappa^2/\e	
				\}}\|_{\infty} \|\nabla_x(\phi^\e, \rho_0,   u_0,  \theta_0)\| \|\mathbb{P}^{\bot}f^\varepsilon\|\nonumber \\ 		\nonumber	&+\e  \|\nabla_x(\phi^\e,\rho_0,   u_0,  {\theta_0})\|_{\infty} \left(\|\langle v \rangle^{7-4\gamma }\mathbb{P}f^\e\|+\|\langle v \rangle^{7-\frac{9\gamma}{2} }\mathbb{P}^{\bot}f^\varepsilon\mathbf{1}_{\{\langle v \rangle^{14-9\gamma }\leq \kappa^2/\e	
				\}}\|_{\nu} \right)\|\mathbb{P}^{\bot}f^\varepsilon\| \\
			\lesssim&\e^4\|h^\e\|_{\infty}\|f^\e\|+(1+t)^{-16/15} \|\sqrt{\e}f^\e\|^2+  {{\kappa}^2 }  \|\mathbb{P}^{\bot}f^\varepsilon\|_{\nu}^2.
		\end{align}
		Using the same argument as for $\mathcal{S}_{13}$ in \eqref{s13. 3},   we obtain
		\begin{align}
			\mathcal{S}_{16}	
			\lesssim&\sqrt{\e}\|\sqrt{\e}f^\e\|^2+\sqrt{\e}\|\langle v \rangle^{2-2\gamma}\mathbb{P}^{\bot}f^\e\|_{\nu}^2.
		\end{align}
		Applying \eqref{pg},  \eqref{fir3},  Lemma \ref{oa} and \eqref{phi},  it holds that
		\begin{align}\label{s13}
			\mathcal{S}_{17}	\lesssim&\e	\sum_{i=1}^{2k-1}
			\varepsilon^i
			\| \nabla_x\phi_R^\varepsilon\|	\left(
			\int
			\left| \langle v \rangle^{4-4\gamma} \frac{\nabla_v F_i}{\sqrt{\mu}}
			\right| ^2
			\d v
			\right)^{1/2} \left\|
			\mathbb{P}^{\bot}f^\e \right\|\nonumber\\
			&+\e^{k }\|   \langle v \rangle^{4-4\gamma} \mathbb{P}^{\bot}(\overline{A})    \|\|   \mathbb{P}^{\bot}f^\e\|+\e\|  [[\mathbb{P},  \mathcal{A}_{\phi^\varepsilon}]] f^\varepsilon  \langle v \rangle^{4-\frac{9\gamma}{2} }  \|\|   \mathbb{P}^{\bot}f^\e\|_{\nu}\nonumber\\
			\lesssim& \sqrt{\e}\|\sqrt{\e}f^\e\|_{H^1}^2+\sqrt{\e}\| \mathbb{P}^{\bot}f^\e\|_{\nu}^2+\e.
		\end{align}
		Substituting \eqref{s11}-\eqref{s13} into \eqref{wev},  we have
		\begin{align}\label{wev1}
			&\; \frac{1}{2}\frac{\d }{\d t} \| \sqrt{\e}  \langle v\rangle^{2-2\gamma}\mathbb{P}^{\bot}f^\e\|_{ }^2+ \frac{{\delta_1}}{2} \| \langle v\rangle^{2-2\gamma}   \mathbb{P}^{\bot}f^\e\|_{ \nu}^2\\\nonumber
			\lesssim& \e^4\|h^\e\|_{\infty}\|f^\e\| +((1+t)^{-16/15}+\sqrt{\e})\|\sqrt{\e}f^\e\|_{H^1}^2  +\left(	C_{\eta}+\kappa^2  +	\sqrt{\e} \right) \|\mathbb{P}^{\bot}f^\varepsilon\|_{\nu}^2+\sqrt{\e}\|\langle v \rangle^{2-2\gamma}\mathbb{P}^{\bot}f^\e\|_{\nu}^2+\e.
		\end{align}
		Combining \eqref{l2},  \eqref{l2x},  \eqref{l2v},  \eqref{l2wxip} and \eqref{wev1},  and choosing $\kappa^2, \, \eta,\,  \delta_0$ small enough. This completes the proof of \eqref{L2a}.
	\end{proof}
	\section{\texorpdfstring{$W_{x,v}^{1,\infty}$ Estimates for the Remainder with Soft Potentials}{W  Estimates for the Remainder with Soft Potentials}}
	In this  section, we   consider the $W_{x, v}^{1,\infty}$-estimate of  $F_{R}^{\protect\varepsilon%
	}$. Define the characteristics $[X(\tau
	;t, x, v),V(\tau;t, x, v)]$ passing though $(t, x, v)$ such that
	\begin{equation}
		\left\{
		\begin{array}{ll} \label{char}
			\displaystyle \frac{\d X(\tau;t, x, v)}{\d \tau}=\frac{1}{\e}V(\tau;t, x, v),  \quad\quad\quad\quad\quad\quad           X(t;t, x, v)=x,&\\[2mm]
			\displaystyle \frac{\d V(\tau;t, x, v)}{\d \tau}=-\nabla_x\phi^\e (\tau, X(\tau;t, x, v)), \quad\quad  V(t;t, x, v)=v.&\\
		\end{array}
		\right.
	\end{equation}
	Denote
	\begin{equation}
		-\frac{1}{\sqrt{\mathcal{M}}}\{Q(\mu,\sqrt{\mathcal{M}}g )+Q( \sqrt{{\mathcal{M}}}g,\mu )\}=\mathcal{L}\left(g\right)=\{\nu (\mu )-\mathcal{K}\}g \label{K_M}.
	\end{equation}%
	Let  $ \mathcal{K}_{w}(h^\e):=\; w  \mathcal{K}\Big(\frac{h^\e} {w }\Big)$,	
	with   $F^\e_R=\frac{\sqrt{\mathcal{M}}h^\e}{w}$ in \eqref{f}, (\ref{F_R}) can be rewritten as
	\begin{equation} \label{eq:h}
		\begin{split}
			& \partial_t h^{\varepsilon}
			+ v \cdot \nabla_x h^{\varepsilon}
			-\nabla_x \phi^{\varepsilon} \cdot \nabla_v h^{\varepsilon}
			+ \frac{\widetilde{\nu} }{\varepsilon} h^{\varepsilon}
			\\
			=& \frac{1}{\varepsilon} \mathcal{K}_{w} h^{\varepsilon}+ \varepsilon^{k-1} w\Gamma_M\left(  \frac{h^{\varepsilon}  }{w},  \frac{h^{\varepsilon  }} {w} \right)
			+ \sum_{i=1}^{2k-1} \varepsilon^{i-1}  {w}  \left\{\Gamma_M\left(  \frac{F_i  }{\sqrt{\mathcal{M}} },  \frac{h^{\varepsilon  }} {w} \right)+ \Gamma_M\left(\frac{h^{\varepsilon  }} {w},  \frac{F_i  }{\sqrt{\mathcal{M}} }   \right) \right\} \\
			&
			+\frac{w}{\sqrt{\mathcal{M}}}\nabla _{x}\phi
			_{R}^{\varepsilon }\cdot \nabla _{v}\left(\mu
			+\sum_{i=1}^{2k-1}\varepsilon ^{i}F_{i}\right)+\varepsilon ^{k-1}\frac{w}{\sqrt{
					\mathcal{M}}}A,
		\end{split}
	\end{equation}
	where $\frac{\widetilde{\nu} }{\varepsilon}$ is defined by
	\begin{align}\label{nuw}
		\frac{\widetilde{\nu} }{\varepsilon}:=\frac{\nu}{\e}-	 \nabla _{x}\phi ^{\varepsilon }\cdot \frac{w\nabla_v \left( \frac{\sqrt{\mathcal{M}}}{w}  \right)}{\sqrt{\mathcal{M}}}  -\frac{\partial_t w}{w}.
	\end{align}
	Notice that
	\(\mathcal{M} \lesssim \mu \lesssim \mathcal{M}^{\varpi},  \frac{1}{2}<\varpi<1,\) we     obtain
	\begin{align}\label{kmo}
		\nonumber	\mathcal{K}  |g| \lesssim& \;\iint_{\mathbb{R}^3\times\mathbb{S}^2 }|v-u|^{\gamma}  b_0(\alpha) {\mathcal{M}^{\varpi-1/2}(u)}\left({\mathcal{M}^{\varpi-1/2}(u^{\prime})}  |g|(v^{\prime}) +{\mathcal{M}^{\varpi-1/2}(v^{\prime})}  |g|(u^{\prime}) \right)\d u\d\omega\\
		&- \iint_{\mathbb{R}^3\times\mathbb{S}^2 }|v-u|^{\gamma}  b_0(\alpha){\mathcal{M}^{\varpi-1/2} (v)}{\mathcal{M}^{\varpi-1/2} (u)} |g|(u)  \d u\d\omega \\\nonumber
		:=&\int_{\mathbb{R}^3} \l  (v, u)|g(u)|\d u,
	\end{align}
	and
	\begin{align}\label{kmnwh}
		\mathcal{K}_{w}  g=w\mathcal{K}   \left(\frac{g}{w}\right)\lesssim
		\int_{\mathbb{R}^3} l (v, u)\frac{w(v)}{w(u)}g(u)\d u=	\int_{\mathbb{R}^3} l_{w} (v, u) g(u)\d u.
	\end{align}
	Note that   $\mathcal{M}^{\varpi-1/2}$ in $l_w$ is parallel to $\mathcal{M}$.
	Employing the same method as in \cite{strainarma2008}, we can deduce that $l_w$ satisfies the properties stated in Lemma~\ref{eq:es:K}.

	In the case of soft potentials, to treat the singularity of  $\mathcal{K}_w$ in \eqref{kmnwh}, we   introduce a smooth cutoff function $0 \leq \chi \leq 1$, $m\leq 4^{\frac{-1}{\gamma+3}}$ such that
	\begin{align*}
		\chi (s)=\left\{
		\begin{array}{ll}
			\displaystyle 0,  \quad\quad          \text{for} \;|s|\leq m,&\\[2mm]
			\displaystyle 1, \quad\quad       \text{for} \;|s|\geq 2m.& \\
		\end{array}
		\right.
	\end{align*}
	Then from $\eqref{kmo}$, we  use $ \chi $ to split $ \mathcal{K}|g| =  \mathcal{K}^{ \chi}|g|+  {\mathcal{K}}^{1- \chi}|g|$ and $ \mathcal{K}_wg =  \mathcal{K}_w^{ \chi}|g|+  {\mathcal{K}}_w^{1- \chi}|g|$. Here, $ \mathcal{K}^{\chi} |g| $
	and $ \mathcal{K}^{\chi}_w |g| $ are defined as
	\begin{align*}
		\nonumber	\mathcal{K}^{\chi} |g| \lesssim& \;\iint_{\mathbb{R}^3\times\mathbb{S}^2 }|v-u|^{\gamma} \chi (|v-u|)b_0(\alpha) {\mathcal{M}^{\varpi-1/2}(u)}\left({\mathcal{M}^{\varpi-1/2}(u^{\prime})}  |g(v^{\prime})| +{\mathcal{M}^{\varpi-1/2}(v^{\prime})}  |g(u^{\prime}) | \right)\d u\d\omega\\
		&- \iint_{\mathbb{R}^3\times\mathbb{S}^2 }|v-u|^{\gamma} \chi (|v-u|)b_0(\alpha){\mathcal{M}^{\varpi-1/2} (v)}{\mathcal{M}^{\varpi-1/2} (u)} |g(u)|  \d u \d\omega \\\nonumber
		:=&\int_{\mathbb{R}^3} l^{ \chi}(v, u)|g(u)|\d u.
	\end{align*}
	and
	\begin{align}\label{kmnw}
		\mathcal{K}_{w}^{\chi} |g|=w\mathcal{K} ^{\chi} \left(\frac{|g|}{w}\right)\lesssim
		\int_{\mathbb{R}^3} l^{ \chi}(v, u)\frac{w(v)}{w(u)}|g(u)|\d u=	\int_{\mathbb{R}^3} l_{w}^{\chi}(v, u) |g(u)|\d u.
	\end{align}

	Our main task is
	to derive $W^{1,\infty }_{x,v}$ estimates of $h^{\varepsilon }$ as follows.
	\begin{proposition}\label{softin}
		Suppose  $-3<\gamma<0$, it holds that
		\begin{align}\label{hin}
			\e^{3/2}	\sup_{0\leq t \leq \e^{-1/2}}  \| {  h}^\varepsilon (t)\|_{W^{1, \infty}}	\lesssim \e^{3/2}  \| {  h}^\varepsilon (0)\|_{W^{1, \infty}} +\sup_{0\le s\le t} \{\|f^\e(s )\|_{H^1}\}+\e^{5/2}.
		\end{align}
	\end{proposition}
	\begin{proof}
		\medskip
		\emph{ \bf Step 1. $L_{x, v}^\infty$-estimate of $ {h}^\e$.}\;	We first consider $\|  h^\e\|_{{ \infty}}$  on the time interval
		$ 0\leq s\leq   t \le  T_0 \ll 1$.
		Applying Duhamel's principle to the system \eqref{eq:h}, we have,
		\begin{align}\label{eq:h1}
			& h^{\varepsilon }(t, x, v)=\exp \left\{-\frac{1}{\varepsilon }\int_{0}^{t}\widetilde{\nu}
			(\tau )d\tau \right\}h^{\varepsilon }(0, X(0;t, x, v),V(0;t, x, v))+\sum_{i=1}^5  \mathcal{J}_i,
		\end{align}
		where
		\begin{align*}
			\mathcal{J}_1=&\int_{0}^{t}\exp \{-\frac{1}{\varepsilon }\int_{s}^{t}\widetilde{\nu} (\tau )d\tau
			\}\left( \frac{1}{\varepsilon }{\mathcal{K}_{w}}h^{\varepsilon }\right) (s, X(s),V(s))\d s,\\
			\mathcal{J}_2=&\int_{0}^{t}\exp \{-\frac{1}{\varepsilon }\int_{s}^{t}\widetilde{\nu} (\tau )d\tau
			\}\left( \frac{\varepsilon ^{k-1}w}{\sqrt{\mu _{M}}}Q(\frac{%
				h^{\varepsilon }\sqrt{\mu _{M}}}{w},\frac{h^{\varepsilon }\sqrt{\mu
					_{M}}}{w})\right) (s, X(s),V(s))\d s, \\	
			\mathcal{J}_3=&   \int_{0}^{t}\exp \{-\frac{1}{\varepsilon }\int_{H^s}^{t}\widetilde{\nu} (\tau )d\tau
			\}\left( \sum_{i=1}^{2k-1}\varepsilon ^{i-1}\frac{w}{\sqrt{\mu _{M}}}%
			\left(Q(F_{i},\frac{h^{\varepsilon }\sqrt{\mu _{M}}}{w})+Q(\frac{h^{\varepsilon }\sqrt{\mu _{M}}}{w}),F_{i}\right)\right) (s, X(s),V(s))\d s,
			\\
			\mathcal{J}_4=& \int_{0}^{t}\exp \{-\frac{1}{\varepsilon }\int_{s}^{t}\widetilde{\nu} (\tau )d\tau
			\}\left( \nabla _{x}\phi _{R}^{\varepsilon }\cdot \frac{w}{\sqrt{\mathcal{M}}%
			}\nabla _{v}(\mu +\sum_{i=1}^{2k-1}\varepsilon ^{i}F_{i})\right)
			(s, X(s),V(s))\d s,  \\
			\mathcal{J}_5=&\int_{0}^{t}\exp \{-\frac{1}{\varepsilon }\int_{s}^{t}\widetilde{\nu} (\tau )d\tau
			\}\left( \varepsilon ^{k-1}\frac{w}{\sqrt{\mathcal{M}}}A\right)
			(s, X(s),V(s))\d s.
		\end{align*}
		\medskip
		\emph{\bf Step 1.1. Estimates for $\mathcal{J}_2$--$\mathcal{J}_5$.}
		By  \eqref{wt}, \eqref{fir3}, \eqref{pri}  and \eqref{nus}, we have
		\begin{align}
			\mathcal{J}_2+\mathcal{J}_3 \lesssim&\left(\e^{k}\|h^\e\|_{\infty}+\e\sum_{i=1}^{2k-1}\e^{i-1}\|F_i\|_{\infty}  \right)\int_{0}^{t}\exp \{-\frac{1}{\varepsilon }\int_{s}^{t}\widetilde{\nu} (\tau )d\tau
			\}\frac{1}{\e}\widetilde{\nu}(s) e^{-as^{\varrho}}\d s\sup_{0 \le s \le t}\{e^{as^{\varrho}}\|h^\e\|_{\infty}\}\nonumber\\
			\lesssim\;&\e e^{-at^{\varrho}}\sup_{0 \le s \le t}\{e^{as^{\varrho}}\|h^\e(s)\|_{\infty}\}.
		\end{align}		
		Using \eqref{nus}, \eqref{fir3}, and \eqref{phi1}, with $\widetilde{\nu}^{-1} \lesssim \nu^{-1} \lesssim \langle v \rangle^{-\gamma}$ and $\sum_{i=1}^{2k-1} \varepsilon^i (1+t)^{i-1} \lesssim \varepsilon$, we derive
		\begin{align}
			\mathcal{J}_4\lesssim&\e \int_{0}^{t}\exp\Big\{-\frac{1}{\e}\int_{s}^{t}\widetilde{\nu}(\tau)\d \tau\Big\}\frac{1}{\e}\widetilde{\nu}(s) e^{-as^{\varrho}}\d s \sup_{0 \le s \le t}\{e^{as^{\varrho}}\|\nabla_x \phi_R^\e\|_{\infty}\} \left\|\langle v \rangle^{-\gamma}	\frac{w}{\sqrt{\mathcal{M}}}	\nabla_v \left(\mu+\sum_{i=1}^{2k-1}\e^i F_i\right) \right\|_{\infty}\nonumber\\
			\lesssim&\e e^{-at^{\varrho}} \left(1+\sum_{i=1}^{2k-1}\e^i(1+t)^{i-1}\right)\sup_{0 \le s \le t}\{e^{as^{\varrho}}\|\nabla_x \phi_R^\e(s)\|_{H^2}\} \\\nonumber
			\lesssim&\e e^{-at^{\varrho}}\sup_{0\le s\le t}\left\{e^{{a} s^{\varrho}}
			\| {f}^\varepsilon (s)\|_{H^1}\right\}+\e.
		\end{align}
		By \eqref{nuexp}, the Sobolev embedding $ H^2(\mathbb{R}^3)\hookrightarrow L^\infty(\mathbb{R}^3)$, and Lemma \ref{fi}, we obtain
		\begin{align}
			\mathcal{J}_5\lesssim&  \int_{0}^{t}  \left( \varepsilon ^{k-1}\frac{w}{\sqrt{\mathcal{M}}}|A|\right)\d s\lesssim t\varepsilon ^{k-1}\sum_{2k+1\leq i+j\leq 4k-2}\e^{i+j-2k} (1+t)^{i+j-2}
			\lesssim  \varepsilon ^{k }  (1+T_0)^{2k}\lesssim {\e^k}.		
		\end{align}
		\medskip
		\emph{\bf Step 1.2. Estimates for $\mathcal{J}_1  $}.
		To handle the singularity of  $\mathcal{J}_1  $, we spilt it into
		\begin{align*}
			\mathcal{J}_1= \mathcal{J}_1^1+ \mathcal{J}_1^2,
		\end{align*}
		where
		\begin{align*}
			\mathcal{J}_1^1=	&\int_{0}^{t}\exp \{-\frac{1}{\varepsilon }\int_{s}^{t}\widetilde{\nu} (\tau )d\tau
			\}\left( \frac{1}{\varepsilon }\mathcal{K}_{w}^{1-\chi} h^{\varepsilon }\right) (s, X(s),V(s))\d s,\\
			\mathcal{J}_1^2=&\int_{0}^{t}\exp \{-\frac{1}{\varepsilon }\int_{s}^{t}\widetilde{\nu} (\tau )d\tau
			\}\left( \frac{1}{\varepsilon }\mathcal{K}_{w}^{\chi} h^{\varepsilon }\right) (s, X(s),V(s))\d s.
		\end{align*}
		By  \eqref{nus} and \eqref{eq:es:Kc:2}, we have
		\begin{align}
			\mathcal{J}_1^1
			&\lesssim
			\varepsilon^{\gamma+3} e^{-a t^{\varrho}}
			\sup_{0\le s\le t}
			\left\{
			e^{a s^{\varrho}}
			\|h^\varepsilon(s)\|_{\infty}
			\right\}.
		\end{align}
		The expression of $h^{\varepsilon}(s, X(s),V(s))$ is analogous to that of $h(t)$ in \eqref{eq:h1}. Substituting it into $\mathcal{J}_1^2$ and
		using the similar estimates as $\mathcal{J}_2$--$\mathcal{J}_5 $ and $\mathcal{J}_1^1 $, we   obtain
		\begin{align}\label{j12}
			\mathcal{J}_1^2 \lesssim	 &   e^{-at^{\varrho}}\|h^\e(0)\|_{\infty}+ \left( {\e}+{\e}^{\gamma+3}\right)e^{-{a} t^{\varrho}}\sup_{0\le s\le t}\left\{e^{{a} s^{\varrho}}
			\| {h}^\varepsilon (s)\|_{\infty}\right\}
			\\\nonumber
			&+\e e^{-at^\varrho }\sup_{0\le s\le t}\left\{e^{{a} s^{\varrho}}
			\| {f}^\varepsilon (s)\|_{H^1}\right\} +\e+\widetilde{\mathcal{J}_1^2}.
		\end{align}
		Here, $\widetilde{\mathcal{J}_1^2}$ is defined as
		\begin{align*}
			\widetilde{\mathcal{J}_1^2}
			= \int_0^t & \exp\Big\{-\frac{1}{\varepsilon}\int_s^t \widetilde{\nu}(\tau)\,d\tau\Big\}
			\int_0^s \exp\Big\{-\frac{1}{\varepsilon}\int_{s_1}^s \widetilde{\nu}(\tau')\,d\tau'\Big\} \\
			& \times \iint
			\big| l_w^\chi(V(s),v') l_w^\chi(V(s_1),v'') h^\varepsilon(s_1, X(s_1),V(s_1)) \big| \, \d v''\, \d v' \, ds_1\, ds.
		\end{align*}
		Using the same method to estimate $\widetilde{\mathcal{J}_1^2}$ as in \cite{lisiam2023}, we have
		\begin{align}\label{j121}
			\widetilde{\mathcal{J}_1^2}	 \lesssim&o(1)e^{-{a} t^{\varrho}}\sup_{0\le s\le t}\left\{e^{{a} s^{\varrho}}
			\| {h}^\varepsilon (s)\|_{\infty}\right\}+\frac{1}{\e^{{3}/{2}}}e^{-{a} t^{\varrho}}\sup_{0\le s\le t}\left\{e^{{a} s^{\varrho}}
			\| {f}^\varepsilon (s)\|  \right\}.
		\end{align}
		Substituting the above estimates on $\mathcal{J}_1\sim\mathcal{J}_5$ into (\ref{eq:h1}), and    combining $ \exp \{-\frac{1}{\varepsilon }\int_{0}^{t}\nu
		(\tau )d\tau \}h^{\varepsilon }(0 )\lesssim e^{-at^{\varrho}}\|h_0^{\varepsilon }\|_{\infty},$ `for $0\leq s\leq t\leq T_0 \ll 1$, we derive
		that,
		\begin{align*}
			\| {h}^\varepsilon (t)\|_{ {\infty}}\lesssim 	& e^{-at^{\varrho}}\|h^\e(0)\|_{\infty}+ \left( {\e}+{\e}^{\gamma+3}+o(1)\right)e^{-{a} t^{\varrho}}\sup_{0\le s\le t}\left\{e^{{a} s^{\varrho}}
			\| {h}^\varepsilon (s)\|_{\infty}\right\}\nonumber
			\\
			&+\e e^{-at^\varrho }\sup_{0\le s\le t}\left\{e^{{a} s^{\varrho}}
			\| {f}^\varepsilon (s)\|_{H^1}\right\}+\frac{1}{\e^{{3}/{2}}}e^{-{a} t^{\varrho}}\sup_{0\le s\le t}\left\{e^{{a} s^{\varrho}}
			\| {f}^\varepsilon (s)\|  \right\} +\e.
		\end{align*}
		Multiplying the above inequality by \( e^{a s^{\varrho}} \),   taking the supremum, letting \( s = T_0 \) on the left-hand side and multiplying both sides by \( \varepsilon^{3/2} e^{-a T_0^{\varrho}} \), with $e^{-{a}{{T_0}^\varrho}}\leq \frac{1}{2}$, we obtain
		\begin{align}\label{4.16}
			\e^{3/2} \| {h}^\varepsilon (T_0)\|_{ {\infty}}
			\leq &\frac{1}{2}\e^{3/2}\| {h}^\varepsilon ( 0) \|_{\infty}+ C\sup_{0\le s\le t}
			\| {f}^\varepsilon (s)\|   +C \e^{5/2}\sup_{0\le s\le t}
			\| {f}^\varepsilon (s)\|_{H^1} +C\e^{5/2}.
		\end{align}
		\smallskip
		\emph{\bf Step 2.  $L_{x, v}^\infty$-estimates of $\nabla_x {h}^\varepsilon$  and $\nabla_v {h}^\varepsilon$.}\;
		Applying $  \nabla_x$ and  $  \nabla_v$ to the equation (\ref{eq:h}) respectively, denoting  $ {\nabla}h^\e:=   \nabla_xh^\e +   \nabla_vh^\e $, and using Duhamel's principle, we have
		\begin{align}\label{eq:Dh1}
			\nabla h^{\varepsilon }(t, x, v)=\exp \{-\frac{1}{\varepsilon }\int_{0}^{t}\widetilde{\nu}
			(\tau )d\tau \}\nabla  h^{\varepsilon }(0, X(0;t, x, v),V(0;t, x, v))
			+\sum_{i=1}^{8}\mathcal{ G }_i,
		\end{align}
		where $\mathcal{ G }_1$--$\mathcal{ G }_8$ is given by
		\begin{align*}
			\mathcal{ G }_1&:=\int_0^t \text{exp}\Big\{-\frac{1}{\e}\int_s^t\widetilde{\nu}(\tau)\d \tau\Big\}|[\nabla_x(  \nabla_x \phi^\e)\cdot \nabla_v h^\e](s, X(s),V(s))|\d s,\\
			\mathcal{ G }_2&:= -\int_0^t \text{exp}\Big\{-\frac{1}{\e}\int_s^t\widetilde{\nu}(\tau)\d \tau\Big\}|(  \nabla_x  h^\e)(s, X(s),V(s))|\d s,\\
			\mathcal{ G }_3&:=- \int_0^t \text{exp}\Big\{-\frac{1}{\e}\int_s^t\widetilde{\nu}(\tau)\d \tau\Big\}\left|\left[  \nabla_x \left(\frac{\widetilde{\nu}}{\e} \right)h^\e+  \nabla_v \left(\frac{\widetilde{\nu}}{\e} \right)h^\e \right](s, X(s),V(s))\right|\d s,\\
			\mathcal{ G }_4&:= -\int_0^t \text{exp}\Big\{-\frac{1}{\e}\int_s^t\widetilde{\nu}(\tau)\d \tau\Big\}\frac{1}{\e} {\nabla}\left[{\mathcal{K}_{w}}(h^\e)\right](s, X(s),V(s))\d s,\\
			\mathcal{ G }_5&:=\e^{k-1}\int_0^t \text{exp}\Big\{-\frac{1}{\e}\int_s^t\widetilde{\nu}(\tau)\d \tau\Big\}\nabla\Big[\frac{w}{\sqrt{\mathcal{M}}} Q\Big( \frac{{\sqrt{\mathcal{M}}} h^\e}{w}, \frac{{\sqrt{\mathcal{M}}} h^\e}{w}\Big)\Big] (s, X(s),V(s))\d s,\\
			\mathcal{ G }_6
			&:=\int_0^t \text{exp}\Big\{-\frac{1}{\e}\int_s^t\widetilde{\nu}(\tau)\d \tau\Big\}\\
			&\quad\,\times\sum_{i=1}^{2k-1}\e^{i-1}\nabla\Big\{\frac{w}{\sqrt{\mathcal{M}}}\Big[Q\Big(F_i, \frac{{\sqrt{\mathcal{M}}} h^\e}{w}\Big)
			+Q\Big( \frac{{\sqrt{\mathcal{M}}} h^\e}{w},F_i\Big)\Big]\Big\} (s, X(s),V(s))\d s,\\
			\mathcal{ G }_7&:= \int_0^t \text{exp}\Big\{-\frac{1}{\e}\int_s^t\widetilde{\nu}(\tau)\d \tau\Big\}  {\nabla}\Big[\frac{w}{\sqrt{\mathcal{M}}}\nabla_x\phi_R^\e\cdot\nabla_v\Big(\mu+\sum_{i=1}^{2k-1}\e^{i}F_i\Big)\Big] (s, X(s),V(s))\d s,\\
			\mathcal{ G }_8&:=\int_0^t \text{exp}\Big\{-\frac{1}{\e}\int_s^t\widetilde{\nu}(\tau)\d \tau\Big\} \Big[ \e^{k-1} {\nabla}\Big(\frac{
				wA}{\sqrt{\mathcal{M}}}  \Big) \Big] (s, X(s),V(s))\d s.
		\end{align*}
		\emph{	 \bf Step 2.1. Estimates for $\mathcal{G}_1$--$\mathcal{G}_3,\, \mathcal{G}_5$--$\mathcal{G}_8.$}
		Applying \eqref{phi} and \eqref{nuexp}, we get
		\begin{align*}
			\mathcal{ G }_1+ \mathcal{ G }_2\lesssim&T_0\| \nabla^2_x \phi^\e \|_{ \infty}e^{-at^{\varrho}}  \sup_{0\le s\le t}\Big\{ e^{ {a}s^{\varrho}}\| {  \nabla_v} h^\e  \|_{ \infty}\Big\}+T_0 e^{-at^{\varrho}}  \sup_{0\le s\le t}\Big\{ e^{ {a}s^{\varrho}}\| {  \nabla_x} h^\e  \|_{ \infty}\Big\}\\
			\lesssim &T_0 e^{-at^{\varrho}}  \sup_{0\le s\le t}\Big\{ e^{ {a}s^{\varrho}}\|  {\nabla} h^\e  \|_{\infty }\Big\}.
		\end{align*}	
		Using  \eqref{dnus}, it holds that
		$
		\mathcal{ G }_3
		\lesssim e^{-at^{\varrho}}  \sup_{0\le s\le t}\Big\{ e^{ {a}s^{\varrho}}\|   h^\e  \|_{\infty }\Big\}.
		$
		By \eqref{wt}, \eqref{wtv},  \eqref{nus},    the a  $\mathit{priori}$  assumption \eqref{pri}  and \eqref{fir2}, we have
		\begin{align*}
			\mathcal{ G }_5+		\mathcal{ G }_6
			\lesssim&\left(\e^k\|h^\e\|_
			{W^{1,\infty}}+\e\sum_{i=1}^{2k-1}[\e(1+t)]^{i-1}  \right)e^{-at^{\varrho}}\sup_{0\le s\le t}\Big\{ e^{ {a}s^{\varrho}}\|     h^\e  \|_{\infty }\Big\}
			\lesssim \e e^{-at^{\varrho}}\sup_{0\le s\le t}\Big\{ e^{ {a}s^{\varrho}}\|     h^\e  \|_{W^{1,\infty} }\Big\}.
		\end{align*}	
		Applying the same method used for $\mathcal{G}_4$ and $\mathcal{G}_5$ to $\mathcal{G}_7$ and $\mathcal{G}_8$ respectively, and combining \eqref{phiinfnity},  we obtain
		\begin{align}
			\mathcal{ G }_7+	\mathcal{ G }_8
			\lesssim&\e e^{-{a}t^\varrho}\sup_{0\le s\le t}\Big\{ e^{{a}s^\varrho}\|\nabla_x \phi_R^\e\|_{W^{1,\infty}}\big\}+\e^k
			\lesssim\varepsilon e^{-{a}t^\varrho}\sup_{0\le s\le t}\Big\{ e^{{a}s^\varrho} ( \|h^{\varepsilon}\|_{W^{1,\infty} }+\|f^{\varepsilon}\|_{H^{1}   }	 ) \big\} +\varepsilon.
		\end{align}
		\emph{\bf Step 2.2. Estimate for $\mathcal{G}_4$.}
		Decompose
		\begin{align}\label{g4}
			\mathcal{G}_4=	 \mathcal{G}_4^1+\mathcal{G}_4^2,
		\end{align}
		where
		\begin{align*}
			\mathcal{G}_4^1=& -\int_0^t \text{exp}\Big\{-\frac{1}{\e}\int_s^t\widetilde{\nu}(\tau)\d \tau\Big\}\frac{1}{\e}\left( {\nabla}[\mathcal{K}^{1-\chi}_w(h^\e)]+[\left(\nabla_v\mathcal{K}^{\chi}_{w }\right)(h^\e)]\right) (s, X(s),V(s))\d s	, \\
			\mathcal{G}_4^2=& -\int_0^t \text{exp}\Big\{-\frac{1}{\e}\int_s^t\widetilde{\nu}(\tau)\d \tau\Big\}\frac{1}{\e} [\mathcal{K}^{\chi}_w(\nabla h^\e)](s, X(s),V(s))\d s	.
		\end{align*}
		By ${\nabla}[\mathcal{K}^{1-\chi}_w(h^\e)]=\mathcal{K}^{1-\chi}_w(\nabla h^\e)+\left(\nabla_v\mathcal{K}^{1-\chi}_w\right)(h^\e)$ and applying \eqref{eq:es:Kc:2},  \cite[Lemma 2.2]{guo2002cpam}, the kernel of  $\left(\nabla_v\mathcal{K}^{\chi}_{w }\right)$ has the same properties as $\mathcal{K}^{\chi}_{ w }$, with \eqref{eq:es:Kc:1}, we gain
		\begin{align}\label{g41}
			\mathcal{G}_4^1\lesssim  \e^{\gamma+3}e^{-{a}t^\varrho}\sup_{0\le s\le t}\Big\{ e^{{a}s^\varrho}  \|h^{\varepsilon}\|_{W^{1,\infty} } \Big\} +e^{-{a}t^\varrho}\sup_{0\le s\le t}\Big\{ e^{{a}s^\varrho}  \|h^{\varepsilon}\|_{ {\infty} } \Big\}	.
		\end{align}
		Iterating \eqref{eq:Dh1} of $ {\nabla}h^\e$ in $\mathcal{G}_4^2$, using above method  to estimate term by term. Applying \eqref{eq:es:Kc:1}, \eqref{nus}  and  $ \widetilde{\nu}\gtrsim \nu$，
		we have
		\begin{align}\label{g42}
			\mathcal{G}_4^2\lesssim &^{-at^{\varrho}}\|\nabla h^\e(0)\|_{\infty}	+\left(T_0 +\e+\e^{\gamma+3} \right)e^{-at^{\varrho}}  \sup_{0\le s\le t}\Big\{ e^{ {a}s^{\varrho}}\|  {\nabla} h^\e  \|_{\infty }\Big\} \nonumber\\
			&+\left(1+\e+\e^{\gamma+3}\right)e^{-{a}t^\varrho}\sup_{0\le s\le t}\Big\{ e^{{a}s^\varrho}  \|h^{\varepsilon}\|_{ {\infty} } \Big\}	+\e e^{-{a}t^\varrho}\sup_{0\le s\le t}\left\{e^{{a}s^\varrho}  \|f^{\varepsilon}(s)\|_{ {H^1}   }	 \right\}+\e +\widetilde{\mathcal{G}}_4^2,
		\end{align}
		where $\widetilde{\mathcal{G}}_4^2$ is given by
		\begin{align*}
			\widetilde{\mathcal{G}}_4^2 = \int_0^t&\text{exp}
			\Big\{-\frac{1}{\e}\int_s^t\widetilde{\nu}( \tau)\text{d}\tau\Big\}\frac{\widetilde{\nu}(s)}{\e}   \int_0^s  \text{exp}\Big\{-\frac{1}{\e}\int_{s_1}^{s}\widetilde{\nu} (\tau')\text{d}\tau'\Big\}\frac{\widetilde{\nu}(s_1)}{\e}   \\\nonumber
			&  \times \nu^{-1}(V(s)) \nu^{-1}(V(s_1))\int_{ v'\in \mathbb{R}^3}\int_{ v''\in \mathbb{R}^3}\left| {l_{w}^{\chi} } (V(s),v') {l_{w}^{\chi}}  (V(s_1),v'')  {\nabla h}^\varepsilon (s_1, X(s_1),V(s_1))\right|\text{d}v''\text{d}v'\text{d}s_1\d s.
		\end{align*}
		From  \eqref{phi}, \eqref{char}
		and $0\leq s\leq t\leq T_0 \ll 1$, fix $N>0$ large enough such that
		\begin{equation*}
			\sup_{0\leq t\leq T,\;0\leq s\leq T}|V(s)-v|\leq \|\nabla_x \phi_R^\e\|_{\infty}	|s-t|\leq C|s-t|\leq \frac{N}{2}  	.
		\end{equation*}
		Then,
		according to the   range of values of $v, v^{\prime },v^{\prime
			\prime }$,we divide
		\begin{align}\label{g421}
			\widetilde{\mathcal{G}}_4^2=\sum_{i=1}^{4}		\widetilde{\mathcal{G}}_{4, i}^2	,
		\end{align}
		where
		\begin{align*}
			\widetilde{\mathcal{G}}_{4, 1}^2	 &= \mathbf{1}_{\{|v|\ge N\}}\,\widetilde{\mathcal{G}}_4^2,
			&\qquad
			\widetilde{\mathcal{G}}_{4, 2}^2	&= \mathbf{1}_{\{|v|\le N,\; |v'|\ge 2N \;\text{or}\; |v'|\le 2N,\; |v''|\ge 3N\}}\,\widetilde{\mathcal{G}}_4^2,
			\\
			\widetilde{\mathcal{G}}_{4, 3}^2	&= \mathbf{1}_{\{|v|\le N,\; |v'|\le 2N,\; |v''|\le 3N,\; s-s_1\le \varepsilon \eta\}}\,\widetilde{\mathcal{G}}_4^2,
			&
			\widetilde{\mathcal{G}}_{4, 4}^2	 &= \mathbf{1}_{\{|v|\le N,\; |v'|\le 2N,\; |v''|\le 3N,\; s-s_1\ge \varepsilon \eta\}}\,\widetilde{\mathcal{G}}_4^2.
		\end{align*}
		Using the same argument as the corresponding term in \cite{GuoCMP2010}, we have	
		\begin{align}\label{g123}
			\sum_{i=1}^{3}	\widetilde{\mathcal{G}}_{4, i}^2\lesssim  \left(\frac{1}{N}+e^{- \frac{\eta }{8}N^{2}}+\eta\right)e^{-at^{\varrho}} \sup_{0 \le s \le t}\{e^{as^{\varrho}}\|\nabla h\|_{\infty}\}.
		\end{align}
		As in \eqref{kmnw}, $l^{\chi}_{w }(v, u)$ has possible integrable singularity of $\frac{%
			1}{|v-u|},$ we can choose $l_{N}(p, u)$ smooth with
		compact support such that
		\begin{equation*}
			\sup\limits_{|p|\leq3N}\int_{|u|\leq3N}\left| l_N(p, u)- \nu^{-1}{l^{\chi}_{ w}} (p, u)\right|\d u\leq\frac{1}{N}.
		\end{equation*}
		Splitting
		\begin{equation*}
			\begin{split}
				(V(s),v^{\prime })\,\nu^{-1}l^{\chi}_{w}(V(s_{1}),v^{\prime \prime
				})=\{	\nu^{-1} l^{\chi}_{w}(V(s),v^{\prime })-  l_{N}(V(s),v^{\prime
				})\} \nu^{-1}l^{\chi}_{w}(V(s_{1}),v^{\prime \prime })&  \\
				+\{ \nu^{-1}l^{\chi}_{w}(V(s_{1}),v^{\prime \prime })- l_{N}(V(s_{1}),v^{\prime \prime
				})\}  l_{N}(V(s),v^{\prime })+  l_{N}(V(s),v^{\prime })&
				\,  l_{N}(V(s_{1}),v^{\prime \prime }),
			\end{split}%
		\end{equation*}%
	and
		 using   \eqref{nus} and \eqref{eq:es:Kc:1}, we deduce
		\begin{align}\label{i4}
			\widetilde{\mathcal{G}}_{4, 4}^2	\lesssim     \frac{1}{N}  e^{-at^{\varrho}}
			\sup_{0\leq s\leq t}\{e^{a s^{\varrho}}\|\nabla h^{\varepsilon}(s)\|_{\infty}\}+\overline{\mathcal{G}}
			,
		\end{align}
		where $\overline{\mathcal{G}}$ is given by
		\begin{align}
			\overline{\mathcal{G}}=&\int_0^t\text{exp}
			\Big\{-\frac{1}{\e}\int_s^t\widetilde{\nu}( \tau)\text{d}\tau\Big\}\frac{\widetilde{\nu}(s)}{\e}
			\int_0^{s-\e \eta}  \text{exp}\Big\{-\frac{1}{\e}\int_{s_1}^s\widetilde{\nu}(s_1) (\tau')\text{d}\tau'\Big\} \frac{\widetilde{\nu}}{\e}   		\\\nonumber
			&\times
			\iint_{\{|v'|\leq 2N,| v''|\leq 3N\}}| l_{N}  (V(s),v')  l_{N}  (V(s_1),v'')  {\nabla h}^\varepsilon (s_1, X(s_1),v'')|\d s \text{d}s_1 \text{d}v'\text{d}v''.  	
		\end{align}
		Introducing a new variable, by \eqref{char}, we have
		\begin{equation*}
			y=X(s_{1})=X(s_{1};s, X(s;t, x, v),v^{\prime }),\quad V(s_{1})=V(s_{1};s, X(s;t, x, v),v^{\prime }) \label{vidav}.
		\end{equation*}%
		With \eqref{phi} and  $|v'|\leq 2N$, we have
		\begin{align*}
			&	y - X(s) = (s - s_1)v' - \int_{s_1}^{s} \int_{s'}^{s} \nabla_x \phi^\varepsilon\big(\tau, X(\tau; s, X(s;t, x, v), v')\big) \, d\tau \, \d s'
			\le C N (s - s_1).
		\end{align*}
		We now apply Lemma \ref{le:jac} to $X(s_{1};s, X(s;t, x, v),v^{\prime })$ with $%
		x=X(s;t, x, v),\tau =s_{1},t=s$. By (\ref{tezhengxianjielun2}), we can choose small but
		fixed $T_{0}>0$ such that for $s-s_{1}\geq \eta \varepsilon,$
		$
		|\frac{dy}{\d v^{\prime }}|\geq\frac{|s-s_1|^3}{2}\geq \frac{\eta ^{3}\varepsilon ^{3}}{2}\,
		$.
		Noting  $l_{N}(V(s),v^{\prime })\,l_{N}(V(s_{1}),v^{\prime \prime })$ is
		bounded, and combining
		\begin{align}\label{jh}
			&  	\iint_{\{|v'|\leq 2N,| v''|\leq 3N\}}
			\, |\nabla h^{\varepsilon}(s_{1},X(s_{1}),v'')|
			\d v'\d v''	    \\\nonumber
			\le &  	\frac{C_{N,\eta}}{\eta ^{3/2}\varepsilon^{3/2}}\{\int_{\{\left|y-X(s)\right|\leq CN(s-s_1),| v''|\leq 3N\}}
			\, |\nabla h^{\varepsilon}(s_{1},y, v'')|^2
			\d y \, \d v'' \}^{\frac{1}{2}}   \\\nonumber
			\leq&	\frac{C_{N,\eta}}{\eta ^{3/2}\varepsilon^{3/2}}
			\left\{\int_{\{\left|y-X(s)\right|\leq CN(s-s_1),| v''|\leq 3N\}}
			\left|\nabla \left(w\frac{\sqrt{\mu}f^{\varepsilon}}{\sqrt{\mathcal{M}}}\right)(s_{1},y, v'')\right|^{2}
			\d y  \d v''
			\right\}^{1/2}\\\nonumber
			\leq&	\frac{C_{N,\eta}}{ \varepsilon^{3/2}}\|f^\e(s_1)\|_{H^1}
			,
		\end{align}
		then, by \eqref{nus}, it holds that
		\begin{align}\label{i411}
			\overline{\mathcal{G}}
			&\lesssim
		e^{-at^{\varrho}}		\frac{C_{N,\eta}}{ \varepsilon^{3/2}}\sup_{0 \le s \le t}\{e^{as ^{\varrho}}\|f^\e(s )\|_{H^1} \}.
		\end{align}
		
		Combining  \eqref{i411},   \eqref{i4},  \eqref{g123},  and  \eqref{jh},
		we have
		\begin{align}\label{g421r}
			\widetilde{\mathcal{G}}_{4 }^2	 \lesssim 	\left( \frac{1}{N}+e^{- \frac{\eta }{8}N^{2}}+\eta \right) e^{-at^{\varrho}}
			\sup_{0\leq s\leq t}\{e^{a s^{\varrho}}\|\nabla h^{\varepsilon}(s)\|_{\infty}\}+	 	e^{-at^{\varrho}}		\frac{C_{N,\eta}}{ \varepsilon^{3/2}}\sup_{0 \le s \le t}\{e^{as ^{\varrho}}\|f^\e(s )\|_{H^1} \}.
		\end{align}
		Substituting \eqref{g421r}  into \eqref{g42} and combining \eqref{g41}, \eqref{g4} to obtain
		\begin{align}
			\mathcal{G}_4 \lesssim &e^{-at^{\varrho}}\|\nabla h^\e(0)\|_{\infty}	+\left(T_0 +\e+\e^{\gamma+3}+\frac{1}{N}+e^{- \frac{\eta }{8}N^{2}}+\eta \right)e^{-at^{\varrho}}  \sup_{0\le s\le t}\Big\{ e^{ {a}s^{\varrho}}\|  {\nabla} h^\e  \|_{\infty }\Big\} \\\nonumber
			&+\left(1+\e+\e^{\gamma+3}\right)e^{-{a}t^\varrho}\sup_{0\le s\le t}\Big\{ e^{{a}s^\varrho}  \|h^{\varepsilon}\|_{ {\infty} } \Big\}	+ 	e^{-at^{\varrho}}		\frac{C_{N,\eta}}{ \varepsilon^{3/2}}\sup_{0 \le s \le t}\{e^{as ^{\varrho}}\|f^\e(s )\|_{H^1} \}+\e.
		\end{align}
		\par
		Putting the above estimates on $\mathcal{G}_1$--$\mathcal{G}_8$ into (\ref{eq:Dh1}),   we derive
		that
		\begin{align*}
			\|\nabla h^\e(t)\|_{\infty} \lesssim & e^{-at^{\varrho}}\|\nabla h^\e(0)\|_{\infty}	+\left(T_0 +\e+\e^{\gamma+3}+\frac{1}{N}+e^{- \frac{\eta }{8}N^{2}}+\eta \right)e^{-at^{\varrho}}  \sup_{0\le s\le t}\Big\{ e^{ {a}s^{\varrho}}\|  {\nabla} h^\e(s)  \|_{\infty }\Big\} \\\nonumber
			&+\left(1+\e+\e^{\gamma+3}\right)e^{-{a}t^\varrho}\sup_{0\le s\le t}\Big\{ e^{{a}s^\varrho}  \|h^{\varepsilon}(s)\|_{ {\infty} } \Big\}+e^{-at^{\varrho}}		\frac{C_{N,\eta}}{ \varepsilon^{3/2}}\sup_{0 \le s \le t}\{e^{as ^{\varrho}}\|f^\e(s )\|_{H^1} \}+\e.			
		\end{align*}
		Multiplying the above inequality by \( e^{a s^{\varrho}} \),   taking the supremum, choosing $T_0, \e,  \frac{1}{N}, \eta$ small enough, letting \( s = T_0 \) on the left-hand side and multiplying both sides by \( \varepsilon^{3/2} e^{-a T_0^{\varrho}} \), with $e^{-{a}{{T_0}^\varrho}}\leq \frac{1}{2}$, we obtain
		\begin{align}\label{dh}
			&	\e^{3/2}  \| {\nabla h}^\varepsilon (T_0)\|_{ {\infty}}
			\leq \frac{1}{2}\e^{3/2} \|\nabla h_0\|_{\infty} +C\e^{3/2}\left(1+\e+\e^{\gamma+3}\right)  \sup_{0\le s\le t}\Big\{  \|h^{\varepsilon}(s)\|_{ {\infty} } \Big\}	 +  {C_{N,\eta}}\sup_{0\le s\le t} \{\|f^\e(s )\|_{H^1}\}+C\e^{5/2}.
		\end{align}	
		\emph{\bf Step 3. Iteration.}\;
		When $0\leq t \leq T_0\ll 1$, there exists a sufficiently large constant $C_1$,
		adding $C_1  \eqref{4.16}$ and  \eqref{dh}, we have		
		\begin{align}\label{t0t}
			\e^{3/2}  \| {  h}^\varepsilon (T_0)\|_{W^{1, \infty}} \leq \frac{1}{2} \e^{3/2}  \| {  h}^\varepsilon (0)\|_{W^{1, \infty}} +C\sup_{0\le s\le t} \{\|f^\e(s )\|_{H^1}\}+C\e^{5/2}.
		\end{align}		
		When $T_0\leq t \leq \e^{-1/2}$, there exists a positive integer $n$ such that $t = nT_0 + \tilde{\tau}, 0 \leq \tilde{\tau}   \leq  T_0.$ Using the same iteration method in \cite[proof of roposition 5.1]{GuoCMP2010} for  \eqref{t0t},   we obtain \eqref{hin}.
	\end{proof}

	\section{Proof of the Main Results for the Remainder }
	\noindent
\begin{proof}[\bf The proof of Theorem~\ref{t1}]
		Denote
		\begin{align*}
			\mathcal{E}(t)=& \| \sqrt{{\theta }_0} f^{\varepsilon} \|^2+	  \| \sqrt{\e {\theta }_0} \nabla_xf^{\varepsilon} \|^2+   \| \sqrt{\e} \nabla_v(\mathbb{P}^{\bot}f^\varepsilon)\|^2+ \| \sqrt{\e} \langle v \rangle^{2-2\gamma}\mathbb{P}^{\bot}f^\varepsilon\|^2+\|\e^{5/4}\langle v \rangle^{- \gamma }\nabla_x  f^\e\|^2\\
			& + \| \nabla \phi_R^\varepsilon \| ^2 +\| \sqrt{e^{\phi_0}  e^{\varepsilon^k \phi_R^\varepsilon}} \phi_R^\varepsilon \| ^2.
		\end{align*}
		By the a  $\mathit{priori}$  assumption \eqref{pri}, we have
		\begin{align*}
			\mathcal{E}(t)	\sim \|f\|^2
			+ \|\sqrt{\e}\nabla f\|^2
			+ \|\sqrt{\e}\langle v\rangle^{2-2\gamma} f\|^2+\|\e^{5/4}\langle v \rangle^{- \gamma }\nabla_x  f^\e\|^2 +\|\phi_R^\varepsilon\|_{H^1}^2.
		\end{align*}
		Substituting \eqref{hin} into \eqref{L2},  we have
		\begin{align*}
			\frac{\d}{\d t}\mathcal{E} (t)
			\lesssim &	\e^{5/2}\left\{\e^{3/2}  \| {  h}^\varepsilon (0)\|_{W^{1,   \infty}} +\sup_{0\le s\le t} \{\|f^\e(s )\|_{H^1}\}+\e^{5/2}
			\right\}	\| {f}^\varepsilon  \| _{H^1}  +\left((1+t)^{-16/15}+\sqrt{\e}\right)\mathcal{E}  (t)  +\sqrt{\e} \\
			\lesssim&\left((1+t)^{-16/15}+\sqrt{\e}\right)\mathcal{E}  (t)+ \e^{11/2 }\| {h}^\varepsilon ( 0) \|^2_{W^{1,  \infty}}+ \sup_{0\le s\le  t}
			\e^{5/2}	\| {f}^\varepsilon (s)\| ^2_{H^1}    +\sqrt{\e}.
		\end{align*}
		Applying Gr\"onwall's inequality,   with $\left|\int_0^{\e^{-1/2}} (1+t)^{-16/15} dt \right|\leq C$ and $\left|\int_0^{\e^{-1/2}} \sqrt{\e} dt \right|=1$,  we have
		\begin{align*}
			\mathcal{E} (t)
			\lesssim& \mathcal{E} (0)+\e^{5} \| {h}^\varepsilon ( 0) \|^2_{W^{1,  \infty}} +\e \sup_{0\le s\le  t}	\|\sqrt{\e} {f}^\varepsilon (s)\| ^2_{H^1}+1,
		\end{align*}
		where ${\varepsilon }
		\sup_{0 \le s \le t}
		\|\sqrt{\e}f^\varepsilon(s)\|_{H^1}^2$
		can be absorbed by $	\mathcal{E}(t)$. Hence,we obtain
		\begin{align*}
			&\sup_{0\leq t \leq \e^{-1/2}}\left\{	 \|f^\e(t)\| + \| \sqrt{\e}\nabla f^\e(t)\|  + \|\sqrt{\e}\langle v\rangle^{2-2\gamma} f^\e(t)\| +\|\e^{5/4}\langle v \rangle^{- \gamma }\nabla_x  f^\e\|^2 +\|\phi_R^\varepsilon(t)\|_{H^1}\right\}\\
			\lesssim& 1+ \|f^\e(0)\| + \| \sqrt{\e}\nabla f^\e(0)\| +  \|\sqrt{\e}\langle v\rangle^{2-2\gamma} f^\e(0)\| + \|\e^{5/4}\langle v\rangle^{-\gamma} \nabla_xf^\e(0)\| +\|\phi_R^\varepsilon(0)\|_{H^1}+\e^{5/2} \| {h}^\varepsilon ( 0) \|_{W^{1,  \infty}}
			.
		\end{align*}
		Combining $\eqref{hin}$,  we obtain \eqref{t1s}.
	\end{proof}
	\noindent
	\begin{proof}[\bf The proof of Theorem~\ref{t2}]
		Using the same arguments as in Step~1,   Step~2 and Step~3  of the proof of the  Proposition~\ref{L2},
		we note that the term corresponding to \eqref{dis1} can be controlled directly by
		\begin{align*}
			\big\langle \nabla_x \mathbb{P}^{\bot}f^\varepsilon,  \;\varepsilon \nabla_v (\mathbb{P}^{\bot}f^\varepsilon) \big\rangle
			&\lesssim \varepsilon \|\nabla_v (\mathbb{P}^{\bot}f^\varepsilon)\|_{\nu}^2
			+ \varepsilon \|\langle v\rangle^{-\gamma} \nabla_x \mathbb{P}^{\bot}f^\varepsilon\|_{\nu}^2 \lesssim \varepsilon \|\nabla_v (\mathbb{P}^{\bot}f^\varepsilon)\|_{\nu}^2
			+ \varepsilon \|\nabla_x \mathbb{P}^{\bot}f^\varepsilon\|_{\nu}^2.
		\end{align*}
		Therefore,   Step~4 in the proof of Proposition~\ref{L2} is unnecessary,   and it suffices to take $\gamma=0$ directly in Step~5.
		Let $0 \le \gamma \le 1$ and suppose $0 \le t \le \varepsilon^{-1/2}$,   we   obtain
		\begin{align}\label{l2a}
			&\frac{\d}{\d t}\left\{
			\|\sqrt{\theta_0} f^{\varepsilon}\|^2
			+\|\sqrt{\varepsilon \theta_0}\nabla_x f^\varepsilon\|^2
			+\|\sqrt{\varepsilon}\nabla_v (\mathbb{P}^{\bot}f^\varepsilon)\|^2
			+\|\sqrt{\varepsilon}\langle v\rangle^2 \mathbb{P}^{\bot}f^\varepsilon\|^2
			+\|\nabla \phi_R^\varepsilon\|^2
			+\| \sqrt{e^{\phi_0} e^{\varepsilon^k \phi_R^\varepsilon}} \phi_R^\varepsilon \|^2
			\right\} \nonumber\\
			&\quad
			+ \frac{\delta_1 \theta_M}{\varepsilon} \|\mathbb{P}^{\bot} f^\varepsilon\|_{\nu}^2
			+ \delta_1 \theta_M \|\nabla_x (\mathbb{P}^{\bot}f^\varepsilon)\|_{\nu}^2
			+ \delta_1 \|\nabla_v (\mathbb{P}^{\bot}f^\varepsilon)\|_{\nu}^2
			+ \delta_1 \|\langle v\rangle^2 \mathbb{P}^{\bot}f^\varepsilon\|_{\nu}^2 \nonumber\\
			\lesssim &
			\varepsilon^4 \|h^\varepsilon\|_{W^{1,  \infty}} \|f^\varepsilon\|_{H^1}
			+\big((1+t)^{-16/15} + \sqrt{\varepsilon}\big)
			\big( \|f^\varepsilon\|^2 + \|\sqrt{\varepsilon}\nabla f^\varepsilon\|^2 +\|\sqrt{\varepsilon}\langle v\rangle^2 \mathbb{P}^{\bot}f^\varepsilon\|^2\big)
			+ \sqrt{\varepsilon}.
		\end{align}
		\par	
		We now apply the same procedure as in the proof of Proposition~\ref{softin} to derive the corresponding estimate for the hard potentials case.
		The main difference lies in the treatment of the collision frequency $\nu$ and the operator $\mathcal{K}_w$.
		See Lemma~\ref{nuinter} for $\nu$.
		The essential distinction concerns $\mathcal{K}_w$:
		in the soft potentials case,   $\mathcal{K}_w$ is singular and must be decomposed as
		$
		\mathcal{K}_w = K^\chi_w + K^{1-\chi}_w
		$
		using a cutoff function,   whereas in the hard potentials case $\mathcal{K}_w$ is non-singular.
		Hence,   the term involving $\mathcal{K}_w$ can be estimated directly by \eqref{lhard}.
		Let $0 \le \gamma \le 1$ and suppose $0 \le t \le \varepsilon^{-1/2}$,   we obtain
		\begin{equation}\label{hinH}
			\varepsilon^{3/2}
			\sup_{0 \le t \le \varepsilon^{-1/2}}
			\| h^\varepsilon(t) \|_{W^{1,  \infty}}
			\lesssim
			\varepsilon^{3/2} \| h^\varepsilon(0) \|_{W^{1,  \infty}}
			+ \sup_{0 \le s \le t} \|f^\varepsilon(s)\|_{H^1}
			+ \varepsilon^{5/2}.
		\end{equation}
		\par
		Using the same arguments as in the proof of Theorem~\ref{t1},   based on \eqref{l2a} and \eqref{hinH},   we complete the proof of Theorem~\ref{t2}.
	\end{proof}

	\appendix
	\section{Estimates on the Time-Dependent Coefficients }
	In this section, for \(i = 0, \ldots, 2k-1\), we derive time-dependent estimates for \(\rho_i\), \(u_i\), \(\phi_i\), and \(F_i\). These estimates play a crucial role in the \(H^1_{x,v}\) and \(W^{1,\infty}_{x,v}\) bounds.

	\begin{lemma} \label{F0th}
		Let the  integer $s_1\ge 5$ and $n_0$ be an given positive constant. There exists $\delta_0>0$ such that if the initial perturbation
		$(\rho_0^{\mathrm{in}}(x),  u_0^{\mathrm{in}}(x))$ satisfies $\nabla\times u_0^{\mathrm{in}}(x)=0 $ and
		\begin{align*}
			\|(\rho_0^{\mathrm{in}}(x)-n_0,   u_0^{\mathrm{in}}(x))	\|_{H^{2s_1+1}} +\|(\rho_0^{\mathrm{in}}(x)-n_0,   u_0^{\mathrm{in}}(x))	\|_{W^{s_1+\frac{12}{5},  \frac{10}{9}}} \leq 2\delta_0,
		\end{align*}
		then \eqref{ep}  admits a global solution $(\rho_0,   u_0,  \phi_0)(t,   x)$. Moreover,   provided $\delta_0$ is sufficiently small,   for every integer $0\leq s_2 < s_1-1 $,  it holds that
		\begin{equation} \label{decay}
			\|({\rho_0-n_0},  u_0,  \phi_0 )\|_{W^{{s_2},  \infty}}\lesssim\delta_0 (1+t)^{-\frac{16}{15} },  \quad   \|({\rho_0-n_0},  u_0,  \phi_0)\|_{H^{2s_2 }}\lesssim {\delta_0}.
		\end{equation}
		and
		\begin{align}\label{ttphi0}
			\|\partial_{t}\phi_0 \|_{\infty}\lesssim \delta_0 (1+t)^{-16/15}.
		\end{align}
	\end{lemma}
	\begin{proof}
		As \cite[Theorem 1]{GuoCMP2011},    \eqref{ep}  admits a global solution $(\rho_0,   u_0,  \phi_0)(t,   x)$ and  one has
		\begin{align}\label{guo}
			&	\sup_{t\ge0}
			\!\Big(
			\||\nabla|^{-1}(1-\Delta)^{s_1+\frac12}({\rho_0-n_0},  u_0)(t,   x)\|_{L^2}
			+(1+t)^{\frac{16}{15}}
			\|(1-\Delta)^{\frac{s_1}{2}}({\rho_0-n_0},  u_0)(t)\|_{L^{10}}
			\Big)
			\le 2\delta_0.
		\end{align}		
		For the    integer $0\leq s_2 < s_1-1 $,   by the  Sobolev--Morrey Embedding in $\mathbb{R}^3$,     and  \eqref{guo},  we have
		\begin{align*}
			&	\|\nabla_x^{s_2}({\rho_0-n_0},  u_0) \|_{ \infty}\lesssim\|\nabla_x^{s_2}({\rho_0-n_0},  u_0) \|_{W^{1,   10}  }\lesssim    	\|(1-\Delta)^{\frac{s_1}{2}}({\rho_0-n_0},  u_0) \| _{L^{10}}
			\lesssim
			\delta_0\,  (1+t)^{-\frac{16}{15} }.
		\end{align*}
		Then,    with \cite[Theorem 1.1]{DuanSIAM2009},  we have
		\begin{align*}
			\|\phi_0\|_{W^{{s_2},  \infty}}\lesssim	 \|\rho_0(t,   x)-n_0\|_{W^{{s_2},  \infty}}\lesssim \delta_0(1+t)^{-\frac{16}{15} }.
		\end{align*}
		Using elliptic estimates for $\eqref{ep}_3$,  let $n_0=1$ and  applying the Taylor expansion
		$e^{\phi_0}-1-\phi_0= {O}(\phi_0^2)$,   we have
		\begin{align*}
			\|\phi_0\|_{H^{2s_2+2}} \leq& \|\phi_0-e^{\phi_0}+1\|_{H^{2s_2}}+	 \|{\rho_0-1}\|_{H^{2s_2}}\\\nonumber
			\lesssim&\|\phi_0   \|_{\infty}\|\phi_0    \|_{H^{2s_2}}
			+	 \|{\rho_0-1}\|_{H^{2s_2}}\\\nonumber
			\lesssim&\delta_0(1+t)^{-\frac{16}{15} }  \|\phi_0    \|_{H^{2s_2}}+	 \|{\rho_0-1}\|_{H^{2s_2}},
		\end{align*}
		where $C\delta_0(1+t)^{-\frac{16}{15} } \|\phi_0    \|_{H^{2s_2}}$ can be absorbed by the left side,   and by  \eqref{guo},  it holds that
		\begin{align*}
			\|\phi_0\|_{H^{2s_2+2}} \lesssim \|{\rho_0-1}\|_{H^{2s_2}} \lesssim\delta_0.
		\end{align*}
		Thus,   we prove \eqref{decay}. Next,   we prove \eqref{ttphi0}. Taking the time derivative of $\eqref{ep}_3$,  using the  elliptic estimates and Sobolev--Morrey embedding in \(\mathbb{R}^3\),         $\eqref{ep}_1$ and \eqref{decay},   we conclude
		\begin{align*}
			\|\partial_{t}\phi_0\|_{\infty}\lesssim 	\|\partial_{t}\phi_0\|_{W^{1,   10}} \lesssim  	\|\partial_{t}\rho_0\|_{W^{3,   10}}	\lesssim  \|\nabla (\rho_0,   u_0)\|_{W^{3,   10}}	   \lesssim \delta_0 (1+t)^{-16/15}.
		\end{align*}
		
	\end{proof}
	\begin{lemma}\label{fi}
		Let $s \geq 0$,  given the initial data
		$\rho^{\mathrm{in}}_{n+1}(0,   x),  u^{\mathrm{in}}_{n+1}(0,   x),  %
		{\theta }^{\mathrm{in}}_{n+1}(0,   x)\in H^{s}$,   the linear system \eqref{eqF_k} is well-posed in $H^s$. Moreover,     	let $1\le i\le 2k-1$,   for $t\ge0$,   it holds that
		\begin{align}
			\label{fir1}&\|\left(\rho_{i},  u_{i},  \theta_i	\right) \|_{H^s}\lesssim (1+t)^{i-1}.\\
			\label{fir2}&\|\phi_i \|_{H^{s+2}}\lesssim (1+t)^{i-1},  \quad \|\partial_{t}\phi_i \|_{H^{s+1}}\lesssim (1+t)^{i-1}.\\
			\label{fir3}& |F_i|\lesssim(1+t)^{i-1}(1+|v|^{3i})\mu\;,    |\nabla_vF_i|\lesssim(1+t)^{i-1}(1+|v|^{3i+1})\mu,  \; \;|\nabla_xF_i|\lesssim(1+t)^{i-1}(1+|v|^{3i+2})\mu,   \nonumber\\
			&|\nabla_v\nabla_vF_i|\lesssim(1+t)^{i-1}(1+|v|^{3i+2})\mu,  \;
			|\nabla_x\nabla_vF_i|\lesssim(1+t)^{i-1}(1+|v|^{3i+3})\mu.  	
		\end{align}		
		
	\end{lemma}
	\begin{proof}
		The well-posedness can be established as an immediate consequence of the linear theory;cf.
		\cite{evans1998partial}. We prove
		\eqref{fir1}--\eqref{fir3} by mathematical induction.
		We begin with the case $F_1$.
		For $n=0$, the system \eqref{eqF_k} can be written in the symmetric hyperbolic form:
		\begin{equation}\label{U}
			A_0\{\partial_tU_1+V_1\}+\sum_{i=1}^3A_i\partial_iU_1 +\mathcal{B}U_1=\mathcal{F}_1
		\end{equation}		
		where  	
		\[
			\begin{array}{l}
			U_{i} = \begin{pmatrix} \rho_{i} \\ (u_{i})^t \\ \theta_{i} \end{pmatrix},   V_{i} = \begin{pmatrix} 0 \\ (\nabla\phi_{i})^t \\ 0 \end{pmatrix},   A_0 = \begin{pmatrix} (\theta_0)^2 & 0 & 0 \\ 0 & (\rho_0)^2\theta_0\,  \mathbb{I} & 0 \\ 0 & 0 & \dfrac{(\rho_0)^2}{6} \end{pmatrix}, \\
              A_i = \begin{pmatrix} (\theta_0)^2 u_0^i & \rho_0(\theta_0)^2 e_i & 0 \\ \rho_0(\theta_0)^2 (e_i)^t & (\rho_0)^2\theta_0 u_0^i\,  \mathbb{I} & \dfrac{(\rho_0)^2\theta_0}{3}(e_i)^t \\ 0 & \dfrac{(\rho_0)^2\theta_0}{3}e_i & \dfrac{(\rho_0)^2 u_0^i}{6} \end{pmatrix}.
		\end{array}
		\]
		$(\cdot)^t$ denotes the transpose of row vectors,   $e_i$'s for
		$i=1,   2,   3$ are the standard unit (row) base vectors in
		$\mathbb{R}^3$,   $\partial_j(\cdot)$ denotes the derivative with
		respect to $x_j$ with implicit summation over $j=1,   2,   3$,  and $\mathbb{I}$ is the $3\times 3$ identity matrix.
		$\eqref{eqF_k}_4$ equals to
		\begin{align}\label{phi10}
			\Delta \phi _{1}= e^{\phi _{0}} \phi _{1}-\rho_1.
		\end{align}
		By    elliptic estimates to \eqref{phi10} with  $\|1-e^{\phi_0}\|_{\infty}\lesssim \|\phi_0\|_{\infty}\lesssim \delta_0$ from \eqref{decay},   we have
		\begin{align}
			\|\phi_1\|_{H^{s+2}}\lesssim \|\rho_1\|_{H^s}.	
		\end{align}
		Applying $\partial_t$ derivative for  \eqref{phi10},
		using elliptic estimates,   with  $\eqref{eqF_k}_1$,   \eqref{decay},    and \eqref{ttphi0},  we have
		\begin{align*}
			\|\partial_{t} \phi_1\|_{H^{s+1}}\lesssim& \| \partial_{t}\phi_0e^{\phi_0}\phi_1\|_{H^{s-1}}+\|  (1-e^{\phi_0})\partial_{t}\phi_1\|_{H^{s-1}} + \| \partial_{t}\rho_1\|_{H^{s-1}}
			\\
			\lesssim&(1+t)^{-16/15}\|\phi_1\|_{H^{s}}+\delta_0(1+t)^{-16/15}\|\partial_{t}\phi_1\|_{H^{s}}+ \|(\rho_1,   u_1 )\|_{H^s}.
		\end{align*}
		Since  $\delta_0 \ll 1$,  we have
		\begin{align}\label{2. 12}
			\|\partial_{t} \phi_1\|_{H^{s+1}} \lesssim \|(\rho_1,   u_1 )\|_{H^{s}}+(1+t)^{-16/15}\|\phi_1\|_{H^{s}}.
		\end{align}
		\par		Taking the $L^2$ inner product of \eqref{U} with $U_1$ gives
		\begin{align}\label{1. 24}
			\frac{1}{2}\frac{\d}{\d t}\|\sqrt{A_0}U_1\|^2+\langle A_0 V_1,   U_1 \rangle =\langle \frac{1}{2}\partial_tA_0 U_1+ \sum_{j=1}^3A_j\partial_jU_1 +\mathcal{B}U_1-\mathcal{F}_1,   U_1 \rangle.
		\end{align}
		By integration by parts and $\nabla_x \cdot u_1=\frac{-\partial_{t}\rho_1-u_0\cdot\nabla_x  \rho_1-\rho_1\nabla_x \cdot u_0}{\rho_0}$ from \eqref{eqF_k},  we obtain
		\begin{align*}
			&\langle A_0 V_1,   U_1 \rangle=\langle \rho_0^2{\theta }_0,  \nabla_x\phi_1\cdot u_1 \rangle=-\langle \rho_0^2{\theta }_0\phi_1,  \nabla_x\cdot u_1 \rangle-\langle \nabla_x\left(\rho_0^2{\theta }_0\right)\cdot u_1,   \phi_1 \rangle=\langle \rho_0{\theta }_0\phi_1,  \partial_t \rho_1 \rangle
			+ \mathcal{U}_1^1,
		\end{align*}
		where,
		\begin{align*}
			\langle \rho_0{\theta }_0\phi_1,  \partial_t \rho_1 \rangle=	 \partial_t\langle \rho_0{\theta }_0\phi_1,   \rho_1 \rangle-\langle \rho_0{\theta }_0\partial_t\phi_1,   \rho_1 \rangle-\langle\partial_t\left( \rho_0{\theta }_0\right)\phi_1,   \rho_1 \rangle.
		\end{align*}
		Using \eqref{phi10}  and integration by parts,   we obtain
		\begin{align*}
			-\langle \rho_0{\theta }_0\partial_t\phi_1,   \rho_1 \rangle=&\langle \rho_0{\theta }_0\partial_t\phi_1,   \Delta \phi _{1}-e^{\phi _{0}} \phi _{1} \rangle=-\frac{1}{2}\partial_t \left\|\sqrt{\rho_0{\theta }_0e^{\phi _{0}} }\phi_1\right\|^2 -\frac{1}{2}\partial_t \left\|\sqrt{\rho_0{\theta }_0  }\nabla\phi_1\right\|^2
			\\\nonumber
			&  +\frac{1}{2}\langle \partial_t\left(\rho_0{\theta }_0\right),   |\nabla\phi _{1}|^2  \rangle+\frac{1}{2}\langle \partial_t\left(\rho_0{\theta }_0e^{\phi_0}\right),   | \phi _{1}|^2  \rangle - \langle \nabla\left(\rho_0{\theta }_0\right)\partial_t\phi_1,   \nabla\phi _{1}  \rangle.
		\end{align*}
		Based on the above decomposition of $\langle A_0 V_1,   U_1 \rangle$,   	we have \begin{align}\label{avu}
			\langle A_0 V,   U_1 \rangle=\frac{\d}{\d t}	\widetilde{\mathcal{E}_1}^2-\mathcal{U}_1,  		
		\end{align}
		where   $\widetilde{\mathcal{E}_1}^2= \langle \rho_0{\theta }_0\phi_1,   \rho_1 \rangle	 -\frac{1}{2} \|\sqrt{\rho_0{\theta }_0e^{\phi _{0}} }\phi_1\|^2 -\frac{1}{2} \|\sqrt{\rho_0{\theta }_0  }\nabla\phi_1\|^2$ and
		\begin{align*}
			\mathcal{U}_1=& \mathcal{U}_1^1+\frac{1}{2}\langle \partial_t\left(\rho_0{\theta }_0\right),   |\nabla\phi _{1}|^2  \rangle+\frac{1}{2}\langle \partial_t\left(\rho_0{\theta }_0e^{\phi_0}\right),   | \phi _{1}|^2  \rangle - \langle \nabla\left(\rho_0{\theta }_0\right)\partial_t\phi_1,   \nabla\phi _{1}  \rangle.
		\end{align*}
		Applying \eqref{phi10} and integration by parts,   we have
		\begin{align*}
			\langle \rho_0{\theta }_0\phi_1,   \rho_1 \rangle=\langle \rho_0{\theta }_0\phi_1,   e^{\phi_0}\phi_1-\Delta\phi_1
			\rangle=\|\sqrt{\rho_0{\theta }_0e^{\phi_0}}\phi_1\|^2+\|\sqrt{\rho_0{\theta }_0 }\nabla\phi_1\|^2+\langle \nabla\left(\rho_0{\theta }_0\right)\phi_1,   \nabla\phi_1 \rangle.
		\end{align*}
		By \eqref{decay} and young's inequality,   we have
		$
		\left|\langle \nabla\left(\rho_0{\theta }_0\right)\phi_1,   \nabla\phi_1 \rangle\right|\leq \frac{C\delta_0 	 }{2} \|\phi_1\|^2+\frac{C\delta_0 	 }{2} \|\nabla\phi_1\|^2.
		$
		Since $\rho_0,   {\theta }_0$   is a positive constant and $\delta_0\ll 1$,   we have
		\begin{align}\label{1. 29}
			\widetilde{\mathcal{E}_1}^2\sim \|\phi_1\|^2_{H^1}.
		\end{align}
		Substituting \eqref{avu} into \eqref{1. 24},
		denoting $\mathcal{E}_1^2=\frac{1}{2}\|\sqrt{A_0}U_1\|^2+\widetilde{\mathcal{E}_1}^2$,   $\|\sqrt{A_0}U_1\|^2\sim \| U_1\|^2$,   combining \eqref{1. 29},   we obtain
		\begin{align}\label{eup}
			\mathcal{E}_1^2\sim  {U}^2_1	 +\|\phi_1\|^2_{H^1}.
		\end{align}
		$\mathcal{U}_1$ in \eqref{avu}   was moved to the left side of  \eqref{1. 24},    $\partial_i A_i,  \mathcal{B}$ and $\mathcal{F}_1$ consist of $\rho_0,     u_0,   \theta_0$ and their first spatial derivatives,   using integration by parts,   \eqref{eqF_k},     \eqref{decay},   \eqref{ttphi0} and \eqref{2. 12},    we have
		\begin{align*}
			\frac{\d}{\d t}	\mathcal{E}_1^2\lesssim&\left(\|\partial_t A_0\|_{\infty}+\sum_{i=1}^{3}\|\partial_i A_i\|_{\infty}+\|\mathcal{B}\|_{\infty}\right)\|U_1\|^2+\|\mathcal{F}_1\|\|U_1\|+ \mathcal{U}_1
			\\
			\lesssim& \|\nabla_x( \rho_0,   u_0)\|_{\infty}\left(\|U_1\|^2+ \|U_1\|\right)+(\|(\nabla \rho_0,   u_0)\|_{W^{1,  \infty}}+\|\partial_{t} \phi_0\|_{\infty})(\|(\rho_1,   u_1)\|^2+	\| \phi_1\|_{H^1}^2+	\| \partial_{t}\phi_1\| ^2) \\
			\lesssim&\ (1+t)^{-16/15}\left(\|U_1\|^2+ \|U_1\|\right)+\ (1+t)^{-16/15}\|\phi_1\| _{H^1}^2
			.
		\end{align*}
		Then,
		$
		{\mathcal{E}_1}\frac{\d}{\d t}\left( {\mathcal{E}_1}+1\right)\lesssim	 (1+t)^{-16/15} {\mathcal{E}_1}\left( {\mathcal{E}_1}+1\right).
		$
		Applying Gr\"onwall's inequality and combining  \eqref{eup}  yield
		\begin{align}\label{us0}
			\| U_1\|   +\|\phi_1\|_{H^1}\leq C.
		\end{align}		
		Applying \( \partial^\alpha (|\alpha|\leq s) \) to \eqref{U},  taking inner product with \( \partial^\alpha U_1 \),    using the same method as estimating   the 0--th derivative \eqref{us0},   summing over all
		$
		\alpha
		$  yields
		\begin{align}\label{us}
			\| U_1\|_{H^{s}}+\|\phi_1\|_{H^{s+1}}\leq C,
		\end{align}
		which implies \eqref{fir1} holds for $i=1$.
		By elliptic estimate to \eqref{phi10},  it also holds  $\|\phi_1\|_{H^{s+2}}\lesssim \|\rho_1\|_{H^{s}}\lesssim \|U_1\|_{H^{s}}\leq C$.	Substitute  above results into  \eqref{2. 12},  we have $\|\partial_t\phi_1\|_{H^{s+1}}\lesssim 1$,  which completes the proof of \eqref{fir2} for $i=1$.
		\par
		Substituting \eqref{us} into   \eqref{F_1},   it holds that $|\sqrt{\mu}\mathbb{P}(F_1/\sqrt{\mu})|\lesssim (1+|v|^2)\mu$.
		Combining \eqref{i-p} for $n=0$, $L^{-1}$ preserves decay in $v$ \cite{Caflisch1980} and \eqref{decay},  we have $\sqrt{\mu}
		\bigl| \mathbb{P}^{\bot}\!\left(\frac{F_{1}}{{\sqrt{\mu}}}\right)\bigr|
		\;\le\;
		\left\|\left(\nabla_x \rho_0,    \nabla_x u_0,   \nabla_x \phi_0 \right)\right\|_{\infty}
		\,   (1+|v|^{3})  { {\mu}}\,  \lesssim (1+|v|^{3})  { {\mu}}.
		$ Thus
		$
		|F_1|\lesssim(1+|v|^{3})  {\mu}.
		$
		$	|\nabla_v \mu| \lesssim (1 + |v| ) \mu,     |\nabla_x \mu| \lesssim (1 + |v|^2) \mu$,    then we can prove  \eqref{fir3} holds for $i=1$.
		
		\par	 	
		Suppose that \eqref{fir1}--\eqref{fir3} hold for $1 \le i \le n$. Adopting the same approach as in \cite[Lemma 2.2]{GuoCMP2010}, we can show that \eqref{fir1}--\eqref{fir3} hold for \(i = n+1\).
	\end{proof}
	\section{Other Preliminary Lemmas}
	The operators $ {L}_M  $ is defined in \eqref{delm}.
	We   list the following results of $ \L_M$  :
	\begin{lemma}[\cite{Duan2012VlasovPoissonBoltzmann, strainarma2008}] \label{lm}
		For any   multi-index  $\b$ with  $|\b|\geq 0$,   and   $\lambda_1\geq 0$,
		we have  \\
		$(1)$ If $-3< \gamma< 0$,   then for any  $\eta \geq 0$,   there exists $C_{\eta} > 0$ such that
		\begin{align}
			\label{es-energy-linear:1}
			\left\langle \langle v \rangle ^{2\lambda_1}  \partial_v^\b( 		\L_Mf),    \partial_v^\b f \right\rangle \geq \;& \|  \langle v \rangle ^{\lambda_1}
			\partial_v^{\b} f\|_{\nu}^2
			-\eta\sum_{|\b'|\leq |\b|} \| \langle v \rangle ^{\lambda_1} \partial_v^{\b'}
			f\|_{\nu}^2 -C_{\eta}\|\chi_{\{|v|\leq 2C_{\eta}\}} f\| ^2.
		\end{align}
		$(2)$
		If $0\leq \gamma\leq1$,
		then for any $\eta \geq 0$,   there exists $C_{\eta} > 0$ such that
		\begin{align}
			\label{es-energy-linear:1-hard}
			\left\langle   \langle v \rangle ^{2\lambda_1}  \partial_v^\b( 		\L_Mf),   \partial_v^\b f  \right\rangle\geq  & \|  \langle v \rangle ^{\lambda_1}
			\partial_v^{\b} f\|_{\nu}^2
			-\eta\sum_{|\b'|\leq |\b|} \| \langle v \rangle ^{\lambda_1}  \partial_v^{\b'}
			f\|_{\nu}^2 -C_{\eta}\| f\|_{\nu}^2.
		\end{align}
	\end{lemma}
	The operators $\Gamma $ and $\Gamma_M $ are defined in \eqref{deln} and  \eqref{delm} respectively. We also can write $\Gamma=\Gamma_{+}-\Gamma_{-}$,  where
	\begin{align*}
		&\Gamma_{-}(  f_1,     f_2) = \int_{\mathbb{R}^3 \times \mathbb{S}^2} |u-v|^\gamma b_0(\alpha) \sqrt{\mu(u)} \,     f_1(u) \,     f_2(v) \,   \d u \,   \d \omega,  	\\		
		&	\Gamma_{+}(  f_1,     f_2) = \int_{\mathbb{R}^3 \times \mathbb{S}^2} |u-v|^\gamma b_0(\alpha) \sqrt{\mu(u)} \,     f_1(u') \,   f_2(v') \,   \d u \,   \d \omega.
	\end{align*}
	Then,    we have
	\begin{align*}
		\nabla_x \Gamma(f_1,   f_2)	=  \Gamma(\nabla_xf_1,   f_2)+ \Gamma(f_1,  \nabla_xf_2)+ \Gamma_x(f_1,   f_2),
	\end{align*}
	\begin{align*}
		\nabla_v \Gamma(f_1,   f_2)	=  \Gamma(\nabla_vf_1,   f_2)+ \Gamma(f_1,  \nabla_vf_2)+ \Gamma_v(f_1,   f_2),
	\end{align*}
	\begin{align*}
		\nabla_v \Gamma_M(f_1,   f_2)	=  \Gamma_M(\nabla_vf_1,   f_2)+ \Gamma_M(f_1,  \nabla_vf_2)+ \Gamma_{M,   v}(f_1,   f_2).
	\end{align*}
	Here,
	\begin{align*}
		\Gamma_x(f_1,   f_2)=&	\int_{\mathbb{R}^3 \times \mathbb{S}^2} |u-v|^\gamma b_0(\alpha)\nabla_x \sqrt{\mu(u)} \,    \{  f_1(u') \,   f_2(v')- f_1(u) \,     f_2(v)    \}\,   \d u \,   \d \omega,  \\
		\Gamma_v(f_1,   f_2)=&	\int_{\mathbb{R}^3 \times \mathbb{S}^2} |u-v|^\gamma b_0(\alpha)\nabla_u \sqrt{\mu(u)} \,    \{  f_1(u') \,   f_2(v')- f_1(u) \,     f_2(v)    \}\,   \d u \,   \d \omega,  \\
		\Gamma_{M,   v}(f_1,   f_2)=&	\int_{\mathbb{R}^3 \times \mathbb{S}^2} |u-v|^\gamma b_0(\alpha)\nabla_u \sqrt{\mathcal{M}(u)} \,    \{  f_1(u') \,   f_2(v')- f_1(u) \,     f_2(v)    \}\,   \d u \,   \d \omega.
	\end{align*}
	\begin{lemma}\label{gamma}
		Let $-3 <\gamma \leq 1$,    $f_i(i=1,   2,   3) $  be smooth functions,   and   $(i,   j) = (1,   2)$  or $(i,   j) = (2,   1)$.
		For any constants  $\eta>0$  denote $ {w_{ \eta}}(v)= e^{\eta|v|^2}$and  ${\overline{w} }(v)=\langle v \rangle^{\lambda_1}e^{\lambda_2|v|^2},     \lambda_1 \geq 0,    \lambda_2 \geq 0  $,
		it holds that
		\begin{align}\label{Gf}
			\langle \Gamma_X(f_1,   f_2),  {\overline{w} } ^2f_3\rangle  \lesssim \|  w_{\eta}{\overline{w} }  f_i \|_\infty \|{{\overline{w} }}f_j\|_\nu \|{\overline{w} } f_3\|_\nu,
		\end{align}	
		and
		\begin{align}\label{gw}
			\langle \Gamma_{X}( f_1,      f_2),   {\overline{w} } ^2f_3 \rangle \lesssim \|  w_{\eta}{\overline{w} }^2 f_i \|_\infty \|{\overline{w} }^2 f_j\|_\nu \|  f_3\|_\nu,
		\end{align}
		where $ \Gamma_{X}$	present $ \Gamma$,   $ \Gamma_{M}$,   $ \Gamma_{x}$,   $ \Gamma_{v}$ or $ \Gamma_{M,   v}$.

	\end{lemma}
	\begin{proof}
		With  $\int_{\mathbb{R}^3} |u-v|^\gamma \sqrt{\mu(u)} \,   \d u \lesssim \langle v \rangle^\gamma \lesssim \nu(v)$,  it holds that
		\begin{align*}
			&\langle \Gamma_{-}(  f_1,     f_2),   {\overline{w} } ^2f_3 \rangle
			\le \|w_{\eta} f_1   \|_\infty \|{\overline{w} } f_2\|_\nu \|{\overline{w} } f_3\|_\nu.	
		\end{align*}
		Noticing that $|u-v| = |u'-v'|$,   $\d u\,  \d v = \d u'\,  \d v'$,   combining  ${\overline{w} }^2 (v)\leq {\overline{w} }^2 (v'){\overline{w} }^2 (u')$ and $\sqrt{\mu(u')}\leq 1$,  we obtain
		\begin{align}\label{ga+}
			\langle \Gamma_{+}(  f_1,     f_2),   {\overline{w} } ^2f_3 \rangle
			\lesssim& \left(\int_{\mathbb{R}^3 \times \mathbb{R}^3} |u-v|^\gamma \sqrt{\mu(u)} \,   |    f^2_1(u')   f^2_2(v')  | {\overline{w} }^2 (u'){\overline{w} }^2 (v')\,   \d u \,   \d v\right)^{\frac{1}{2}}\|{\overline{w} } f_3\|_{\nu}\\\nonumber
			\lesssim&\|w_{\eta}  {\overline{w} }  f_1\|_{\infty}\left(\int_{\mathbb{R}^3 \times \mathbb{R}^3} |u-v|^\gamma e^{-2\eta|u|^2}  \,   |  f^2_2(v)  |  {\overline{w} }^2 (v)\,   \d u \,   \d v\right)^{\frac{1}{2}}
			\lesssim\|  w_{\eta}{\overline{w} } f_1\|_{\infty} \|{\overline{w} } f_2\|_\nu \|{\overline{w} } f_3\|_\nu.	
		\end{align}
		Based on symmetry,    we also have
		\begin{align*}
			&\langle \Gamma_{-}(  f_1,      f_2),   {\overline{w} } ^2f_3 \rangle\lesssim \|w_{\eta}  f_2 \|_\infty \|{\overline{w} } f_1\|_\nu \|{\overline{w} } f_3\|_\nu,  \\
			&\langle \Gamma_{+}( f_1,      f_2),   {\overline{w} } ^2f_3 \rangle \lesssim \|  w_{\eta}{\overline{w} } f_2 \|_\infty \|{\overline{w} } f_1\|_\nu \|{\overline{w} } f_3\|_\nu.	
		\end{align*}
		The above argument also can apply to $\mathcal{M}$. By \eqref{decay},
		$
		|\nabla_x \sqrt{\mu(u)}|\lesssim 	\mu(u)^{1/4}, \:  |\nabla_u \sqrt{\mu(u)}|\lesssim 	\mu(u)^{1/4},
		$
		which preserves the Gaussian decay.
		Thus,   \eqref{Gf} is proved.
		\par
		Using the same method as above and applying ${\overline{w} }^4(v)\leq {\overline{w} }^4 (v'){\overline{w} }^4 (u')$,   we can 	  prove   \eqref{gw}.
	\end{proof}
	\begin{lemma}\label{oa} Let ${\overline{w} }=\langle v \rangle^{\lambda_1}e^{\lambda_2|v|^2}, \,  0\leq \lambda_1, \,  0\leq \lambda_2 \ll 1	$,  it holds that
		\begin{align}\label{oa1}
			\|\overline{w}\overline{A}\|_{H^1_{x,   v}}\lesssim \sum_{\substack{2k+1\leq i+j \leq 4k+2  }} 	\varepsilon^{\,  i+j-2k} (1+t)^{\,  i+j-2}.
		\end{align}	
	\end{lemma}
	\begin{proof}
		Using mathematical induction,   we	  first prove
		\begin{align}	\label{fir4}
			\left\|{\overline{w} }\mathbb{P}^{\bot}\left( \frac{F_i}{\sqrt{\mu}}\right)\right\|_{H^1_{x,   v}}\lesssim	 (1+t)^{i-1}.
		\end{align}
		By \eqref{i-p},  we have
		\begin{align*}
			&\left\|{\overline{w} }\mathbb{P}^{\bot}\left( \frac{F_1}{\sqrt{\mu}}\right)\right\| _{L^2_{x,   v}}	\lesssim \left\|{\overline{w} }\frac{\left\{\partial
				_{t}+v\cdot \nabla _{x}\right\}\mu}{\sqrt{\mu}}\right\|_{L^2_{x,   v}} +\left\|{\overline{w} }\nabla_x \phi_0\frac{ \nabla_v \mu}{\sqrt{\mu}}\right\| _{L^2_{x,   v}} \\
			\lesssim&\|\nabla_x(\rho_0,   u_0)\|_{L^2_x}\|(1+|v|^3)\sqrt{\mu}{\overline{w} }\|_{L^2_v}+\|\nabla_x \phi_0\|_{L^2_x}\|\langle v \rangle \sqrt{\mu}{\overline{w} }\|_{L^2_v}
			\lesssim 1.
		\end{align*}
		Assume $	\left\|{\overline{w} }\mathbb{P}^{\bot}\left( \frac{F_i}{\sqrt{\mu}}\right)\right\| \lesssim	 (1+t)^{i-1}$ holds for $1\leq i \leq n$,  we prove it holds for $i=n+1$ as follows. Using the same method as \eqref{ga+},  combining  ${\overline{w} }(v')\langle v' \rangle^{2\gamma}\lesssim e^{\eta|v|^2}e^{\eta|u|^2}$,  \eqref{fir1} and \eqref{fir3},  we have
		\begin{align*}
			&	\left\| {\overline{w} }
			\Gamma_{+} \left(\frac{F_i}{\sqrt{\mu}},   \frac{F_j}{\sqrt{\mu}}\right)
			\right\|\\
			\lesssim &  \left(\iint
			\left(	\int_{\mathbb{R}^3 \times \mathbb{S}^2} |u-v|^\gamma b_0(\theta) \,   \sqrt{\mu(u')} \,
			\left(	\frac{F_i}{\sqrt{\mu}}(u)\right) ^2\,   	\left(	\frac{F_j}{\sqrt{\mu}}(v)\right) ^2 \,   \d u \d \omega \right) {\overline{w} }(v') \langle v' \rangle^{2\gamma} \d v \d x\right) ^{1/2}\\
			\lesssim& \left\| e^{\eta|u|^2}  	\frac{F_i}{\sqrt{\mu}}(u) \right\|_{\infty} \left(\left\|e^{ \eta|v|^2}\left(\mathbb{P}	\frac{F_j}{\sqrt{\mu}}(v)\right)\right\|_{L^2_{x,   v}}+\left\|e^{ \eta|v|^2}\mathbb{P}^{\bot}	\left(\frac{F_j}{\sqrt{\mu}}(v)\right)\right\|_{L^2_{x,   v}}\right)\\
			\lesssim&(1+t)^{i-1} \left(  \|(\rho_j,   u_j,  \theta_j)\|   +\left\|{\overline{w} }\mathbb{P}^{\bot}	\left(\frac{F_j}{\sqrt{\mu}}(v)\right)\right\| \right) \\
			\lesssim&(1+t)^{i+j-2}.
		\end{align*}
		Similarly,
		$
		\left\|{\overline{w} }(v) 	\Gamma_{-} \left(\frac{F_i}{\sqrt{\mu}},   \frac{F_j}{\sqrt{\mu}}\right)
		\right\|
		\lesssim(1+t)^{i+j-2},
		$
		thus
		\begin{align}\label{gammaw}
			\left\|{\overline{w} }	\Gamma\left(\frac{F_{i}}{\sqrt{\mu}},  \frac{F_{j}}{\sqrt{\mu}} \right)\right\|_{L^2_{x,   v}}	 	\lesssim&(1+t)^{i+j-2}.
		\end{align}
		Then,  by \eqref{decay},  \eqref{fir2},  \eqref{fir3} and  \eqref{gammaw}, we obtain
		\begin{align*}
			&\left\|{\overline{w} }\mathbb{P}^{\bot}\left( \frac{F_{n+1}}{\sqrt{\mu}}\right)\right\| 	\lesssim \left\|{\overline{w} }\frac{\left\{\partial
				_{t}+v\cdot \nabla _{x}\right\}F_n}{\sqrt{\mu}}\right\| +\sum_{\substack{ i+j=n \\ i,   j\geq 0}}\left\|{\overline{w} }\nabla_x \phi_i\frac{ \nabla_v F_j}{\sqrt{\mu}}\right\| +\sum_{\substack{ i+j=n+1 \\ i,   j\geq 1}} \left\|{\overline{w} } \Gamma\left( \frac{F_i}{\sqrt{\mu}},  \frac{F_j}{\sqrt{\mu}}\right)\right\|  \\
			\lesssim&\|\nabla_x(\rho_0,   u_0)\|_{L^2_x}(1+t)^{n-1}\|{\overline{w} }(1+|v|^{3(n+1)})\sqrt{\mu}\|_{L^2_v}+ \|\nabla_x \phi_i\|_{L^2_x}\|{\overline{w} }(1+|v|^{3(j+1)})\sqrt{\mu}\|_{L^2_v}+(1+t)^{i+j-2}  \\
			\lesssim& (1+t)^{n}
			.
		\end{align*}
		Using the similar argument as 	$\left\|{\overline{w} }\mathbb{P}^{\bot}\left( \frac{F_i}{\sqrt{\mu}}\right)\right\| $ to 	$\left\|\nabla\left({\overline{w} }\mathbb{P}^{\bot}\left( \frac{F_i}{\sqrt{\mu}}\right)\right)\right\| $,  we can prove \eqref{fir4}.
		As the definition of $\overline{A}=A/\sqrt{\mu}$ in \eqref{A},  by \eqref{fir2}-\eqref{fir3},  using the same method as the proof of \eqref{gammaw} to $\Gamma\left(\frac{F_{i}}{\sqrt{\mu}},  \frac{F_{j}}{\sqrt{\mu}} \right)$,   we have
		\begin{align*}
			\| {\overline{w} }\,  \overline{A} \|
			&\lesssim
			\sum_{\substack{ i+j \ge 2k+1 \\ 2 \le i,   j \le 2k-1}}
			\varepsilon^{\,  i+j-2k}
			\Big\|
			{\overline{w} }\,
			\Gamma\!\left(
			\frac{F_i}{\sqrt{\mu}},
			\frac{F_j}{\sqrt{\mu}}
			\right)
			\Big\| _{L^2_{x,   v}}
			+\sum_{\substack{ i+j\geq 2k \\ 1\leq i,   j\leq 2k-1}}%
			\varepsilon ^{i+j-2k+1}
			\| \nabla_x \phi_i \|_{L^2_x}\,
			\|{\overline{w} }\nabla _{v}F_{j}/\sqrt{\mu}\|_{L^2_v} \\
			&\lesssim 	\sum_{\substack{2k+1\leq i+j \leq 4k+2  }} 	\varepsilon^{\,  i+j-2k} (1+t)^{\,  i+j-2}.
		\end{align*}
		Using the same argument as $	\| {\overline{w} }\,  \overline{A} \|$ to  $\| {\overline{w} }\,  \nabla\overline{A} \|$,  we complete  the  proof of \eqref{oa}.
	\end{proof}

	\begin{lemma}\label{le:jac}
		Let $\partial=\partial_x$ or $\partial_v$,  suppose $ T_0 $ be a small  fixed number and  $0\le s \le t\leq   T_0$,  then it holds that
		\begin{equation}
			\frac{1}{2} |t-s|^3 \le \left| \det \left( \frac{\partial X(s;t,   x,   v)}{\partial v} \right) \right| \le 2 |t-s|^3. \label{tezhengxianjielun2}
		\end{equation}
	\end{lemma}
	\begin{proof}
		By  \eqref{char},   we have
		\begin{align}
			\frac{\d^2}{\d s^2}\,  \partial  X(s;t,   x,   v)
			=&-	\nabla^2_x\phi^{\varepsilon}\!\big(s,   X(s;t,   x,   v)\big)
			\cdot \partial X(s;t,   x,   v),
			\label{tezhengxian1}
			\\
			\frac{\d}{\d s}\,  \partial V(s;t,   x,   v)
			= &-	\nabla^2_x\phi^{\varepsilon}\!\big(s,   X(s;t,   x,   v)\big)
			\cdot \partial X(s;t,   x,   v).
			\label{tezhengxian2}
		\end{align}
		Note that $\dfrac{\d }{\d s}\pt X(t;t,   x,   v) =0$. Integrating (\ref{tezhengxian1}) and (\ref{tezhengxian2}) along the characteristics respectively,   we get
		\begin{align}
			\partial  X(s;t,   x,   v)
			&\lesssim \partial  X(t;t,   x,   v)
			+  |t-s|^{2}
			\sup_{0 \le s \le t}
			\big\{ \| \nabla^2_x\phi^{\varepsilon}(s) \|_{{\infty}} \big\}
			\sup_{0 \le s \le t}
			\big\{ |\partial_t X(s;t,   x,   v)| \big\},
			\label{tezhengxian3}
			\\
			\partial  V(s;t,   x,   v)
			&\lesssim \partial  V(t;t,   x,   v)
			+   |t-s|
			\sup_{0 \le s \le t}
			\big\{ \| \nabla^2_x\phi^{\varepsilon}(s) \|_{ \infty} \big\}
			\sup_{0 \le s \le t}
			\big\{ |\partial_t X(s;t,   x,   v)| \big\}.
			\label{tezhengxian4}
		\end{align}
		Combining $\|\nabla^2_x \phi ^\varepsilon\|_{ \infty}\lesssim 1$ form \eqref{phi},  (\ref{tezhengxian3}) and (\ref{tezhengxian4}),  we have
		\begin{equation}
			\sup_{0 \le s \le t} \left\{ \| \partial_x X(s;t,   x,   v) \|_\infty + \| \partial_vV(s;t,   x,   v) \|_\infty \right\} \lesssim 1. \label{tezhengxianjielun1}
		\end{equation}
		To prove \eqref{tezhengxianjielun2},   we consider the Taylor expansion of
		$\partial_v X(s;t,   x,   v)$ in $s$ around $t$:
		\begin{align}
			\begin{split}
				\partial_v X(s;t,   x,   v)
				&= \partial_v X(t;t,   x,   v)
				+ (s-t)\left.\frac{\d}{\d s}\partial_v X(s;t,   x,   v)\right|_{s=t}   + \frac{(s-t)^2}{2}
				\frac{\d^2}{\d s^2}\partial_v X(\tilde{s};t,   x,   v),
			\end{split}
			\label{tezhengxian5}
		\end{align}
		where $\tilde{s}$ lies between $s$ and $t$.
		Noticing
		$
		\partial_v X(t;t,   x,   v)=0,
		\left| \frac{\d}{\d s}\partial_v X(s;t,   x,   v)\right|_{s=t}
		=  I_{3\times3}
		$,  combining \eqref{tezhengxian1},  $\|\nabla^2_x \phi ^\varepsilon\|_{ \infty}\lesssim 1$ form \eqref{phi} and \eqref{tezhengxianjielun1},   for $T_0$ is small,   we obtain
		\begin{align}
			\begin{split}
				\left|	 \frac{s-t}{2}\frac{\d^2}{\d s^2}\partial_v X(\tilde{s} )\right|
				&\le \left|	 \frac{s-t}{2}\right|
				\sup_{0\le s\le t}
				\big\{ \|\nabla^2_{x }\phi^{\varepsilon}(\tilde{s})\|_{ {\infty}}
				\|\partial   X(\tilde{s} )\|_{ {\infty}} \big\}\leq C\left|	 \frac{s-t}{2}\right|\leq \frac{CT_0}{2}\leq \frac{1}{8},
			\end{split}
			\label{tezhengxian6}
		\end{align}
		which implies (\ref{tezhengxianjielun2}).
		
	\end{proof}
	The operator $\mathcal{K}_w$ is defined in \eqref{kmnwh}. The following results describe  its key properties.
	\begin{lemma}[\cite{lisiam2023,   LiYangZhong2021}]\label{eq:es:K}
		Let $\nabla{\mathcal{K}_{w}}:=|  \nabla_x{\mathcal{K}_{w}}|+|  \nabla_v{\mathcal{K}_{w}}|$,  there exists a sufficiently small constant $\eta>0$, it holds that
		
		\smallskip
		\noindent
		(1) \textbf{Soft potentials.}
		If $-3<\gamma<0$,   then the kernel $l^{\chi}_{ w}(v,   u)$ of  $ K ^{\chi}_{M,   w}$ and  the kernel  $\tilde{l}^{\chi}_{ w}(v,   u)$ of  $ \nabla K^{\chi}_{M,   w}$ satisfy
		\begin{align}\label{eq:es:Kc:1}
			\nu^{-1}	\int_{\mathbb{R}^3} \big| l^{\chi}_{ w}(v,   u) e^{ {\eta} |v-u|^2} \big|\,  \d u+ \nu^{-1}	\int_{\mathbb{R}^3} \big| \tilde{l}^{\chi}_{ w}(v,   u) e^{ {\eta} |v-u|^2} \big|\,  \d u
			\;\lesssim\;  \langle v\rangle^{-2}.
		\end{align}
		Moreover,
		\begin{align}\label{eq:es:Kc:2}
			\big| \nu^{-1} \mathcal{K}_{w}^{1-{\chi}}g \big|+	\big| \nu^{-1}(\nabla_v   \mathcal{K}_{w}^{1-{\chi}})g \big|
			\;\lesssim\; \e^{3+\gamma}e^{-C|v|^2}\|g\|_{\infty}.
		\end{align}
		\smallskip
		\noindent
		(2) \textbf{Hard potentials.}
		If  $0\le \gamma \le 1$,   then the kernel $l_{w}(v,   u)$ of $\, {\mathcal{K}_{w}}$ and $ \tilde{l}_{ w}(v,   u)$ of  $\, \nabla{\mathcal{K}_{w}}$ satisfy
		\begin{align} 	\label{lhard}
			\nu^{-1}	\int_{\mathbb{R}^3} \big| l_{w}(v,   u) e^{ {\eta} |v-u|^2} \big|\,  \d u+ \nu^{-1}	\int_{\mathbb{R}^3} \big|  \tilde{l}_{ w}(v,   u) e^{ {\eta} |v-u|^2} \big|\,  \d u
			 \lesssim  \langle v\rangle^{-2}.
		\end{align}
		
	\end{lemma}
	Next,   we turn to deducing the $L^{\infty}_{x,   v}$-estimate on the terms related to the nonlinear collision operator $Q$.
	\begin{lemma}\cite[Lemma 3.3]{WuZhouLi2023Diffusive}
		Let $3< \gamma\leq1$,  it  holds that
		\begin{align}\label{wt}
			\left\|	\nu^{-1}\frac{w }{\sqrt{\mathcal{M}}}Q \left(
			\frac{\sqrt{\mathcal{M}}}{w }f,
			\frac{\sqrt{\mathcal{M}}}{w }g\right)\right\|_{ {\infty} }\lesssim  \|  f \|_{ {\infty} }
			\| g\|_{ {\infty} },
		\end{align}
		and
		\begin{align}\label{wtv}
			\left\|	\nu^{-1}  \nabla_v\left(\frac{w }{\sqrt{\mathcal{M}}}Q \Big(
			\frac{\sqrt{\mathcal{M}}}{w }f,
			\frac{\sqrt{\mathcal{M}}}{w }g\Big)\right)\right\|_{ {\infty} } \lesssim  \|  \nabla_vf \|_{ {\infty} }
			\| g\|_{ {\infty} }+ \|  f \|_{ {\infty} }
			\| \nabla_vg\|_{ {\infty} }+ \|  f \|_{ {\infty} }
			\| g\|_{ {\infty} }.
		\end{align}
	\end{lemma}
	\medskip
	
	\medskip
	\noindent\textbf{Acknowledgments.}  This research is supported by NSFC (No.12271179),
	Basic and Applied Basic Research Project of Guangdong (No.2022A1515012097),
	Basic and Applied Basic Research Project of Guangzhou (No.2023A04J1326).
	
	\medskip
	
	\noindent\textbf{Conflict of Interest Statement}
	
	The authors declare that they have no conflicts of interest.
	\medskip
	
	\noindent\textbf{Data Availability Statement}
	
	No data was used for the research described in the article.

\end{document}